\documentclass[reqno]{amsart}
\usepackage{bm,amsmath,amsthm,amssymb,mathtools,verbatim,amsfonts,tikz-cd,mathrsfs}
\usepackage{enumitem}
\usepackage{float}

\usepackage[hidelinks,colorlinks=true,linkcolor=blue, citecolor=black,linktocpage=true]{hyperref}
\usepackage{graphicx}
\usepackage[alphabetic]{amsrefs}
\usepackage{caption}
\usepackage{morefloats}
\usepackage[top=1.3in, bottom=1.1in, left=1.3in, right=1.3in,marginpar=1in]{geometry}
\usepackage{subfigure}

\usetikzlibrary{matrix,arrows,decorations.pathmorphing}

\usepackage{chngcntr}
\counterwithin{figure}{section}

\newtheorem{thm}{Theorem}[section]
\newtheorem{prop}[thm]{Proposition}

\newtheorem{cor}[thm]{Corollary}
\newtheorem{lem}[thm]{Lemma}

\theoremstyle{definition}
\newtheorem{define}[thm]{Definition}

\theoremstyle{remark}
\newtheorem{rem}[thm]{Remark}

\newcommand{\ve}[1]{\boldsymbol{\mathbf{#1}}}

\newcommand{\ul}[1]{\underline{#1}}

\newcommand{\Z}{\mathbb{Z}}
\newcommand{\N}{\mathbb{N}}
\newcommand{\Q}{\mathbb{Q}}

\renewcommand{\d}{\partial}
\renewcommand{\subset}{\subseteq}

\renewcommand{\tilde}{\widetilde}
\renewcommand{\bar}{\overline}

\newcommand{\iso}{\cong}

\DeclareMathOperator{\can}{{can}}

\DeclareMathOperator{\gr}{{gr}}

\DeclareMathOperator{\id}{{id}}

\DeclareMathOperator{\Mor}{{Mor}}

\DeclareMathOperator{\rank}{{rank}}

\DeclareMathOperator{\Spin}{{Spin}}

\newcommand{\bA}{\mathbb{A}}
\newcommand{\bB}{\mathbb{B}}

\newcommand{\bE}{\mathbb{E}}
\newcommand{\bF}{\mathbb{F}}

\newcommand{\bT}{\mathbb{T}}
\newcommand{\bU}{\mathbb{U}}
\newcommand{\bX}{\mathbb{X}}

\newcommand{\bXI}{\mathbb{XI}}

\newcommand{\cA}{\mathcal{A}}
\newcommand{\cB}{\mathcal{B}}
\newcommand{\cC}{\mathcal{C}}

\newcommand{\cF}{\mathcal{F}}

\newcommand{\cH}{\mathcal{H}}

\newcommand{\cK}{\mathcal{K}}

\newcommand{\cP}{\mathcal{P}}

\newcommand{\frF}{\mathfrak{F}}
\newcommand{\frG}{\mathfrak{G}}

\newcommand{\frs}{\mathfrak{s}}

\newcommand{\fru}{\mathfrak{u}}

\newcommand{\frw}{\mathfrak{w}}
\newcommand{\frx}{\mathfrak{x}}
\newcommand{\fry}{\mathfrak{y}}

\newcommand{\scA}{\mathscr{A}}

\newcommand{\scC}{\mathscr{C}}

\newcommand{\scF}{\mathscr{F}}
\newcommand{\scG}{\mathscr{G}}

\newcommand{\cCFL}{\mathcal{C\!F\!L}}
\newcommand{\cCFK}{\mathcal{C\hspace{-.5mm}F\hspace{-.3mm}K}}

\newcommand{\CF}{\mathit{CF}}

\newcommand{\HF}{\mathit{HF}}

\newcommand{\HFK}{\mathit{HFK}}

\newcommand{\CFI}{\mathit{CFI}}

\newcommand{\XI}{\mathbb{XI}}

\newcommand{\PD}{\mathit{PD}}

\newcommand{\xs}{\ve{x}}
\newcommand{\ys}{\ve{y}}

\newcommand{\as}{\ve{\alpha}}
\newcommand{\bs}{\ve{\beta}}

\renewcommand{\a}{\alpha}
\renewcommand{\b}{\beta}

\newcommand{\veps}{\varepsilon}

\usepackage{leftidx}

\DeclareMathOperator{\Cone}{{Cone}}

\numberwithin{equation}{section}

\newcommand{\ar}{\mathrm{a.r.}}

\newcommand\co{\colon}
\newcommand{\lab}[1]{$\scriptstyle #1$}

\newcommand{\cen}{\mathrm{cen}}

\newcommand{\vopp}{{\tilde{v}}}

\newcommand{\scU}{\mathscr{U}}
\newcommand{\scV}{\mathscr{V}}

\allowdisplaybreaks

\newcommand\CZhat{\widehat{\mathcal{C}}_\Z}

\title{An involutive dual knot surgery formula}
\author{Kristen Hendricks}
\thanks{KH was partially supported by NSF grant DMS-2019396 and a Sloan Research Fellowship.}
\address{Department of Mathematics, Rutgers University, New Brunswick, NJ, USA}
\email{kristen.hendricks@rutgers.edu}
\author{Jennifer Hom}
\thanks{JH was partially supported by NSF grants DMS-1552285 and DMS-2104144.}
\address{School of Mathematics, Georgia Institute of Technology, Atlanta, GA, USA}
\email{hom@math.gatech.edu}
\author{Matthew Stoffregen}
\address{Department of Mathematics, Michigan State University, East Lansing, MI, USA}
\email{stoffre1@msu.edu}
\thanks{MS was partially supported by NSF grant DMS-1702532.}
\author{Ian Zemke}
\address{Department of Mathematics\\Princeton University\\  Princeton, NJ, USA}
\email{izemke@math.princeton.edu}
\thanks{IZ was partially supported by NSF grant DMS-1703685.}

\begin{document}
\maketitle

\begin{abstract}
We prove an involutive analog of the dual knot surgery formula of Eftekhary and Hedden-Levine. We also compute a small model for the local equivalence class of the involutive dual knot complex.
\end{abstract}

\tableofcontents

\section{Introduction}

In the early 2000s, Ozsv\'{a}th and Szab\'{o} defined a collection of invariants for 3-manifolds called \emph{Heegaard Floer homology} \cite{OSDisks, OSProperties}. Their construction associates to a 3-manifold $Y$ equipped with a $\Spin^c$ structure $\frs\in \Spin^c(Y)$ a finitely-generated $\bF[U]$ chain complex $\CF^-(Y, \frs)$, whose homology is one version $\HF^-(Y,\frs)$ of the Heegaard Floer homology of $(Y,\frs)$. The Heegaard Floer homology of the 3-manifold is then $\bigoplus_{\frs \in \Spin^c(Y)} \HF^-(Y,\frs)$. Shortly thereafter, Ozsv\'{a}th and Szab\'{o} \cite{OSKnots} and independently Rasmussen \cite{RasmussenKnots} introduced a refinement of the theory, called \emph{knot Floer homology}, for a pair $(Y,K)$ consisting of a knot $K$ inside of a 3-manifold $Y$. In its modern formulation, knot Floer homology associates to $(Y,K)$ a finitely-generated, free chain complex $\cCFK(Y,K)$ over the ring $\bF[\scU,\scV]$.

By the standards of invariants of 3-manifolds arising from Floer and gauge theory, Heegaard Floer homology has proved unusually friendly to computations. One of the several reasons for this is the existence of a \emph{surgery formula} for the theory, as follows. Given $K\subset S^3$ a knot and $n\in \Z$, let $S^3_n(K)$ be the manifold obtained by Dehn surgery with coefficient $n$ on $K$. Then Ozsv\'{a}th and Szab\'{o} construct from $\cCFK(S^3,K)$ and the coefficient $n$ a chain complex $\bX_n(K)$ which is chain homotopy equivalent to $\CF^-(S^3_n(K))$, and which splits along $\Spin^c$ structures into complexes chain homotopy equivalent to $\CF^-(S^3_n(K), \frs)$. In particular, this means that the Heegaard Floer homology $\HF^-(S^3_n(K))$ is determined by the knot Floer complex $\cCFK(S^3,K)$ of $K$ and the coefficient $n$ \cite{OSIntegerSurgeries}.

Recall that given a knot $K \subset S^3$, there is a \emph{dual knot} $\mu\subset S^3_n(K)$, consisting of the core of the solid torus $D^2 \times S^1$ which is glued into $S^3 \setminus \nu(K)$ to produce the surgery $S^3_n(K)$. In 2019, Hedden and Levine \cite{HeddenLevineSurgery} described a refinement of Ozsv\'{a}th and Szab\'{o}'s surgery formula which computes the knot Floer homology $\cCFK(S_n^3(K),\mu)$ from $\cCFK(S^3,K)$ and the coefficient $n$. The work of Hedden and Levine follows on work of Eftekhary \cite{EftekharyDuals}, who described the formula for a simpler version $\widehat{\HFK}(S^3_n(K),\mu)$ of the knot Floer homology of the dual knot, and made progress on the formula for the full knot Floer complex $\cCFK(S_n^3(K),\mu)$.

In 2015, Hendricks and Manolescu introduced a variant theory of Heegaard Floer homology, called \emph{involutive Heegaard Floer homology}. This theory again has versions for both 3-manifolds and knots. In the 3-manifold case, given $Y$ equipped with a conjugation-invariant $\Spin^c$ structure $\frs$, there is a grading-preserving homotopy involution $\iota$ from $\CF^-(Y,\frs)$ to itself, called the conjugation symmetry. Involutive Heegaard Floer homology associates to $Y$ either the pair $(\CF^-(Y,\frs),\iota)$, called an \emph{$\iota$-complex}, or equivalently a chain complex $\CFI^-(Y, \frs)$ over $\bF[U,Q]/(Q^2)$ constructed from this data. The knot formulation considers, for a knot $K$ in $Y$, a homotopy automorphism
\[\iota_K : \cCFK(Y,K) \rightarrow \cCFK(Y, K).\]
\noindent This data is packaged as an \emph{$\iota_K$-complex} $(\cCFK(Y,K), \iota_K)$. (For more on the algebraic formalisms of $\iota$-complexes and $\iota_K$-complexes, see Section~\ref{sec:iota-complexes}.)

In earlier work \cite{HHSZ}, we proved an involutive analog of Ozsv\'{a}th and Szab\'{o}'s original knot surgery formula. More precisely, we constructed a chain complex $\XI_n(S^3,K)$ over $\bF[U,Q]/Q^2$, which we proved was homotopy equivalent to $\CFI^-(S^3_n(K))$. The complex $\XI_n(S^3,K)$ was completely determined by the $\iota_K$-complex $(\cCFK(S^3,K),\iota_K)$ and the coefficient $n$ (\cite[Theorem 1.5, Section 3.5, Section 21]{HHSZ}).

In the present paper, we build on the work of \cite{HHSZ} to prove an involutive analog of Hedden and Levine's dual knot surgery formula.  In particular, we construct an $\iota_K$-complex $\XI_n^\mu(K)$, which contains equivalent data to the Hedden-Levine surgery formula equipped with a homotopy automorphism. Our main result is as follows.

\begin{thm}\label{thm:main}
 If $K\subset S^3$ is a knot and $n\neq 0$, then the $\iota_K$-complex $\XI_n^\mu(K)$ is uniquely determined by, and is computable from, the $\iota_K$-complex $(\cCFK(S^3,K),\iota_K)$ and the coefficient $n$. Furthermore,
 \[
 \bXI^\mu_n(K)\simeq (\cCFK(S^3_n(K)),\iota_\mu).
 \]
\end{thm}

We note that for knots $K$ in 3-manifolds $Y\neq S^3$ (or, more specifically, if $Y$ is not an L-space), the techniques of the present paper are not sufficient to construct a similar $\iota_K$-complex $\XI^\mu_n(Y,K)$. In general one expects the $\iota_K$-complex $\XI_n^\mu(K)$ to require more holomorphic curve counting maps than are used to construct the $\iota_K$-complex $(\cCFK(Y,K),\iota_K)$.

Our proof of Theorem~\ref{thm:main} differs from the proof of Hedden and Levine in the non-involutive case \cite{HeddenLevineSurgery} in several ways. The proof given by Hedden and Levine involved adding basepoints to a Heegaard diagram for $S_n^3(K)$ and understanding how the induced knot filtration interacted with Ozsv\'{a}th and Szab\'{o}'s original proof of the mapping cone formula \cite{OSIntegerSurgeries}. Our proof follows a different line of reasoning, mentioned as an alternate strategy by Hedden and Levine \cite[Section 1.2]{HeddenLevineSurgery}, which unlike their strategy works only for surgeries on knots in $L$-spaces. We view the dual knot as the result of taking the connected sum of $K$ with a Hopf link $H$, and then performing surgery on the component of $K\# H$ corresponding to $K$. The remaining component is the dual knot $\mu$. We then apply a refinement of the Manolescu-Ozsv\'{a}th link surgery formula \cite{MOIntegerSurgery} to the link $K\# H$ to compute the knot Floer homology after performing surgery on a single link component of $K\# H$. The resulting model of $\cCFK(S_n^3(K),\mu)$ is infinitely-generated over $\bF[\scU,\scV]$. By using homological perturbation theory, we reduce the complex to a finitely-generated complex over $\bF[\scU,\scV]$ and recover the Hedden-Levine formula for surgeries on knots in $S^3$, or more generally knots in $L$-spaces.

\subsection{Local equivalence classes}

The knot surgery formula of Ozsv\'{a}th and Szab\'{o} takes the form of an infinitely-generated chain complex $\bX_n(K)$ over $\bF[U]$, which is homotopy equivalent to the finitely-generated free chain complex $\CF^-(S^3_n(K))$ over $\bF[U]$. Ozsv\'{a}th and Szab\'{o} describe a natural procedure for constructing finitely-generated models of $\bX_n(K)$, called \emph{truncation}. Nonetheless, the finitely-generated truncations of $\bX_n(K)$ usually require many generators. For example, $\bX_{+1}(K)$  \emph{a priori} requires at least $(4g_3(K)-3)\cdot \rank_{\bF[\scU,\scV]} \cCFK(K)$ free-generators. 

There is an additional notion of equivalence between two $\iota$-complexes or between two $\iota_K$-complexes weaker than equivariant chain homotopy equivalence, known as \emph{local equivalence}, which we describe in Section~\ref{sec:iota-complexes}. Relevantly, for applications to the homology cobordism and knot concordance groups, one is usually content to work up to local equivalence.

In our previous work on the involutive surgery formula \cite{HHSZ}, we proved that the 3-manifold local equivalence class of $\XI_n(K)$ coincided with the local equivalence class of its subcomplex $(A_0(K),\iota_K)$, which has rank equal to exactly $\rank_{\bF[\scU,\scV]} \cCFK(K)$ generators over $\bF[U]$. This result was critical to our subsequent applications to the homology cobordism group in \cite{HHSZSeifert}, since it substantially reduces the complexity of computations. In our present work, we prove an analog for the involutive dual knot formula:

\begin{thm} \label{thm:local-classes-intro}
Suppose that $K$ is a knot in $S^3$ and $n\neq 0$. The local equivalence class of $(\cCFK(S^3_{n}(K),\mu),\iota_\mu)$ is determined by the local equivalence class of $(\cCFK(S^3,K),\iota_K)$. Furthermore, the local class of $(\cCFK(S^3_{+1}(K),\mu),\iota_\mu)$ admits a model with $3\cdot \rank_{\bF[\scU,\scV]} \cCFK(S^3,K)$ generators, and $(\cCFK(S^3_{-1}(K),\mu),\iota_\mu)$ admits a model with $5\cdot \rank_{\bF[\scU,\scV]} \cCFK(S^3,K)$ generators.
\end{thm}

In particular, the small models for $\pm 1$ surgeries are essentially the smallest non-trivial truncations of $\XI_{\pm 1}^\mu(K)$. See Theorem~\ref{thm:local-class} for precise statements.

The above result is also novel in the non-involutive setting, and gives a small model of the non-involutive local equivalence class of the dual knot complex of Hedden and Levine. We note that the local class of $\cCFK(S_{\pm 1}^3(K),\mu)$ has been successfully applied to understand questions related to the concordance group; see for example \cite{DHSThomcobknot, HLL, ZhouHomologyConcordance}.

Theorem \ref{thm:local-classes-intro} has implications for the homology concordance group $\CZhat$, generated by manifold-knot pairs $(Y,K)$, where $Y$ is an integer homology sphere bounding an acyclic smooth 4-manifold and $K$ is a knot in $Y$. Two pairs $(Y_0, K_0)$ and $(Y_1, K_1)$ are equivalent in $\CZhat$ if there is a smooth homology cobordism from $Y_0$ to $Y_1$ in which $-K_0 \sqcup K_1$ bounds a smoothly embedded annulus.

Recall from \cite[Theorem 1.1]{DHSThomcobknot} that there are homomorphisms
\[ \varphi_{i,j} \colon \CZhat \to \Z \]
for $(i,j) \in (\Z \times \Z^{ \geq 0}) - (\Z^{\leq 0} \times \{ 0\})$. When $j \neq 0$, these homomorphisms vanish for knots that are homology concordant to knots in $S^3$. Zhou \cite{ZhouHomologyConcordance} used \cite{HeddenLevineSurgery} to compute examples of pairs $(Y, K)$ with $\varphi_{n, n-1}(Y,K) \neq 0$ for each positive integer $n$. Furthermore, note that the $\varphi_{i,j}$ are defined for any knot $K$ in an integer homology sphere $Y$ (i.e., $Y$ need not bound a smooth acyclic 4-manifold).

\begin{cor}\label{cor:phiij}
Let $K$ be a knot in $S^3$ and let $\mu$ denote the core of surgery in $S^3_1(K)$. Then $\varphi_{i,j}(S^3_1(K), \mu) = 0$ if $|i-j|> 2$.
\end{cor}

\subsection{Organization} This paper is organized as follows. In Section~\ref{sec:background} we discuss some necessary background material. Specifically, in Section~\ref{sec:hfi-background} we briefly review involutive Heegaard Floer homology and Section~\ref{sec:iota-complexes} we recall the algebraic framework of $\iota$-complexes and $\iota_K$-complexes. The remainder of the section is devoted to setting up the algebra necessary for the paper, culminating in the proof of a homological perturbation lemma for hypercubes of chain complexes in Section~\ref{sec:homological-perturbation-hypercubes}. In Section \ref{sec:large-surgery-expanded} we set up the framework for our dual knot surgery formula and build an ``expanded'' model of the non-involutive dual knot surgery formula; this initial formula is a mapping cone between complexes which are infinitely generated over $\bF[\scU,\scV]$. In Section \ref{sec:small-model} we use the homological perturbation lemma for hypercubes to reduce this chain complex to a mapping cone between finitely-generated complexes, building a more computable small model for the mapping cone formula and recovering the results of \cite{HeddenLevineSurgery}. In Section~\ref{sec:involution} we compute the involution on the expanded model of the surgery formula and transfer it to the small model, completing the proof of Theorem~\ref{thm:main}. In Section~\ref{sec:formulas} we give explicit formulas for all of the maps appearing in the construction of the small model, and in Section~\ref{sec:example} we use these formulas to compute the example of the $\iota_K$-complex of $(Y, \mu)$, where $Y=S^3_{1}(4_1)$ and $\mu$ is the image of a meridian of $4_1$. Finally, in Section~\ref{sec:local-models} we compute the local equivalence class of the $\iota_K$ complex associated to a dual knot, and prove Theorem~\ref{thm:local-classes-intro}.

\subsection{Acknowledgments} The authors are grateful to Matt Hedden and Adam Levine for helpful conversations and the inspiration of \cite{HeddenLevineSurgery}, and also grateful to the anonymous referee for thoughtful comments and corrections.

\section{Background} \label{sec:background}

\subsection{Involutive Heegaard Floer theory} \label{sec:hfi-background}

We now recall the basics of involutive Heegaard Floer homology, due to the first author and Manolescu \cite{HMInvolutive}. Suppose that $\cH=(\Sigma,\as,\bs,w)$ is a weakly admissible Heegaard diagram for $Y$. Suppose that $\frs$ is a self-conjugate $\Spin^c$ structure. Write $\bar \cH=(\bar \Sigma, \bar \bs, \bar \as, w)$ for the diagram obtained from $\cH$ by reversing the orientation of $\Sigma$ and reversing the roles of $\as$ and $\bs$. (The curves $\bar \as$ are the images of $\as$ on $\bar \Sigma$, and similarly for $\bar \bs$.) There is a canonical chain isomorphism
\[
\eta\colon \CF^-(\bar{\cH},\frs)\to \CF^-(\cH,\frs).
\]
The first author and Manolescu consider the map
\[
\iota \colon \CF^-(\cH,\frs)\to \CF^-(\cH,\frs)
\]
given by the formula
\[
\iota := \eta \circ \Psi_{\cH\to \bar {\cH}},
\]
where $\Psi_{\cH\to \bar \cH}$ is the map from naturality as in \cite{JTNaturality}. They define the involutive Heegaard Floer complex $\CFI(Y,\frs)$ to be the mapping cone
\[
\Cone(Q(\id+\iota)\colon \CF^-(Y,\frs)\to Q\cdot\CF^-(Y,\frs)).
\]
Here $Q$ is a formal variable, which allows us to view $\CFI(Y,\frs)$ naturally as a module over $\bF[U,Q]/Q^2$.

The first author and Manolescu define a natural refinement for knots \cite{HMInvolutive}. If $K\subset Y$ is a knot and $\cH=(\Sigma,\as,\bs,w,z)$ is a doubly pointed Heegaard knot diagram, then we may consider the diagram $\bar \cH=(\bar \Sigma, \bar\bs, \bar\as, z, w)$. There is a canonical chain isomorphism
\[
\eta_K\colon \cCFK(\bar\cH)\to \cCFK( \cH)
\]
which sends $\scU^i\scV^j \xs$ to $\scU^{j} \scV^i \xs$, where we identify $\bT_{\a}\cap \bT_{\b}$ and $\bT_{\bar \b} \cap \bT_{\bar \a}$.  There is also a naturality map $\Psi_{\cH\to \bar \cH}$ which is the composition of a diffeomorphism map for a positive half-twist along $K$, together with the naturality map from \cite{JTNaturality}. The knot involution $\iota_K$ is defined as the composition
\[
\iota_K:=\eta\circ \Psi_{\cH\to \bar \cH}.
\]

One may similarly define a set of maps $\iota_L$, one for each component, for a link $L$ with a fixed choice of orientation for each component. Where it occasions no confusion we also refer to the sum of these maps as $\iota_L$.

\begin{rem}
Note that knot Floer homology naturally decomposes over $\Spin^c$ structures. The interaction with the involution is slightly subtle. Indeed, $\iota_K$ maps $\cCFK(Y,K,\frs)$ to $\cCFK(Y,K,\bar \frs+\PD[K])$. In our present paper, the dual knot $\mu$ is a generator of $H_1(S^3_n(K))$.
\end{rem}

\subsection{$\iota$-complexes, $\iota_K$-complexes, and knot-like complexes}
\label{sec:iota-complexes}

In this section, we recall the algebraic notions of $\iota$-complexes and $\iota_K$-complexes from \cite{HMInvolutive} and \cite{HMZConnectedSum}, along with the noninvolutive notion of a knot-like complex of \cite{DHSThomcobknot}.

\begin{define}
An \emph{$\iota$-complex} is a pair $(C,\iota)$ such that the following are satisfied:
\begin{enumerate}
\item $C$ is a chain complex which is homotopy equivalent to a free, finitely-generated chain complex over $\bF[U]$.
\item $C$ is equipped with an absolute $\Q$-valued grading such that $\d$ is $-1$ graded and $U$ is $-2$ graded.
\item $\iota$ is a grading preserving, $\bF[U]$-equivariant chain map such that $\iota^2\simeq \id$. 
\item $U^{-1} H_*(C)\iso \bF[U,U^{-1}]$ as relatively graded $\bF[U]$-modules.
\end{enumerate}
\end{define}
\begin{rem} 
We could naturally generalize the above definition to allow instead that $U^{-1} H_*(C)\iso \oplus^b \bF[U,U^{-1}]$ as $\bF[U]$-modules.
\end{rem}

For the purposes of studying the homology cobordism group, it is useful to consider the following notion of equivalence between $\iota$-complexes:
\begin{define}
 Two $\iota$-complexes $(C,\iota)$ and $(C',\iota')$ are \emph{locally equivalent} if there are grading preserving, $\bF[U]$-equivariant chain maps $F\colon C\to C'$ and $G\colon C'\to C$ such that $F \iota+\iota' F\simeq 0$ and $G\iota' +\iota G\simeq 0$ and such that $F$ and $G$ are induce isomorphisms between $U^{-1} H_*(C)$ and $U^{-1} H_*(C')$. 
\end{define}

We now discuss $\iota_K$-complexes. If $C$ is a free, finitely-generated chain complex over $\bF[\scU,\scV]$, there are two distinguished endomorphisms $\Phi$ and $\Psi$ of $C$. To define them, we write the differential $\d$ as a matrix with entries in $\bF[\scU,\scV]$. We differentiate the entries of this matrix with respect to $\scU$ to obtain $\Phi$, and we differentiate the entries of the matrix for $\d$ with respect to $\scV$ to obtain $\Psi$. See \cite{SarkarMovingBasepoints} and \cite{ZemQuasi} for appearances of these maps in the context of diffeomorphism maps on knot Floer homology.

\begin{define}\label{def:iota-K-complex} An \emph{$\iota_K$-complex} consists of a pair $(C,\iota_K)$ satisfying the the following:
\begin{enumerate}
\item $C$ is a chain complex over $\bF[\scU,\scV]$ which is homotopy equivalent to a free, finitely-generated chain complex over $\bF[\scU,\scV]$. Furthermore, $C$ is a equipped with a $\Q$-valued bigrading $(\gr_w,\gr_z)$, such that $\scU$ has bigrading $(-2,0)$ and $\scV$ has bigrading $(0,-2)$.
\item $\iota_K\colon C\to C$ is a homotopy automorphism such that $\iota_K$ is \emph{skew-graded} and \emph{skew-equivariant} (that is, $(\gr_w,\gr_z)(\iota_K(\xs))=(\gr_z,\gr_w)(\xs)$, $\iota_K\circ \scU=\scV\circ \iota_K$ and $\iota_K\circ \scV=\scU\circ \iota_K$).
\item $\scU^{-1} H_*(C/(\scV-1))\iso \oplus_{i=1}^b\bF[\scU,\scU^{-1}]$, for some $b\ge 1$.
\item $\iota_K^2\simeq \id+\Phi \Psi$, where $\Phi$ and $\Psi$ are the basepoint actions of $C$.
\end{enumerate}
\end{define}

Normally one restricts to the case that $b=1$ (i.e. there is a single tower). This restriction is not suitable for our purposes since we are interested in the case of knots in \emph{rational} homology 3-spheres. If $K\subset Y$ and $Y$ is a rational homology 3-sphere, then the integer $b$ in Definition~\ref{def:iota-K-complex} is the number of elements in $H_1(Y)$. If additionally $H_*(C/(\scV-1))\iso \oplus_{i=1}^b\bF[\scU]$, then we say that $C$ is of \emph{L-space type}. In particular, if $K\subset Y$ is a knot in a rational homology 3-sphere, then the tuple $(\cCFK(Y,K),\iota_K)$ is an $\iota_K$-complex, which is of L-space type if and only if $Y$ is a Heegaard Floer $L$-space.

\begin{define} We say that two $\iota_K$-complexes $(C,\iota_K)$ and $(C',\iota_K')$ are \emph{locally equivalent} if there exist maps
\[
F\colon C\to C'\quad \text{and} \quad G\colon C'\to C
\]
satisfying the following:
\begin{enumerate}
\item $F$ and $G$ are $\bF[\scU,\scV]$-equivariant and grading preserving.
\item $F\iota_K+\iota_K' F$ and $G\iota_K'+\iota_K G$ are $\bF[\scU,\scV]$-skew equivariantly chain homotopic to 0.
\item $F$ and $G$ induce isomorphisms between $\scU^{-1} H_*(\cC/(\scV-1))$ and $\scU^{-1}H_*(\cC/(\scV-1))$.
\end{enumerate}
\end{define}

Finally, we will also want the following non-involutive definition for complexes similar to those of knots in $S^3$, roughly following \cite{DHSThomcobknot}.

\begin{define} A \emph{knot-like} complex $\scC$ is a free finitely-generated bigraded chain complex over $\bF[\scU, \scV]$ with the property that $\scV^{-1} H_*(\scC/(\scU-1)) \simeq \bF[\scV, \scV^{-1}]$  and $\scU^{-1}H_*(\scC/(\scV-1))\simeq \bF[\scU, \scU^{-1}]$. If additionally $H_*(\scC/(\scV-1)) \simeq \bF[\scU]$, we say that $\scC$ is of \emph{$L$-space type} or \emph{$S^3$-type}. \end{define}

\subsection{Hypercubes and hyperboxes}

We begin with some preliminary definitions of hypercubes of chain complexes.

We write $\bE_n=\{0,1\}^n$. If $\veps,\veps'\in \bE_n$, we write $\veps\le \veps'$ if $\veps_j\le \veps_j'$ for each $j\in \{1,\dots, n\}$. We write $\veps<\veps'$ if $\veps\le \veps'$ and strict inequality holds for at least one index. We begin with the following definition, due to Manolescu and Ozsv\'{a}th \cite{MOIntegerSurgery}:
\begin{define}
An $n$-dimensional \emph{hypercube of chain complexes} consists of a collection of chain complexes $(C_{\veps},D_{\veps,\veps})$, as well as a collection of maps $D_{\veps,\veps'}\colon C_{\veps}\to C_{\veps'}$ whenever $\veps<\veps'$. We furthermore assume that the following compatibility condition is satisfied for each pair $(\veps,\veps'')$ satisfying $\veps<\veps''$:
\begin{equation}
\sum_{\veps\le \veps'\le \veps''} D_{\veps',\veps''}\circ D_{\veps,\veps'}=0.\label{eq:def-hypercube-chain-complex}
\end{equation}
\end{define}

Manolescu and Ozsv\'{a}th also define a notion of a \emph{hyperbox} of chain complexes, as follow. If $\ve{d}\in \Z_{\ge 0}^n$, then we write $\bE(\ve{d})=\{0,\dots, d_1\}\times \cdots \times \{0,\dots, d_n\}$. A \emph{hyperbox} of chain complexes $(C_{\veps},D_{\veps,\veps'})_{\veps\in \bE(\ve{d})}$ consists of a collection of chain complexes $(C_{\veps},D_{\veps,\veps})$ ranging over all $\veps\in \bE(\ve{d})$, together with a map $D_{\veps,\veps'}\colon C_{\veps}\to C_{\veps'}$ whenever $\veps<\veps'$ and $|\veps'-\veps|_{L^\infty}=1$. Furthermore, the compatibility condition in Equation~\eqref{eq:def-hypercube-chain-complex} is satisfied whenever $\veps<\veps''$ and $|\veps''-\veps|_{L^\infty}=1$.

Manolescu and Ozsv\'{a}th defined an important operation called \emph{compression}, which takes an $n$-dimensional hyperbox and returns an $n$-dimensional \emph{hypercube}. We refer the reader to \cite{MOIntegerSurgery}*{Section~5.2} for more background. We note that the description therein is also equivalent to a ``function-composition'' approach, described by Liu \cite{Liu2Bridge}*{Section~4.1.2}. See also \cite{HHSZNaturality}*{Section~2.1}.

\subsection{$A_\infty$-modules}

We now recall the standard notion of an $A_\infty$-module. See \cite{KellerNotes} for more background. We suppose that $A$ is an associative algebra over a ground ring $\ve{k}$, which is of characteristic 2, and let $\mu_2$ denote multiplication in $A$. A \emph{left $A_\infty$-module} over $A$ consists of left $\ve{k}$-module $M$, equipped with a $\ve{k}$-linear map for each $j\ge 0$ 
\[
m_{j+1}\colon A^{\otimes j}\otimes_{\ve{k}} M\to M
\]
such that if $a_n,\dots, a_1\in \cA$ and $\xs\in M$, then
\[
\sum_{j=0}^n m_{n-j+1}(a_n,\dots,a_{j+1}, m_{j+1}(a_{j},\dots, a_{1},\xs))+\sum_{k=1}^{ n-1} m_{n}(a_n,\dots, a_{k+1}a_{k},\dots, a_1, \xs)=0.
\]

We will additionally need to use the notion of a type-$D$ module, due to Lipshitz, Ozsv\'{a}th and Thurston \cite{LOTBordered}. If $A$ is an associative algebra, then a \emph{right type-$D$ module} over $A$ consists of a pair $(N,\delta^1)$ where $N$ is a right $\ve{k}$ module, and 
\[
\delta^1\colon N\to N\otimes_{\ve{k}} A
\]
is a $\ve{k}$-linear map. Furthermore, the following relation is satisfied
\[
(\id\otimes \mu_2)\circ (\delta^1\otimes \id)\circ \delta^1=0.
\]

If $N^{A}$ is a right type-$D$ module, ${}_A M$ is a left $A_\infty$-module, and one of $N^A$ and ${}_A M$ satisfies a boundedness assumption \cite{LOTBordered}*{Section~2}, then Lipshitz, Ozsv\'{a}th and Thurston define a chain complex $N^A\boxtimes {}_A M$. The underlying vector space is $N\otimes_{\ve{k}} M$. The differential is indicated by the diagram
\[
\d=
\begin{tikzcd}[row sep=.2cm] \ve{x}\ar[d]& \ve{y}\ar[dd]\\
\delta\ar[dr,Rightarrow]\ar[dd]&\,\\
\,& m\ar[d]\\
\,&\,
\end{tikzcd}
\]
In the above, $\delta$ indicates the (infinite) sum $\sum_{i=0}^\infty \delta^j$, where $\delta^j$ is obtained by composing $\delta^1$ $j$-times (with $\delta^0=\id$). See \cite{LOTBordered}*{Chapter~2} for more detailed treatment.

\subsection{The snake splitting lemma}

We begin with a basic result in homological algebra:

\begin{lem}\label{lem:snake} (Snake splitting lemma) Suppose that $A$, $B$ and $C$ are chain complexes of $\bF$-vector spaces and $i$ and $p$ are chain maps which make the following sequence exact
\[
\begin{tikzcd}[labels=description]
0\ar[r] &A \ar[r,"i"]& B\ar[r, "p"] \ar[l, "\sigma", bend right, dashed, swap] & C\ar[r] \ar[l, "s", bend right, dashed,swap]& 0.
\end{tikzcd}
\]
Furthermore, suppose that $s\colon C\to B$ is a splitting of $p$ as an $\bF$-linear map (not necessarily a chain map).
\begin{enumerate}
\item A splitting $\sigma\colon B\to A$ of $i$ is uniquely determined by the property that $\sigma \circ s=0$.
\item $sp+i\sigma=\id_B$.
\item The map $\sigma \d s\colon C\to A$ is a chain map.
\item The map $\Pi\colon \Cone(i\colon A\to B)\to C$ given by $p$ is a chain map.
\item The map $F\colon C\to \Cone(i\colon A\to B)$ given by $F=(\sigma \d s, s)$ is a chain map.
\item The maps $\Pi$ and $F$ are homotopy inverses. 
\end{enumerate}
\end{lem}
\begin{proof} (1) If $b\in B$, we note that $b+sp(b)$ is in the kernel of $p$, and hence factors uniquely as $i(a)$, for some $a\in A$. Let $\sigma(b)=a$. Clearly $\sigma s=0$. Uniqueness is an easy exercise.

(2) Note that $p(\id+sp+i\sigma)(b)=0$, for any $b$, so $(\id+sp+i\sigma)(b)$ is in the image of $i$ for all $b\in B$. To show something in the image of $i$ is zero, it is sufficient to show it is zero after composing with $\sigma$. We easily compute that $\sigma(\id+sp+i\sigma)(b)=0$, so $\id+sp+i \sigma=0$.

(3) To see that $\sigma \d s$ is a chain map, we note that since $p$ is surjective and $i$ is injective, it is sufficient to show that $[i\sigma \d_B s p,\d_B]=0$. We note that
\[
0=\d_{B}^2+\d_B^2=(sp+i\sigma)\d_B(sp+ i\sigma)\d_B+\d_B(sp+i\sigma)\d_B(sp+ i\sigma).
\]
Expanding the above expression out, most terms cancel, except for the relation $[i\sigma \d_B s p,\d_B]=0$.

(4) This follows from the relation $p\circ i=0$.

(5) The property of $F$ being a chain map is equivalent to $\sigma \d s$ being a chain map, and $s$ being a null-homotopy of $i \sigma \d s$. To wit:
\[
i\sigma \d s=(\id+sp)\d s=\d s+s p \d s= \d s+s \d p s=[\d, s].
\]

(6) Note that $\Pi F=\id$, so we need only construct a homotopy $H\colon \Cone(A\to B)\to \Cone(A\to B)$ so that $F \Pi+\id=[\d, H]$. We compute 
\[
F \Pi(a,b)=(\sigma \d s p b, sp b).
\]
The homotopy is actually just $H(a,b)=(\sigma b,0)$. To see that this works,
we compute that
\[
[\d_{\Cone}, H](a,b)=(\sigma i a+[\d, \sigma](b),i\sigma b).
\]
The relation $\id_{\Cone}=F\Pi+[\d, H]$ amounts to verifying that
\[
[\d, \sigma]=\id+\sigma i+\sigma \d s p.
\]
To this end,
\[
\id+\sigma i+\sigma \d s p= \sigma \d (\id + i \sigma)=[\d, \sigma].
\]
\end{proof}

\subsection{The homological perturbation lemma for $A_\infty$-modules}

The homological perturbation lemma is an important technique in homological algebra which as many variations. The first instance of this result is usually attributed to Kadeishvili's homotopy transfer theorem for $A_\infty$-algebras \cite{Kadeishvili_Ainfinity}. A reinterpretation in terms of trees is due to Kontsevich and Soibelman \cite{KontsevichSoibelman}. We recall the version for $A_\infty$-modules.

\begin{lem}\label{lem:homological-perturbation-modules}
Suppose that $( M,m_{j}^M)$ is a left $A_\infty$-module over a $dg$-algebra $A$. We assume $A$ is an algebra over a ring $\ve{k}$ (for our purposes, $\ve{k}=\bF_2$ is sufficient). Suppose that $(Z,m_1^Z)$ is a chain complex which is a left $\ve{k}$-module. Suppose also that we have maps
\[
i\colon Z\to M\quad \pi\colon M\to Z\quad h\colon M\to M
\]
which satisfy the following:
\begin{enumerate}
\item $\d_{\Mor}(i)=0$ and $\d_{\Mor}(\pi)=0$.
\item $\pi\circ i=\id_{Z}$.
\item $i\circ \pi=\id_M+\d_{\Mor}(h)$.
\item $h\circ h=0$.
\item $h\circ i=0$.
\item $\pi\circ h=0$.
\end{enumerate}
In the above, $\d_{\Mor}(f)$ means $[f,m_1]$.  Then there is an $A_\infty$-module structure $(Z, m_{j}^Z)$, extending the differential $m_1^Z$, which is homotopy equivalent to $(M,m_{j}^M)$. In fact, there are morphisms of $A_\infty$-modules  $I\colon {}_{A} Z\to {}_{A} M$, $\Pi\colon {}_{A} M\to {}_{A} Z$ and $H\colon {}_{A} M\to {}_{A} M$, extending $\pi$, $i$ and $h$, which satisfy identical relations to the above (as long as we interpret $\d_{\Mor}$ and composition in the sense of $A_\infty$-module morphisms). 
\end{lem}

In the homological perturbation lemma, the maps $I$, $\Pi$ and $H$ are canonically determined. See Figure~\ref{fig:Z-structure-rels}. Therein, the map
\[
\Delta\colon A^{\otimes n}\to \bigoplus_{k=0}^n A^{\otimes k}\otimes A^{\otimes n-k}
\]
is the canonical co-multiplication, given by
\begin{equation}
\Delta(a_n\otimes \cdots \otimes a_1)=\sum_{k=0}^n (a_n\otimes \cdots\otimes  a_{k+1})\otimes (a_{k}\otimes \cdots \otimes a_1).
\label{eq:Delta-def}
\end{equation}
\begin{figure}[ht]
\[
 m^Z(\ve{a},x)=\hspace{-.5cm}\begin{tikzcd}[column sep=0cm, row sep=.3cm]
\ve{a} \ar[d,Rightarrow]& \, & x \ar[d,dashed]\\
\Delta
	\ar[drr,Rightarrow, bend right =3]
	\ar[dddrr,Rightarrow, bend right=5]
	\ar[dddddrr,Rightarrow, bend right =8]
	\ar[dddddddrr,Rightarrow, bend right=11]
&&i \ar[d,dashed]\\
&&m_{>0}^M \ar[d,dashed]\\
&& h\ar[d,dashed]\\
&&m_{>0}^M \ar[d,dashed]\\
&& h\ar[d,dashed]\\
&& \vdots \ar[d,dashed]\\
&& h\ar[d,dashed]\\
&& m_{>0}^M\ar[d,dashed]\\
&& \pi\ar[d,dashed]\\
&& \,
\end{tikzcd}
\quad\Pi(\ve{a},x):=\hspace{-.5cm}
\begin{tikzcd}[column sep=0cm, row sep=.3cm]
 \ve{a} \ar[d,Rightarrow]& \, & x \ar[d,dashed]\\
\Delta
	\ar[drr,Rightarrow, bend right =3]
	\ar[dddrr,Rightarrow, bend right=5]
	\ar[dddddrr,Rightarrow, bend right =8]
	\ar[dddddddrr,Rightarrow, bend right=11]
&&h \ar[d,dashed]\\
&&m_{>0}^M \ar[d,dashed]\\
&& h\ar[d,dashed]\\
&&m_{>0}^M \ar[d,dashed]\\
&& h\ar[d,dashed]\\
&& \vdots \ar[d,dashed]\\
&& h\ar[d,dashed]\\
&& m_{>0}^M\ar[d,dashed]\\
&& \pi\ar[d,dashed]\\
&& \,
\end{tikzcd}
 \quad 
I(\ve{a},x):=\hspace{-.5cm}
\begin{tikzcd}[column sep=0cm, row sep=.3cm]
 \ve{a} \ar[d,Rightarrow]& \, & x \ar[d,dashed]\\
\Delta
	\ar[drr,Rightarrow, bend right =3]
	\ar[dddrr,Rightarrow, bend right=5]
	\ar[dddddrr,Rightarrow, bend right =8]
	\ar[dddddddrr,Rightarrow, bend right=11]
&&i \ar[d,dashed]\\
&&m_{>0}^M \ar[d,dashed]\\
&& h\ar[d,dashed]\\
&&m_{>0}^M \ar[d,dashed]\\
&& h\ar[d,dashed]\\
&& \vdots \ar[d,dashed]\\
&& h\ar[d,dashed]\\
&& m_{>0}^M\ar[d,dashed]\\
&& h\ar[d,dashed]\\
&& \,
\end{tikzcd}
\quad
H(\ve{a},x):=\hspace{-.5cm}
\begin{tikzcd}[column sep=0cm, row sep=.3cm]
 \ve{a} \ar[d,Rightarrow]& \, & x \ar[d,dashed]\\
\Delta
	\ar[drr,Rightarrow, bend right =3]
	\ar[dddrr,Rightarrow, bend right=5]
	\ar[dddddrr,Rightarrow, bend right =8]
	\ar[dddddddrr,Rightarrow, bend right=11]
&&h \ar[d,dashed]\\
&&m_{>0}^M \ar[d,dashed]\\
&& h\ar[d,dashed]\\
&&m_{>0}^M \ar[d,dashed]\\
&& h\ar[d,dashed]\\
&& \vdots \ar[d,dashed]\\
&& h\ar[d,dashed]\\
&& m_{>0}^M\ar[d,dashed]\\
&& h\ar[d,dashed]\\
&& \,
\end{tikzcd}
\]
\caption{The $A_\infty$-module structure maps on ${}_{A}Z$, and the morphisms $\Pi$, $I$ and $H$. In the above, $m_{>0}^M$ denotes any $m_j^M$ for $j>1$. Also, we write $\Delta$ for a repeated application of the map defined in Equation~\ref{eq:Delta-def}.}
\label{fig:Z-structure-rels}
\end{figure}

\subsection{A homological perturbation lemma for hypercubes}
\label{sec:homological-perturbation-hypercubes}

In this section, we describe an algorithm for constructing hypercubes. It is a natural adaptation of the homological perturbation lemma for $A_\infty$-algebras and $A_\infty$-modules to the setting of hypercubes. This result is similar to work of Huebschmann and Kadeishvili \cite{HK_Homological_Perturbation}.

\begin{lem}\label{lem:homological-perturbation-cubes}
Suppose that $(A_\veps,f_{\veps,\veps'})_{\veps\in \bE_n}$ is a hypercube of chain complexes, and $(B_\veps)_{\veps\in \bE_n}$ is a collection of chain complexes. Furthermore, suppose there are maps 
\[
\pi_{\veps}\colon A_\veps\to B_\veps\qquad i_{\veps}\colon B_\veps\to A_\veps \qquad h_\veps\colon A_\veps\to A_\veps,
\]
satisfying
\begin{equation}
\pi_{\veps}\circ i_{\veps}=\id_{B_\veps}, \quad i_{\veps}\circ \pi_{\veps}=\id_{A_\veps}+[\d, h_\veps], \quad \pi_\veps\circ h_\veps=0,\quad h_\veps\circ i_{\veps}=0\quad \text{and} \quad h_\veps\circ h_\veps=0.\label{eq:homological-perturbation-assumptions}
\end{equation}
With the above data chosen, there are canonical structure maps $g_{\veps,\veps'}\colon B_{\veps}\to B_{\veps'}$ so that $(B_\veps,g_{\veps,\veps'})$ is a hypercube of chain complexes. Furthermore, there are homomorphisms of hypercubes
\[
\Pi\colon (A_\veps,f_{\veps,\veps'})\to (B_\veps,g_{\veps,\veps'})\quad \text{and} \quad I\colon  (B_\veps,g_{\veps,\veps'})\to (A_\veps,f_{\veps,\veps'})
\]
as well as a morphism $H\colon (A_\veps,f_{\veps,\veps'})\to (A_\veps,f_{\veps,\veps'})$
such that Equation~\eqref{eq:homological-perturbation-assumptions} is also satisfied with $\Pi$, $I$ and $H$ replacing $\pi_\veps$, $i_\veps$ and $h_\veps$.
\end{lem}

The structure maps $g_{\veps,\veps'}$ and the morphisms $\Pi$, $I$ and $H$ have conceptually simple formulas. The map $g_{\veps,\veps'}$ is given by the formula
\[
g_{\veps,\veps'}=\sum_{\veps=\veps_1<\cdots<\veps_n=\veps'} \pi_{\veps_n}\circ f_{\veps_{n-1},\veps_{n}}\circ h_{\veps_{n-1}}\circ f_{\veps_{n-2},\veps_{n-1}}\circ \cdots \circ h_{\veps_2}\circ f_{\veps_1,\veps_2}\circ i_{\veps_1}.
\]
We define $\Pi_{\veps,\veps}=\pi_{\veps}$, $I_{\veps,\veps}=i_{\veps}$ and $H_{\veps,\veps}=h_{\veps}$. Additionally, we set
\[
\begin{split}
\Pi_{\veps,\veps'}&=\sum_{\veps=\veps_1<\cdots<\veps_n=\veps'} \pi_{\veps_n}\circ f_{\veps_{n-1},\veps_n}\circ h_{\veps_{n-1}}\circ \cdots \circ f_{\veps_1,\veps_2}\circ h_{\veps_1}\\
I_{\veps,\veps'}&=\sum_{\veps=\veps_1<\cdots<\veps_n=\veps'}  h_{\veps_{n}}\circ f_{\veps_{n-1},\veps_n}\circ \cdots \circ f_{\veps_1,\veps_2} \circ i_{\veps_1}\\
H_{\veps,\veps'}&=\sum_{\veps=\veps_1<\cdots<\veps_n=\veps'}  h_{\veps_{n}}\circ f_{\veps_{n-1},\veps_n}\circ \cdots \circ f_{\veps_1,\veps_2} \circ h_{\veps_1}.
\end{split}
\]

\begin{rem}
 The above lemma may also be applied to hyper\emph{boxes} of chain complexes. For hyperboxes, we apply the above statement to each constituent hypercube.
\end{rem}

We break the proof into pieces.

\begin{lem} \label{lem:(B,g)-hypercube}
 $(B_{\veps},g_{\veps,\veps'})_{\veps\in \bE_n}$ is a hypercube of chain complexes.
\end{lem}
\begin{proof}

  The desired hypercube relations read
  \begin{equation}
  [\d, g_{\veps,\veps''}]=\sum_{\veps<\veps'<\veps''} g_{\veps',\veps''}\circ g_{\veps,\veps'}.
  \label{eq:g-veps-hypercube-relations}
  \end{equation}
  We now verify the above relation. Take a sequence $\veps=\veps_1<\cdots<\veps_m=\veps''$ and consider the commutator
\[
[\d,  \pi_{\veps_m}\circ f_{\veps_{m-1},\veps_m}\circ h_{\veps_{m-1}}\circ f_{\veps_{m-2},\veps_{m-1}}\circ h_{\veps_{m-2}}\circ \cdots \circ h_{\veps_2}\circ f_{\veps_1,\veps_2}\circ i_{\veps_1}].
\]
To compute the commutator, we take the expression
\[
\d \circ   \pi_{\veps_m}\circ f_{\veps_{m-1},\veps_m}\circ h_{\veps_{m-1}}\circ f_{\veps_{m-2},\veps_{m-1}}\circ h_{\veps_{m-2}}\circ \cdots \circ h_{\veps_2}\circ f_{\veps_1,\veps_2}\circ i_{\veps_1},
\]
and move $\d$ from the left to the right, using the hypercube relations of $(A_{\veps},f_{\veps,\veps'})$ and Equation~\eqref{eq:homological-perturbation-assumptions}. We obtain the following extra terms:
\begin{enumerate}
\item $\pi_{\veps_m}\circ f_{\veps_m,\veps_{m-1}}\circ h_{\veps_{m-1}}\circ \cdots \circ  [\d, f_{\veps_j,\veps_{j-1}}]\circ h_{\veps_{j-1}}\circ \cdots \circ h_{\veps_2}\circ f_{\veps_2,\veps_1}\circ i_{\veps_1}$. These appear when we commute $\d$ past $f_{\veps_j,\veps_{j-1}}$.
\item $\pi_{\veps_m}\circ f_{\veps_m,\veps_{m-1}}\circ h_{\veps_{m-1}}\circ \cdots \circ  f_{\veps_{j},\veps_{j+1}}\circ i_{\veps_j} \circ \pi_{\veps_j} \circ f_{\veps_{j-1},\veps_j} \cdots \circ h_{\veps_2}\circ f_{\veps_1,\veps_2}\circ i_{\veps_1}$. These appear when we commute $\d$ past $h_{\veps_j}$.
\item  $\pi_{\veps_m}\circ f_{\veps_{m-1},\veps_m}\circ h_{\veps_{m-1}}\circ \cdots \circ  f_{\veps_j,\veps_{j+1}}\circ f_{\veps_{j-1},\veps_j} \cdots \circ h_{\veps_2}\circ f_{\veps_1,\veps_2}\circ i_{\veps_1}$. These appear when we commute $\d$ past $h_{\veps_j}$.
\end{enumerate}
  When summing over all sequences $\veps=\veps_1<\cdots<\veps_m=\veps''$, terms of the first and third types cancel, and we are left with the right hand side of Equation~\eqref{eq:g-veps-hypercube-relations}, which completes the proof.
\end{proof}

We now prove the remaining subclaims of Lemma~\ref{lem:homological-perturbation-cubes}:

\begin{lem}\,
\begin{enumerate}
\item $\d_{\Mor}(\Pi)=0$ and $\d_{\Mor}(I)=0$,
\item $\Pi\circ I=\id_B$,
\item $\Pi\circ I=\id_A+\d_{\Mor}(H)$,
\item $\Pi\circ H=0$,
\item $H\circ I=0$. 
\end{enumerate}
\end{lem}
\begin{proof} The proof follows from the similar reasoning to Lemma~\ref{lem:(B,g)-hypercube}. Consider first the claim that $\Pi$ and $I$ are cycles. The claim for $\veps'=\veps$ is clear. Suppose $\veps< \veps''$ and $\veps=\veps_1<\cdots<\veps_m=\veps''$. Consider the expression
\[
\d\circ  \pi_{\veps_m}\circ f_{\veps_m,\veps_{m-1}}\circ h_{\veps_{m-1}}\circ \cdots \circ f_{\veps_1,\veps_2}\circ h_{\veps_1}
\]
and consider the effect of moving $\d$ from the left to right. The difference from $
f_{\veps_m,\veps_{m-1}}\circ h_{\veps_{m-1}}\circ \cdots \circ f_{\veps_1,\veps_2}\circ h_{\veps_1}\circ \d$ consists of the following terms:
\begin{enumerate}
\item\label{Phi-hyperbox-mor-1} $\pi_{\veps_m}\circ f_{\veps_{m-1},\veps_m}\circ h_{\veps_{m-1}}\circ \cdots \circ f_{\veps',\veps_{j+1}}\circ f_{\veps_j,\veps'}
\circ h_{\veps_{j-1}}\circ \cdots \circ f_{\veps_{1},\veps_2}\circ h_{\veps_1}$, for $\veps_j<\veps'<\veps_{j+1}$. (These appear when we commute $\d$  past $f_{\veps_{j+1},\veps_j}$).
\item\label{Phi-hyperbox-mor-2} $\pi_{\veps_m} \circ f_{\veps_m,\veps_{m-1}}\circ \cdots f_{\veps_{j+1},\veps_j}\circ f_{\veps_{j-1},\veps_j}\circ \cdots \circ f_{\veps_1,\veps_2}\circ h_{\veps_1}$. (These appear when we commute $\d$ past an $h_{\veps_j}$).
\item\label{Phi-hyperbox-mor-3} $\pi_{\veps_m}\circ f_{\veps_{m-1},\veps_m}\circ \cdots f_{\veps_{j},\veps_{j+1}}\circ  i_{\veps_j}\circ \pi_{\veps_j}\circ f_{\veps_{j-1},\veps_j}\circ h_{\veps_{j-1}}\circ \cdots \circ f_{\veps_1,\veps_2} \circ h_{\veps_1}$. (These appear when we commute $\d$ past $h_{\veps_j}$ for $j\neq 1$).
\item\label{Phi-hyperbox-mor-4} $\pi_{\veps_m} \circ f_{\veps_m,\veps_{m-1}}\circ h_{\veps_{m-1}} \circ \cdots \circ h_{\veps_2}\circ  f_{\veps_1,\veps_2}$. (These appear when we commute $\d$ past $h_{\veps_1}$).
\item\label{Phi-hyperbox-mor-5} $\pi_{\veps_m}\circ f_{\veps_{m-1},\veps_m}\circ h_{\veps_{m-1}}\circ \cdots \circ h_{\veps_2}\circ  f_{\veps_1,\veps_2} \circ  i_{\veps_1}\circ \pi_{\veps_1}$. (These also appear when we commute $\d$ past $h_{\veps_1}$).
\end{enumerate}
Terms \eqref{Phi-hyperbox-mor-1} and \eqref{Phi-hyperbox-mor-2} cancel when summed over all increasing sequences. Terms~\eqref{Phi-hyperbox-mor-3} and ~\eqref{Phi-hyperbox-mor-5} sum to $\d_{B}\circ \Pi$. Terms \eqref{Phi-hyperbox-mor-4} sum to $\Pi\circ \d_A$. The relation so-obtained corresponds exactly to $\d_{\Mor}(\Pi)=0$.

The proof that $\d_{\Mor}(I)=0$ is similar, and we leave the manipulation to the reader.

Next, the claim that $\Pi\circ I=\id_B$ follows from the facts that $h_\veps\circ h_\veps=0$ and $\pi_\veps\circ i_\veps=\id_{B_\veps}$. The claims that $H\circ I=0$ and $\Pi\circ H=0$ are similar.

Finally, we consider the claim that $I\circ \Pi=\id_A+\d_{\Mor}(H)$. Similar to before, we consider moving $\d$ from left to right in the expression
\[
\d\circ h_{\veps_m}\circ f_{\veps_{m-1},\veps_m}\circ h_{\veps_{m-1}}\circ \cdots \circ h_{\veps_2}\circ f_{\veps_1,\veps_2}\circ h_{\veps_1}. 
\]
In addition to $h_{\veps_m}\circ f_{\veps_{m-1},\veps_m}\circ h_{\veps_{m-1}}\circ \cdots \circ h_{\veps_2}\circ f_{\veps_1,\veps_2}\circ h_{\veps_1}\circ \d$, we obtain the following terms:
\begin{enumerate}
\item\label{list:[d,H]-summand-1} $h_{\veps_m}\circ f_{\veps_{m-1},\veps_m}\circ \cdots \circ f_{\veps',\veps_{j+1}}\circ f_{\veps_j,\veps'}\circ \cdots \circ f_{\veps_1,\veps_2} \circ h_{\veps_1}$. These appear when we commute $\d$ past an $f_{\veps_j,\veps_{j+1}}$.
\item\label{list:[d,H]-summand-2}  $h_{\veps_m}\circ f_{\veps_{m-1},\veps_m} \circ \cdots\circ f_{\veps_{j},\veps_{j+1}}\circ f_{\veps_{j-1},\veps_j} \circ \cdots \circ f_{\veps_1,\veps_2}\circ h_{\veps_1}$. These appear when we commute $\d$ past an $h_{\veps_j}$ for $j\neq 1,m$.
\item \label{list:[d,H]-summand-3} $h_{\veps_m}\circ f_{\veps_{m-1},\veps_m}\circ \cdots \circ f_{\veps_{j},\veps_{j+1}} \circ i_{\veps_j}\circ \pi_{\veps_j} \circ f_{\veps_{j-1},\veps_j}\circ h_{\veps_{j-1}} \circ \dots \circ f_{\veps_1,\veps_2}\circ h_{\veps_1}$. These appear when we commute $\d$ past an $h_{\veps_j}$.
\item \label{list:[d,H]-summand-4} $f_{\veps_{m-1},\veps_m}\circ h_{\veps_{m-1}}\circ \cdots \circ f_{\veps_1,\veps_2} \circ h_{\veps_1}$ and $h_{\veps_m}\circ f_{\veps_{m-1},\veps_m}\circ \cdots \circ h_{\veps_2}\circ f_{\veps_1,\veps_2}$. These appear when we commute $\d$ past $h_{\veps_m}$ or $h_{\veps_1}$, respectively.
\end{enumerate}
When summed over all sequences, summands \eqref{list:[d,H]-summand-1} and~\eqref{list:[d,H]-summand-2} cancel. Terms of type~\eqref{list:[d,H]-summand-3} sum to $I\circ \Pi$. Terms of type~\eqref{list:[d,H]-summand-4}, together with the initial terms involving $\d$ on their ends, sum to $\d_{\Mor}(H)$. The proof is now complete.
\end{proof}

\section{An expanded model for the dual knot complex}
\label{sec:large-surgery-expanded}

In this section, we write down an expanded model of the dual knot formula by taking the connected sum of $K$ with a Hopf link, and then applying a variation on the normal surgery formula to surger on $K$.

\subsection{Preliminaries}

Let $K$ be a knot in $S^3$ and $\mu$ be a meridian. Let $L$ denote $K\cup \mu$, a link in $S^3$.  We orient $\mu$ so that $K\cup \mu$ is the connected sum of a \emph{positive} Hopf link and $K$. Let $n>0$ be an integral, positive surgery coefficient. We have that $[\mu]\in H_1( S^3_n(K))$ is a generator of $H_1(S^3_n(K))\iso \Z/n$. Let $W_n(K)$ be the surgery cobordism from $S^3$ to $S^3_n(K)$, and write $W_n':=W_n'(K)$ for $-W_n(K)$, viewed as a cobordism from $S^3_n(K)$ to $S^3$. 

Write $S_K$ for the core of the 2-handle in $W_n(K)$, and write $S_\mu$ for $[0,1]\times \mu$. We may naturally view $(W_n',S_K\cup S_\mu)$ as a link cobordism from $(S^3_n(K),\mu)$ to $(S^3,L)$. The class $[S_K]\in H_2(W_n',\d W_n';\Z)$ is the preimage of a class in $H_2(W_n';\Z)$, for which we also write $[S_K]$. This class is obtained by capping $K=\d S_K$ with a Seifert surface.

 The natural map
\[
H_2(W'_n;\Q)\to H_2(W'_n,\d W'_n;\Q)
\]
is an isomorphism. We write $[\widehat{S}_\mu]\in H_2(W_n';\Q)$ for a preimage of $[S_\mu]$ under this map.

 It is easy to see that
\begin{equation}
[S_K]\cdot [\widehat{S}_\mu]=- 1.\label{eq:Sigma-mu-Sigma-K}
\end{equation}
We can view the intersection as occurring between $K$ and a disk $D$ in $S^3$ such that $\d D=-\mu$, used to cap $S_\mu$. In particular, as elements of $H_2(W_n',\d W'_n;\Q)$, with respect to the $W_n'$ orientation we have 
\[
[S_\mu]= \frac{1}{n} [S_K].
\]
Furthermore, note that
\[
[\widehat{S}_\mu]^2=-\frac{1}{n}.
\]

\subsection{Alexander grading changes}

In this section, we recall and expand on the grading change formulas associated to the link cobordism maps of the fourth author \cite{ZemCFLTQFT}. The gradings associated to link cobordisms between integrally null-homologous links are described in  \cite{ZemAbsoluteGradings}. In this section, we describe the grading change formulas for cobordisms between \emph{rationally} null-homologous links.

Suppose that $L_1\subset Y_1$ and $L_2\subset Y_2$ are two $\ell$-component links. Write $L_j^1,\dots, L_j^\ell$ for the components of $L_j$. Suppose that each component is rationally null-homologous. We suppose further that we have choices $T_1=(T_1^i)_{1\le i\le \ell}$ and $T_2=(T_2^i)_{1\le i\le \ell}$ of collections of rational 2-chains, such that $\d T_j^i= -L_j^i$; that is, $T_j^i$ is a \emph{rational Seifert surface} for $L_j^i$.

Suppose that $(W,S)\colon (Y_1,L_1)\to (Y_2,L_2)$ is an oriented link cobordism. The choices of rational Seifert surfaces give a unique lift of $[S_i]\in H_2(W,\d W, \Q)$ to $H_2(W,\Q)$. Write $[\widehat{S}_i]$ for this class. If $S$ has components $S_1,\dots, S_\ell$, such that $\d S_i=-L_1^i\cup L_2^i$, then $[\widehat S]$ decomposes as a sum $[\widehat{S}_1]+\dots +[\widehat{S}_\ell]$. 

We now discuss decorations and link cobordisms. The cobordism maps from \cite{ZemCFLTQFT} require a choice of decoration. In this section, we restrict to the case where $S$ is a collection of annuli, and we pick a decoration $\cA$ on $S$ consisting of a pair of arcs on each component of $S$, which connect $L_1$ to $L_2$.

\begin{prop}\,\label{prop:grading-changes}
\begin{enumerate}
\item Suppose that $L\subset Y$ is an $\ell$-component link, each of whose components is rationally null-homologous. Let $T=(T^i)_{1\le i\le \ell}$ denote a collection of rational Seifert surfaces for the components of $L$. Then there is a well-defined $\ell$-component Alexander grading $A_T=(A_{T,1},\dots, A_{T,\ell})$ on $\cCFL(Y,L,\frs)$. If $\frs$ is torsion, the grading is independent of the choice of $S$.
\item Suppose $(W,S)\colon (Y_1,L_2)\to (Y_2,L_2)$ is a link cobordism, consisting of $\ell$ disjoint annuli, and let $\cF=(S,\cA)$ denote the decoration of $S$ where we decorate each component with a pair of longitudinal arcs. Suppose that $T_1$ and $T_2$ are collections of rational Seifert surfaces, as above. Then
\[
A_{T_2,i}(F_{W,\cF,\frs}(\xs))-A_{T_1,i}(\xs)=\frac{\langle c_1(\frs),[\widehat S_i] \rangle-[S]\cdot [\widehat S_i]}{2}.
\]
\end{enumerate}
\end{prop}

The proof is essentially identical to the proof in the integrally null-homologous case \cite{ZemAbsoluteGradings}, which is itself adapted from the original strategy of Ozsv\'{a}th and Szab\'{o} for defining an absolute grading on $\CF^-(Y,\frs)$ \cite{OSIntersectionForms}. The only difference is that at each step we work with classes in $H_2(W,\Q)$ instead of $H_2(W,\Z)$. We outline the argument briefly. One first expresses $Y\setminus N(L)$ as integral surgery on the complement of an unlink in $S^3$. Using such a presentation, the absolute grading on $\cCFL(Y,L,\frs)$ is defined via a concrete formula, as in \cite{ZemAbsoluteGradings}*{Section~5.5}. From here, one proves that the absolute grading is independent of the surgery presentation using Kirby calculus for manifolds with boundary.  By the argument in \cite{ZemAbsoluteGradings}*{Section~4}, it is sufficient to show that the grading is independent of blow-ups/downs contained in 3-balls, as well as blow-ups/downs along meridians of $L$, and handleslides. The verification of independence of these moves is no different than in the integrally null-homologous case, as in \cite{ZemAbsoluteGradings}*{Section~6}. With this in place, the grading change formula is automatic from the definition of the grading.

\subsection{An expanded large surgery formula for dual knots}

We now consider the 2-handle map
\[
\Gamma_s\colon \cCFK(S_n^3(K),\mu,p)\to \cCFL(S^3,L),
\]
which counts holomorphic triangles with $\frs_w(\psi)=\frx_s$. A few comments are in order. Firstly, we are thinking of $\cCFL(S^3,L)$ as being a four basepoint link Floer complex, over the ring $\bF[\scU_1,\scV_1,\scU_2,\scV_2]$, wherein $\scU_1$ and $\scV_1$ correspond to $K$, while $\scU_2$ and $\scV_2$ correspond to $\mu$. Moreover, the point $p$ is an extra free basepoint, so that the complex $\cCFK(S_n^3(K),\mu,p)$ is a module over $\bF[U,\scU_2,\scV_2]$. The variable $U$ corresponds to $\scU_1\scV_1$. We regard $\cCFK(S^3_n(K),\mu,p)$ as having a two component Alexander grading, but being concentrated in Alexander grading $\{0\}\times (r+\Z)$, for some $r\in \Q$.

Observe that there is a chain homotopy $U\simeq \scU_2\scV_2$. Such a homotopy is obtained by picking an arc $\lambda$ on the Heegaard surface which connects $p$ to a basepoint and using the relative homology map $A_{\lambda}$ (cf. \cite{ZemGraphTQFT}*{Lemma~5.1}). In fact, the normal proof of independence of adding basepoints \cite{OSLinks}*{Proposition~6.5} adapts to show that there is a chain isomorphism (for appropriate diagrams and almost complex structures)
\begin{equation}
\cCFK(S^3_n(K),\mu,p)\simeq \Cone\bigg(U+\scU_2\scV_2\colon \cCFK(S_n^3(K),\mu)[U]\to \cCFK(S_n^3(K),\mu)[U]\bigg).
\label{eq:stabilize-add-basepoint}
\end{equation}

Recall that $\frx_s$ is the $\Spin^c$ structure on $W_n'$ such that
\begin{equation}
\langle c_1(\frx_s),[S_K]\rangle +n=2s.\label{eq:x_s-def}
\end{equation}
One easily computes from this that
\begin{equation}
c_1^2(\frx_s)=-\frac{(2s-n)^2}{n}.
\label{eq:c1^2-x-s}
\end{equation}

Let $[s]$ denote the restriction of $\frx_s$ to $S^3_n(K)$.

\begin{lem}\label{lem:grading-change}
 The Alexander grading change of the cobordism map $\Gamma_s$ is 
 \[
A(\Gamma_s)=\left(s+ \frac{1}{2}, \frac{ 2s+1}{2n}\right). 
 \]
 (Recall that the first component corresponds to $K$, while the second corresponds to $\mu$).
\end{lem}
\begin{proof}
Using Proposition~\ref{prop:grading-changes} and Equation~\eqref{eq:Sigma-mu-Sigma-K} we see that the change of $A_1$ is
\[
\frac{\langle c_1(\frx_s), [S_K]\rangle -S_K^2-S_\mu\cdot S_K}{2}=s+ \frac{1}{2}.
\]
Similarly, the change of $A_2$ is
\[
\frac{\langle c_1(\frx_s),[S_\mu]\rangle -S_\mu^2-S_\mu\cdot S_K}{2}= \frac{2s+1}{2n}.
\]
\end{proof}

\begin{cor} 
If $[s]\in \Spin^c(S^3_n(K))$, then the Alexander grading on $\cCFK(S_n^3(K),\mu,[s])$ takes values in the set 
\[
\frac{1}{2}-\frac{ 2s+1}{2n}+\Z=\frac{n- 2s-1}{2n}+\Z.
\]
\end{cor}

\begin{prop}\label{prop:large-surgery}
 Suppose $s\in \Z$ is fixed, and $n$ is sufficiently large, relative to $s$. Then the map
 \[
\Gamma_s\colon \cCFK_*(S_n^3(K), \mu,p,[s])\to \cCFL_{(s+ 1/2,*+h)}(S^3,L)
 \]
is an isomorphism, where $*$ denotes the Alexander grading, and $h= (2s+1)/2n$ is the Alexander grading shift from Lemma~\ref{lem:grading-change}.
\end{prop}
\begin{proof}
The proof is essentially identical to \cite{OSKnots}*{Theorem~4.1}. One picks a diagram with a special winding region. For sufficiently large framing $n$, one may use a standard area filtration argument to show that the triangle map $\Gamma_s$ is a chain isomorphism. The stated grading change follows from Lemma~\ref{lem:grading-change}.
\end{proof}

We now recall the link Floer complexes of the Hopf links. The positive and negative Hopf link complexes $\cH^+$ and $\cH^-$ are shown below:
\[
\cH^+=\begin{tikzcd}[labels=description,row sep=1cm, column sep=1cm] \ve{a} \ar[d, "\scV_1"]\ar[r, "\scU_2"]& \ve{b}\\
\ve{c}& \ve{d} \ar[u, "\scU_1"] \ar[l, "\scV_2"]
\end{tikzcd}
 \quad
  \text{and} 
\quad 
\cH^-=\begin{tikzcd}[labels=description,row sep=1cm, column sep=1cm] \ve{a}^\vee& \ve{b}^\vee\ar[l, "\scU_2"] \ar[d, "\scU_1"]\\
\ve{c}^\vee\ar[r, "\scV_2"] \ar[u, "\scV_1"]& \ve{d}^\vee
\end{tikzcd}
\]
The Alexander gradings are
\[
\begin{split} A(\ve{a})&=(\tfrac{1}{2}, -\tfrac{1}{2}),\\
A(\ve{b})&=(\tfrac{1}{2}, \tfrac{1}{2}),\\
A(\ve{c})&=(-\tfrac{1}{2}, -\tfrac{1}{2}),\\
A(\ve{d})&= (-\tfrac{1}{2}, \tfrac{1}{2}).
\end{split}
\]
The Maslov gradings $\gr=(\gr_w,\gr_z)$ are given by
\[
\begin{split}
\gr(\ve{a})&=(-\tfrac{1}{2},-\tfrac{1}{2}),\\
\gr(\ve{b})&=(\tfrac{1}{2},-\tfrac{3}{2}),\\
\gr(\ve{c})&=(-\tfrac{3}{2},\tfrac{1}{2}),\\
\gr(\ve{d})&=(-\tfrac{1}{2},-\tfrac{1}{2})
\end{split}
\]

\noindent The dual generators have gradings multiplied by $-1$.

For large $n$, we may compute the complex $\cCFK(S_n^3(K), \mu,p,[s])$ using the tensor product formula and the large surgery formula:
\begin{equation}
\cCFK(S_n^3(K), \mu,p,[s])\iso \begin{tikzcd}[labels=description,row sep=1cm, column sep=1cm] \cA_s(K)[\scU_2,\scV_2] \ar[d, "\scV_1"]\ar[r, "\scU_2"]& \cA_s(K)[\scU_2,\scV_2]\\
\cA_{s+1}(K)[\scU_2,\scV_2]& \cA_{s+1}(K)[\scU_2,\scV_2] \ar[u, "\scU_1"] \ar[l, "\scV_2"]
\end{tikzcd}
\label{eq:large-surgery}
\end{equation}

We will write $\cA_{s}(L)$ for the right-hand side of the above equation, which we may identify with the subcomplex of $\cCFL(K\# H)$ generated over $\bF$ by $x$ such that $A_1(x)=s+1/2$, where $A_1$ denotes the component of the Alexander grading associated to $K$.

\begin{rem}
Throughout the paper, we focus on the positive Hopf link $\cH^+$. The same analysis may be performed using the negative Hopf link. Note that we recover a slightly different version of the dual knot formula than Hedden and Levine \cite{HeddenLevineSurgery}. To recover their version, one should use the negative Hopf link complex.
\end{rem}

\subsection{ An expanded model of the dual knot mapping cone formula}

In this section, we obtain an expanded model for the dual knot formula, by taking the tensor product with a Hopf link, and using a variation of Ozsv\'{a}th and Szab\'{o}'s  integer surgery formula \cite{OSIntegerSurgeries}.

For integers $s\in \Z$, we consider complexes $\cB_{s}(L)$ and $\tilde{\cB}_{s}(L)$. We define $\cB_{s}(L)$ to be the subcomplex of $\scV_1^{-1} \cCFL(K\# H)$  generated over $\bF$ by elements $x\in \scV_1^{-1}\cCFL(K\# H)$ such that $A_1(x)=s+1/2$. Here, $A_1$ denotes the component of the Alexander grading associated to $K$. Similarly, we define $\tilde{\cB}_{s}(L)$ to be the subcomplex of $\scU_1^{-1} \cCFL(K\# H)$ in the same Alexander grading.

  The complex $\cB_{s}(L)$ is conveniently described by:

\begin{equation}
\cB_{s}(L)\simeq \begin{tikzcd}[labels=description,row sep=1cm, column sep=1cm] \cB_s(K)[\scU_2,\scV_2] \ar[d, "\scV_1"]\ar[r, "\scU_2"]& \cB_s(K)[\scU_2,\scV_2]\\
\cB_{s+1}(K)[\scU_2,\scV_2]& \cB_{s+1}(K)[\scU_2,\scV_2] \ar[u, "\scU_1"] \ar[l, "\scV_2"]
\end{tikzcd}
\label{eq:large-surgery-B-s}
\end{equation}
The complex $\tilde{\cB}_{s}(L)$ admits a similar description.

There are inclusions $v\colon \cA_{s}(L)\to \cB_{s}(L)$ and $\vopp\colon \cA_{s}(L)\to \tilde{\cB}_{s}(L)$.  There is furthermore a flip map $\frF_n\colon \tilde{\cB}_{s}(L)\to \cB_{s+n}(L)$ constructed by forgetting about the $w_1$-basepoints, moving $z_1$ to $w_1$, and then adding $z_1$ back. (This construction is essentially the same as the Ozsv\'{a}th and Szab\'{o}'s definition of the map $h$ in the original mapping cone formula \cite{OSIntegerSurgeries}).

We now discuss completions. Suppose that $K\subset Y$ is a knot.  We write $\ve{\cCFK}(Y,K)$ for the free $\bF[[\scU,\scV]]$-module spanned by intersection points of a Heegaard diagram. In our present context, we also wish to consider completed versions of $\cA_{s}(L)$ and $\cB_{s}(L)$. We recall that these were free modules over $\bF[U,\scU_2,\scV_2]$. We define $\ve{\cA}_{s}(L)$ and $\ve{\cB}_{s}(L)$ to be the free $\bF[[U,\scU_2,\scV_2]]$-modules with the same generators. We will also consider
\[
\prod_{s\in \Z} \ve{\cA}_{s}(L)\quad \text{and} \quad \prod_{s\in \Z} \ve{\cB}_{s}(L).
\]

The main result of this section is the following:

\begin{thm}\label{thm:expanded-model} There is a relatively graded homotopy equivalence of chain complexes over the ring $\bF[[U,\scU_2,\scV_2]]$
\begin{equation}\label{eq:expanded-model}
\ve{\cCFK}(S^3_n(K),\mu)\simeq \Cone\left(v+\frF_n \vopp\colon  \prod_{s\in \Z}\ve{\cA}_{s}(L)\to  \prod_{s\in \Z} \ve{\cB}_{s}(L)\right).
\end{equation}
The maps $v$ and $\tilde{v}$ are the canonical inclusions, and $\frF_n\colon \tilde{\cB}_{s}(L) \to \cB_{[s+n]}(L)$ is a homotopy equivalence of $\bF[U,\scU_2,\scV_2]$-modules.
\end{thm}

We call the right-hand side of Equation~\eqref{eq:expanded-model} the \emph{expanded dual knot mapping cone formula for} $\ve{\cCFK}(S_n^3(K),\mu)$. The proof of Theorem~\ref{thm:expanded-model} follows from a modification of Ozsv\'{a}th and Szab\'{o}'s knot surgery formula \cite{OSIntegerSurgeries}, as we now sketch.

If $m\ge 1$ is an integer, there is a  twisted complex $\ul{\cCFL}(K\# H)$ which is freely generated over $\bF[U,T,\scU_2,\scV_2]/(1-T^m)$ by tuples $\xs\cdot  U^i T^j \scU_2^n \scV_2^m$, where $\xs\in \bT_{\a}\cap \bT_{\b}$, with $i,n,m\ge 0$ and $j\in \Z$. The differential counts holomorphic disks weighted by $U^{n_{w_1}(\phi)}T^{n_{z_1}(\phi)-n_{w_1}(\phi)}\scU_2^{n_{w_2}(\phi)}\scV_2^{n_{z_2}(\phi)}$. The proof of the surgery exact sequence of \cite{OSIntegerSurgeries}*{Theorem~3.1} gives the following homotopy equivalence of chain complexes over $\bF[[U,\scU_2,\scV_2]]$:
\begin{equation}
\ve{\cCFK}(S_n^3(K),\mu)\simeq \Cone\left(F_{W'_{n+m},S_\mu}\colon  \ve{\cCFK}(S_{n+m}^3(K),\mu)\to \underline{\ve{\cCFL}}(K\# H)\right)
\label{eq:mapping-cone-isomorphism}
\end{equation}
In the above, $W'_{n+m}$ denotes the 2-handle cobordism from $S^3_{n+m}(K)$ to $S^3$. The map $F_{W'_{n+m},S_\mu}$ in Equation~\eqref{eq:mapping-cone-isomorphism} is the link cobordism map for the natural link cobordism from $(S_{n+m}^3(K),\mu)$ to $(S^3,K\# H)$, summed over all $\Spin^c$ structures. Note that $\underline{\ve{\cCFL}}(K\# H)$ is chain isomorphic to $  \ve{\cCFK}(\bU,p)\otimes \bF[\Z/m]$, where $\bU$ denotes an unknot in $S^3$ (corresponding to the component $\mu$). Compare \cite{OSIntegerSurgeries}*{Equation~(7)}.

If $C$ is a chain complex over $\bF[U,\scU_2,\scV_2]$, and $\delta > 0$, we define
\[
C^\delta=C/(U^\delta, \scU_2^\delta,\scV_2^\delta).
\]

\begin{lem}\label{lem:vanishing-terms}
 Suppose $\delta>0$. If $m$ is suitably large relative to $\delta$, then the natural cobordism map 
\[
F_{W'_{n+m},S_\mu,\frs}^\delta\colon \cCFK^\delta(S_{n+m}^3(K),\mu)\to \cCFK^\delta(\bU,p)
\]
 is trivial unless $\frs$ is one of $\frx_s$ or $\fry_s$, for $-(n+m)/2\le s\le (n+m)/2$. Here, $\frx_s$ is the $\Spin^c$ structure defined in Equation~\eqref{eq:x_s-def}, and $\fry_s=\frx_{s+m+n}$.
\end{lem}
\begin{proof} Since $\fry_s=\frx_{s+m+n}$, it is sufficient to show that the $\delta$-truncated cobordism map is only non-trivial on $\frx_s$ for $-(n+m)/2\le s\le  3(n+m)/2$. Among $\Spin^c$ structures on $W_{n+m}'(K)$ which restrict to $[s]\in \Spin^c(S^3_{n+m}(K))$, the $\Spin^c$ structures  $\frx_s$ and $\fry_s$ have Chern classes with maximal and next to maximal square. See \cite{OSIntegerSurgeries}*{Lemma~4.4}. We consider the $\gr_w$ and Alexander grading changes of the cobordism map for the $\Spin^c$ structure $\frx_s$, for an arbitrary $s\in \Z$. We focus on the case that $n>0$. The $\gr_w$-grading change of the cobordism in $\Spin^c$-structure $\frx_s$ is given by the formula
\begin{equation}
\gr_w(\frx_s)=\frac{c_1(\frx_s)^2+1}{4}=\frac{-(n+m-2s)^2/(n+m)+1}{4}. \label{eq:gr_w-reduced-cobordism}
\end{equation}

The Alexander grading change is given by
\begin{equation}
A(\frx_s)=\frac{\langle c_1(\frx_s),[S_\mu]\rangle-S_\mu^2}{2}=-\frac{n+m-2s-1}{2(n+m)}. \label{eq:A-F-W-mu}
\end{equation}

Suppose that $\frx_s$ is the $\Spin^c$ structure of maximal square for a given class $[s]\in \Spin^c(S^3_{n+m}(K))$, where $-(n+m)/2\le s\le 3(n+m)/2$.
 The third to maximal $\Spin^c$ structure restricting to $[s]$ will be one of $\frx_{s\pm (n+m)}$. Using Equation~\eqref{eq:gr_w-reduced-cobordism}, we compute that
 \begin{equation}
\begin{split}\gr_w(\frx_s)-\gr_w(\frx_{s\pm(n+m)})&=-\frac{(n+m-2s)^2-(n+m-2s\mp 2(n+m))^2}{4(n+m)}\\
&=\mp (n+m-2s)+(n+m).
 \end{split}
 \label{eq:difference-gradings}
 \end{equation}
If $2s<n+m$, then $\frx_{s+n+m}$ is the second to maximal square, and $\frx_{s-(n+m)}$ is the third to maximal square. If $2s>n+m$, then $\frx_{s-(n+m)}$ is the second to maximal square and $\frx_{s+n+m}$ is the third to maximal. In either case, we see that Equation~\eqref{eq:difference-gradings} is at least $n+m$. 

On the other hand, Equation~\eqref{eq:A-F-W-mu} implies that the difference in the Alexander grading  between the maximal and third to maximal $\Spin^c$ structures is always $\pm 1$. Since $A=\tfrac{1}{2}(\gr_w-\gr_z)$, we conclude that $|\gr_w(\frx_s)-\gr_z(\frx_s)|\le 2$ for all $s$. A generalization of the truncation arguments of \cite{OSIntegerSurgeries} and \cite{HeddenLevineSurgery} shows that if $m$ is sufficiently large, then the image of the third to maximal $\Spin^c$ structure must lie in the submodule generated by the ideal $(U^\delta, \scU_2^\delta,\scV_2^\delta)\subset \bF[U,\scU_2,\scV_2]$. An extension of this argument then shows that all of the $\Spin^c$ structures with even lower square also have image in this submodule.
\end{proof}

Furthermore, we have the following:
\begin{lem}\label{lem:vanishing-maps-in-cone} If $\delta> 0$ is fixed, then there is an integer $b>0$ such that for all sufficiently large $m$, the map
\[
F_{W'_{n+m},S_\mu,\frx_s}^\delta\colon \cCFK^\delta(S^3_{n+m}(K),\mu,[s])\to \cCFK^\delta(\bU,p)
\]
is 0 if $s>b$ and is a homotopy equivalence if $s<-b$. Similarly the map $F_{W'_{n+m},S_\mu,\fry_s}^\delta$ vanishes if $s<-b$ and is a homotopy equivalence if $s>b$.
\end{lem}

In Lemma~\ref{lem:vanishing-maps-in-cone}, the statement about the vanishing of certain cobordism maps follows from a grading argument similar to Lemma~\ref{lem:vanishing-terms}. The fact that the map $F_{W'_{n+m},S_\mu,\frx_s}^\delta$ is a homotopy equivalence for $s\gg 0$ follows from the large surgeries formula in Proposition~\ref{prop:large-surgery} and the fact that the inclusion $\ve{\cA}_{s}(L)\to \ve{\cB}_{s}(L)$ is a chain isomorphism for $s\gg 0$. Compare the commutative diagrams in \cite{OSIntegerSurgeries}*{Theorem~2.3}. The statement about $F_{W'_{n+m},S_\mu,\fry_s}^\delta$ is similar.

By Proposition~\ref{prop:large-surgery}, if $m$ is sufficiently large we may identify each $\ve{\cCFK}(S^3_{n+m}(K),\mu,[s])$ with the subcomplex $\ve{\cA}_{s}(L)$ of $\ve{\cCFL}(K\# H)$. Similarly, $\underline{\ve{\cCFL}}(K\# H)$ may be identified with $\ve{\cCFL}(K\# H)\otimes \bF[\Z/m]$, or as the direct sum of $\ve{\cB}_{s}(L)$ for $-m/2<s\le m/2$. We write
\begin{equation}
\theta_w\colon \underline{\ve{\cCFL}}(K\# H)\to \bigoplus_{-m/2<s\le m/2} \ve{\cB}_{s}(L)
\label{eq:def-theta-w}
\end{equation}
for the trivializing map. The map $\theta_w$ may be defined by the formula
\begin{equation}
\theta_w(\xs\cdot U^i T^j \scU_2^p\scV_2^q)=\xs\cdot \scU_1^i \scV_1^{j+i+N} \scU_2^p \scV_2^q\label{eq:def-thetw-w-map}
\end{equation}
where $N\in m\cdot\Z$ is the unique number such that $-m/2<A_1(\xs)+1/2+i+j+N \le m/2$. 

 The same argument as in the mapping cone formula identifies $F_{W'_{n+m},S_\mu,\frx_s}$ with the inclusion of $\ve{\cA}_{s}(L)$ into $\ve{\cB}_{s}(L)$, and also identifies $F_{W'_{n+m},S_\mu,\fry_s}$ as the inclusion of $\ve{\cA}_{s}(L)$ into $\tilde{\ve{\cB}}_{s}(L)$ followed by a homotopy equivalence of $\tilde{\ve{\cB}}_{s}(L)$ with $\ve{\cB}_{s}(L)$. Compare \cite{OSIntegerSurgeries}*{Theorem~2.3}.

Using the above, it is straightforward to adapt the proof in the case of the ordinary integer surgery formula for knots to obtain that
\[
\cCFK^\delta(S_n^3(K),\mu)\simeq \Cone\left(v+\frF_n \vopp\colon  \bigoplus_{-b\le s\le b}\cA_{s}^\delta(L)\to  \bigoplus_{-b+n\le s\le b} \cB_{s}^\delta(L)\right),
\]
for all $\delta$ and all sufficiently large $b$.
From here a straightforward algebraic argument proves the non-truncated version in the statement of Theorem~\ref{thm:expanded-model}.

\begin{rem}
 If $\Lambda$ is an integral framing on the 2-component link $K\# H$, Manolescu and Ozsv\'{a}th \cite{MOIntegerSurgery} construct a 2-dimensional hypercube $\cC_{\Lambda}(K\# H)$ which computes the Heegaard Floer homology of $S^3_\Lambda(K\# H)$. Here, $\Lambda$ denotes an integral framing on the 2-component link $K\# H$. Another interpretation of the mapping cone complex in Theorem~\ref{thm:expanded-model} is as a codimension 1 subcube of $\cC_{\Lambda}(K\# H)$. Each of the 2 coordinate directions of $\cC_{\Lambda}(K\# H)$ corresponds to a component of $K\# H$. The cone in Theorem~\ref{thm:expanded-model} corresponds to the face of $\cC_{\Lambda}(K\# H)$ where the $\mu$ coordinate is 0. The completions described above for the mapping cone are the same as the completions used in the Manolescu--Ozsv\'{a}th link surgery formula.
\end{rem}

\subsection{Gradings}
In this section, we compute the absolute gradings on the expanded dual knot mapping cone. Recall that the quasi-isomorphism between the $\delta$-truncated complex $\cCFK^\delta(S^3_n(K),p,\mu)$ and the truncated mapping cone takes the following form:
\begin{equation}
\begin{tikzcd} \cCFK^\delta(S^3_n(K),p, \mu)\ar[d, "F"] \ar[dr,dashed, "J"]\\
\cCFK^\delta(S^3_{n+m}(K),p,\mu)
	\ar[r]
	\ar[d, "\Gamma"]
	\ar[dr,dashed]
 &\underline{\cCFL}^\delta(S^3, K\cup \mu) \ar[d, "\theta_w"]\\
\bigoplus_{-b\le s\le b} \cA_{s}^{\delta}(K\cup \mu)\ar[r]& \bigoplus_{-b+n\le s\le b} \cB^{\delta}_{s}(K\cup \mu)
\end{tikzcd}
\label{eq:quasi-isomorphism-as-diagram}
\end{equation}
In the above, the map $\Gamma$ is the composition of a large surgeries map with projection onto the $s$ summands for $-b\leq s \leq b$, where $b$ is such that $m\gg b\gg 0$. Similarly $\theta_w$ is the composition of the trivialization map from Equation~\eqref{eq:def-theta-w} with a projection map onto the $s$ summands with $-b+n\leq s\leq b$.

To understand the absolute gradings appearing in the mapping cone formula, it is easiest to understand the dashed diagonal map $J$. Understanding the grading change induced by $F$ is a straightforward adaptation and left to the reader. We see that the map $J$ counts holomorphic rectangles and is naturally associated to the cobordism from $(S_n(K),\mu)\sqcup L(m,1)$ to $(S^3,K\cup \mu)$. If we let $D(-m,1)$ denote the disk bundle of Euler number $-m$ over the 2-sphere and write $S_m\subset D(-m,1)$ for the 2-sphere  with self-intersection $-m$, then this cobordism may be viewed as the connected sum of $(W_n',S_K\cup S_\mu)$ with the link cobordism $(D(-m,1),S_m)$, where we take the connected sums both of the four-manifolds $W_n'\# D(-m,1)$ and also of the surfaces $S_K\# S_m$. 

We first observe that the complex $\bA^\delta(L)\langle b\rangle:=\bigoplus_{-b\le s\le b} \cA^\delta_{s}(K\cup \mu)$ and the complex $\bB^\delta(L)\langle b\rangle:=\bigoplus_{-b+n\le s\le b} \cB^\delta_{s}(K\cup \mu)$ appearing in Equation~\eqref{eq:quasi-isomorphism-as-diagram} are finitely generated over $\bF$. In particular, only finitely many $\gr_w$ and $\gr_z$ gradings are represented.

Assume that $\delta$ and $b$ have been fixed, and let $m$ be large. If $t\in 2\Z+1$, we write $\fru_t\in \Spin^c(D(-m,1))$ for the $\Spin^c$ structure which satisfies
\[
\langle c_1(\fru_{t}),S_m\rangle=t\cdot  m.
\]
 All $\Spin^c$ structures on $W_{n}'\# D(-m,1)$ may then be written as $\frx_s\# \fru_t$, for some $t$.

To compute the absolute gradings on the mapping cone, it suffices to compute the grading change of the map $\theta_w\circ J$. Note that there is also a map 
\[
H=\cCFK^\delta(S_n^3(K),p,\mu)\to \cCFL^\delta(S^3,K\cup \mu),
\]
obtained by counting holomorphic quadrilaterals on the same diagrams as $J$. The grading change of the map $H$ may be computed by using the standard grading change formulas for the link cobordism 
\[
(W_n' \# D(-m,1),S_K\# S_m\cup S_\mu)\colon (S_n^3(K),p,\mu)\to (S^3,K\cup \mu).
\]

 Both $H$ and $\theta_w\circ J$ decompose over $\Spin^c$ structures $\frw\in \Spin^c(W'_{n}\# D(-m,1))$. If $H_\frw$ and $J_{\frw}$ are the summands corresponding to the same $\Spin^c$ structure, then
 \begin{equation}
 \scV_1^{x\cdot m} H_{\frw}=(\theta_w \circ J_{\frw})\label{eq:H-J-relation}
 \end{equation}
 for some $x\in \Z$. In particular, $H_{\frw}$ and $J_{\frw}$ will have the same expected $\gr_w$-grading change. We may compute
 \begin{equation}
 \gr_{w}(J_{\frx_s\# \fru_{t}})=\frac{c_1(\frx_s\# \fru_{t})^2+m+1}{4}+1=\frac{c_1(\frx_s)^2-m(t^2-1)+1}{4}+1.
 \label{eq:gr_w-J-map}
 \end{equation}
 In the above formula, the summand $m/4$ on the left is contributed by the grading of the canonical generator of $L(-m,1)$, which is an input of both $J$ and $H$. Since $\bA^\delta(L)\langle b\rangle$ and $\bB^{\delta}(L)\langle b\rangle$ have only finitely many $\gr_w$-gradings supported, the only $\Spin^c$ structures which can have non-trivial $J_{\frx_s\# \fru_t}$  are of the form $\frx_s\# \fru_{\pm 1}$.


We now compute the $\gr_z$-grading changes. By adapting Lemma~\ref{lem:grading-change}, one computes easily that 
\[
A_1(\frx_s\# \fru_t)=s+\frac{1}{2}+\frac{(t+1)m}{2}.
\]
In particular, $A_1(\frx_s\# \fru_{-1})=s+1/2$ and $A_1(\frx_s\# \fru_{1})=s+1/2+m$. Hence, using the definition of $\theta_w$ in Equation~\eqref{eq:def-thetw-w-map} we see that  Equation~\eqref{eq:H-J-relation} takes the following form in our present situation:
\begin{equation}
 H_{\frx_s\# \fru_{-1}}=\theta_w\circ J_{\frx_s\# \fru_{-1}}\quad \text{and} \quad \scV^{-m}_1 H_{\frx_s\# \fru_{1}}=\theta_w\circ J_{\frx_s\# \fru_{1}}.
 \label{eq:H-J-relation-2}
\end{equation}

\noindent Moreover, we may compute that
\[
\begin{split}
\gr_{z}(H_{\frx_s\# \fru_{t}})=&\frac{c_1(\frx_s\# \fru_{t}-\PD[S_K\# S_m+S_\mu])^2+m+1}{4}+1\\
=&\frac{c_1(\frx_s-\PD[S_K+S_\mu])^2+c_1(\fru_{t+2})^2+m+1}{4}+1\\
=&\frac{c_1(\frx_s-\PD[S_K+S_\mu])^2-((t+2)^2-1)m+1}{4}+1
\end{split}
\]
Using Equation~\eqref{eq:H-J-relation-2}, we see that 
\begin{equation}
\gr_z(\theta_w\circ J_{\frx_s\# \fru_{\pm 1}})=\frac{c_1(\frx_s-\PD[S_K+S_\mu])^2+1}{4}+1.
\label{eq:gr_z-J-map}
\end{equation}
Combining Equations~\eqref{eq:gr_w-J-map} and ~\eqref{eq:gr_z-J-map} with  Equation~\eqref{eq:c1^2-x-s} and its analog for $\frx_s-\PD[S_K+S_\mu]$, we compute
\begin{equation}
\gr_w(\theta_w\circ J_{\frx_s\# \fru_{\pm 1}})=\frac{-(n-2s)^2/n+1}{4}\quad \text{and} \quad \gr_z(\theta_w\circ J_{\frx_s\# \fru_{\pm 1}})= \frac{-(n+2(s+1))^2/n+1}{4}.
\label{eq:grading-changes}
\end{equation}
The case when $n$ is negative is similar, except that the $1$ is replaced by $-5$, due to the signature of the cobordism changing sign.

\begin{rem}
Equation~\eqref{eq:grading-changes} differs slightly from the grading shift in Equation~\eqref{eq:gr_w-reduced-cobordism}. The grading shifts for $\gr_w$ coincide, but the shifts for $\gr_z$ and $A$ differ. This is due to the fact that in Equation~\eqref{eq:gr_w-reduced-cobordism}, we computed the grading change for the link cobordism which has surface $S_\mu$, whereas for Equation~\eqref{eq:grading-changes}, we compute the shift for the link cobordism with surface $S_K\cup S_\mu$.
\end{rem}

Finally, we note that in Theorem~\ref{thm:expanded-model}, the right-hand side  naturally computes the stabilized complex $\ve{\cCFK}(S^3_n(K),\mu,p)$, where $p$ is an extra basepoint. Stabilizing does not change the homotopy type over $\bF[\scU_2,\scV_2]$, however there is a grading shift (cf. Equation~\eqref{eq:stabilize-add-basepoint}). Indeed we have
\[
\ve{\cCFK}(S^3_n(K),\mu,p)\simeq \ve{\cCFK}(S^3_n(K),\mu)[\tfrac{1}{2},\tfrac{1}{2}].
\]
Here, if $M$ is an $\bF$-module, we write $M[n,m]$ for $M\otimes \bF_{(n,m)}$, where $\bF_{(n,m)}$ is the rank 1 vector space concentrated in $(\gr_w,\gr_z)$-bigrading $(n,m)$.

With respect to these considerations, the modules in Theorem~\ref{thm:expanded-model} have the following graded refinements:
\begin{equation}
\begin{split}
\bA(L)&=\prod_{s\in \Z} \ve{\cA}_s(L)\left[\frac{(n-2s)^2/n+\epsilon(n)}{4}-\frac{1}{2}, \frac{(n+2(s-1))^2/n+\epsilon(n)}{4}-\frac{1}{2}\right]\\\bB(L)&=\prod_{s\in \Z} \ve{\cB}_s(L)\left[\frac{(n-2s)^2/n+\epsilon(n)}{4}-\frac{3}{2}, \frac{(n+2(s-1))^2/n+\epsilon(n)}{4}-\frac{3}{2}\right]
\end{split}
\label{eq:gradings}
\end{equation}
In the above, 
\begin{equation}
\epsilon(n)=\begin{cases}-1& \text{if } n>0\\
5& \text{if } n<0.
\end{cases}
\label{eq:grading-epsilon-def}
\end{equation}
The term $\epsilon(n)$ corresponds to $-3\sigma-2\chi$ in the grading change formula.

\section{The small model of the mapping cone formula} \label{sec:small-model}

In this section, we describe models of $\bA(L)$ and $\bB(L)$ which are finitely-generated over $\bF[\scU_2,\scV_2]$, using the homological perturbation lemma for hypercubes of Section~\ref{sec:homological-perturbation-hypercubes}. The work in this section gives an alternate proof of Hedden and Levine's dual knot mapping cone formula \cite{HeddenLevineSurgery}. In the subsequent section, we will also push the involution to the small model.

\subsection{The small model for $\bA(L)$}\label{subsec:homperturbAs}

We now apply Lemma~\ref{lem:homological-perturbation-cubes} to $\bA(L)$. In the following we adopt the algebraic perspective of Lipshitz, Ozsv\'{a}th and Thurston and consider $A_\infty$-bimodules, as well as type-$DD$ bimodules, and the appropriate notions of box tensor products between these objects. The unfamiliar reader may consult \cite{LOTBordered}*{Chapter 2} and \cite{LOTBimodules}*{Section~2} for the necessary algebraic background.

 Let $M_s$ be the free $\bF[\scU_1,\scV_1]$ chain-complex given by the following diagram:
\begin{equation}
M_s:=\begin{tikzcd}[labels=description,row sep=1cm, column sep=1cm] \cA_s(K) \ar[d, "\scV_1"]& \cA_s(K)\\
\cA_{s+1}(K)& \cA_{s+1}(K) \ar[u, "\scU_1"]
\end{tikzcd}\label{eq:module-M-s}
\end{equation}
We can view $M_s$ as an $A_\infty$-module (actually, a dg-module) over the exterior algebra $\Lambda:=\Lambda^*(\theta_\scU,\theta_{\scV})$, where $1$ acts by the identity, and $\theta_{\scU}$ and $\theta_{\scV}$ have the following actions:
\[
\begin{tikzcd}[labels=description,row sep=.75cm, column sep=1cm] \cA_s(K)\ar[r, "\theta_{\scU}"] & \cA_s(K)\\
\cA_{s+1}(K) & \cA_{s+1}(K)\ar[l, "\theta_{\scV}"] 
\end{tikzcd}
\]
Here, the action of $\theta_{\scU}\theta_{\scV}$ vanishes. The map $\theta_{\scU}$ maps the left copy of $\cA_s(K)$ to the right, via the identity, and similarly $\theta_{\scV}$ maps the right copy of $\cA_{s+1}(K)$ to the left copy, via the identity. Letting $\cP$ be the polynomial ring $\cP=\bF[\scU_2,\scV_2]$, the complex in equation~\eqref{eq:large-surgery} can be understood as the box tensor product
\[
\cA_{s}(L)={}_{\cP}\cP_{\cP}\boxtimes {}^{\cP}\cK{}^{\Lambda}\boxtimes {}_{\Lambda} M_s,
\]
where $\cK$ is the Koszul dualizing $\mathit{DD}$-bimodule, which is has underlying group $\bF$, and structure map
\[
\delta^{1,1}(1)=\scU_2\otimes 1\otimes \theta_{\scU}+\scV_2\otimes 1\otimes  \theta_{\scV}.
\]
By Lemma~\ref{lem:snake}, $M_s$ is homotopy equivalent as a chain complex to
\[
Z_s:= (\cA_{s+1}(K)/\scV_1 \cA_s(K))\oplus (\cA_{s}(K)/\scU_1 \cA_{s+1}(K)),
\]
with quotient complex differential. The homological perturbation lemma for $A_\infty$-modules induces an $A_\infty$-module action of $\Lambda$ on $Z_s$, as well as $A_\infty$-homotopy equivalences between ${}_{\Lambda} M_s$ and ${}_{\Lambda} Z_s$.

To understand the $A_\infty$-module morphisms between ${}_{\Lambda} M_s$ and ${}_{\Lambda} Z_s$, we consider first the maps yielding the homotopy equivalence of $M_s$ and $Z_s$ as chain complexes. The map $\Pi\colon M_s\to Z_s$ is the canonical projection map. The inclusion map $I$ in the other direction is obtained by picking a splitting (of $\bF$-modules) of the short exact sequences
\[
\begin{tikzcd}[labels=description]
0\ar[r] &\cA_s(K) \ar[r, "\scV_1"]& \cA_{s+1}(K) \ar[l, "\sigma", bend right, dashed,swap] \ar[r] & (\cA_{s+1}/ \scV_1 \cA_s)(K)\ar[r] \ar[l, "s", bend right, dashed,swap]& 0
\end{tikzcd}
\]
and
\[
\begin{tikzcd}[labels=description]
0\ar[r] &\cA_{s+1}(K) \ar[r,"\scU_1"]& \cA_{s}(K) \ar[r] \ar[l, "\sigma'", bend right, dashed, swap] & (\cA_{s}/ \scU_1 \cA_{s+1})(K)\ar[r] \ar[l, "s'", bend right, dashed,swap]& 0.
\end{tikzcd}
\]
The splittings $s$ and $s'$ are of $\bF$-vector space homomorphisms. We may now use Lemma~\ref{lem:snake} to induce the map $F$ from part (5) of the lemma; this is the inclusion map $I$. Note that $s$ and $s'$ will usually be neither chain maps nor $\bF[\scU_1,\scV_1]$-equivariant. Moreover, the splittings $s$ and $s'$ are canonically determined from a choice of basis of $\cCFK(K)$: if $y=[\scU_1^i\scV_1^{j}\cdot \xs]\in (\cA_{s+1}/\scV_1 \cA_s)(K)$, then $s(y)= \scU_1^{i} \scV_1^j\cdot \xs$ if $j=0$, and $s(y)=0$ otherwise. The splitting $s'$ is defined similarly. This induces similar splittings $\sigma$ and $\sigma'$. Similarly to $s$ and $s'$, the maps $\sigma$ and $\sigma'$ will usually be neither $\bF[\scU_1,\scV_1]$-equivariant nor chain maps. Concretely, the map $\sigma$ is given by $ \scU_1^i \scV_1^j\cdot \xs\mapsto \scU_1^i \scV_1^{j-1}\cdot \xs$ if $j>0$, and $ \scU_1^i\scV_1^0\cdot \xs\mapsto 0$. The homotopy $H\colon M_s\to M_s$ is $\sigma\oplus \sigma'$. 

Note that clearly $[\d, \Pi]=0$ and $[\d, I]=0$. Furthermore, it is straightforward to check that
\[
\Pi\circ I=\id_Z,\quad I\circ \Pi=\id+[\d, H],\quad H\circ H=0, \quad \Pi\circ H=0,\quad H\circ I=0.
\]

With the above $\Pi$, $I$ and $H$ chosen, the homological perturbation lemma, stated in Lemma~\ref{lem:homological-perturbation-modules}, induces a $A_\infty$-module structure on $Z_s$ over $\Lambda$. This action is indicated schematically in Figure~\ref{fig:Z-structure-rels}.

We now set
\[
\cA_s^\mu(K):={}_{\cP} \cP_{\cP} \boxtimes {}^{\cP} \cK^{\Lambda}\boxtimes {}_{\Lambda} Z_s.
\]
By construction, $\cA_s^\mu(K)\simeq \cA_{s}(L)$ as chain complexes over $\bF[\scU_2,\scV_2]$.

We now endeavor to describe the differential on ${}_{\cP} \cP_{\cP}\boxtimes {}^{\cP} \cK^{\Lambda}\boxtimes {}_{\Lambda} Z_s$ in concrete terms. Note the only terms $a_n \otimes  \cdots  \otimes a_1 \otimes x$ which have non-zero action on ${}_{\Lambda} Z_s$ are of the form $\theta_{\scU} \otimes \theta_{\scV} \otimes \theta_{\scU} \otimes \cdots  \otimes  x$ or $\theta_{\scV} \otimes \theta_{\scU} \otimes \theta_{\scV} \otimes \cdots \otimes  x$; that is, each algebra element consists of a single $\theta$ factor, which alternate between $\theta_{\scU}$ and $\theta_{\scV}$.  

The differential $\d^\mu$ has contributions from different configurations of maps. In general, the  differential is a sum over $n\in \N$ of applying the map $\delta^{1,1}$ $n$ times, and then doing the following to the resulting algebra elements: The $\Lambda$ outputs are input into ${}_{\Lambda} Z_s$, while the $\cP$ outputs are multiplied by applying $\mu_2$ repeatedly.

The first such contribution occurs when $\delta^{1,1}$ is applied 0 times, i.e. we just apply $m_1^Z$. If $x=\scU^i\ve{x}$, this contributes to $\d^\mu([x])$ the summands of $\d(x)$ which have no $\scV_1$-powers.

Next, we have the summands which involve at least one application of $\delta^{1,1}$. The left-$\cP$ outputs are multiplied, while the right-$\Lambda$ outputs are input into $m_j^Z$. To illustrate, we start at an element $x=\scU_1^i\ve{x}\in (\cA_{s+1}/\scV_1 \cA_s)(K)[\scU_2,\scV_2]$. The inclusion map $I$ sends $x$ into $\Cone(\scV_1 \colon \cA_s(K)\to \cA_{s+1}(K))$. The term $I(x)$ has a summand in the domain copy of $\cA_s(K)$ and a summand in the codomain copy of $\cA_{s+1}(K)$ of the mapping cone. The summand in the codomain is just $x=\scU^i\ve{x}$. The map $H$ vanishes on this element, so this term does not contribute. The summand of $I(x)$ in $\cA_s(K)$ is obtained by applying $\sigma$ to $\d(\scU^i\ve{x})$. The map $\sigma$ simply lowers the $\scV_1$-power by 1 (and annihilates summands with no $\scV_1$-power). We subsequently apply $m_2^M(\theta,-)$ where $\theta\in \{\theta_\scU,\theta_\scV\}$, then we apply $H$, then $m_2^M(\theta,-)$, and so forth. We finish with an application of $m_2^M(\theta,-)$ followed by $\Pi$. Note that $\Pi$ is only non-vanishing on the kernel of $H$, so  the above procedure stops when we run out of either powers of $\scU_1$ or $\scV_1$. We summarize with the following lemma:

\begin{lem}\label{lem:homological-perturbation-differential-A} The differential on $\cP_{\cP}\boxtimes {}^{\cP}\cK^{\Lambda}\boxtimes {}_{\Lambda}Z_s$, which we denote by $\d^\mu$, takes the following form:
\begin{enumerate}[label=($d$-\arabic*), ref=($d$-\arabic*)]
\item\label{small-del-1} Suppose $[x]\in (\cA_{s+1}/\scV_1 \cA_s)(K)$. Suppose that $\scU_1^{m} \scV_1^{n} \ve{y}$ is a summand of $\d x$. 
\begin{enumerate}
\item If $n>m$, then $\d^\mu([ x])$ has a summand of \[\scU_2^{m+1}\scV_2^m (\scV_1^{n-m-1}\cdot \ys)\in (\cA_s/\scU_1 \cA_{s+1})[\scU_2,\scV_2].\] 
\item If $n\le m$, then $\d^\mu([x])]$ has a summand of \[\scU_2^n \scV_2^n (\scU_1^{m-n} \cdot \ys)\in (\cA_{s+1}/\scV_1 \cA_s)[\scU_2,\scV_2].\]
\end{enumerate}
\item\label{small-del-2} Suppose $[x]\in (\cA_{s}/\scU_1 \cA_{s+1})(K)$, and suppose that $\scU_1^m \scV_1^n \cdot \ys$ is a summand of $\d x$ in $\cCFK(K)$.
\begin{enumerate}
\item If $m\le n$, then $\d^\mu([x])$ has a summand of \[\scU_2^m \scV_2^{m}( \scV_1^{n-m}\cdot  \ve{y})\in (\cA_s/\scU_1 \cA_{s+1})[\scU_2,\scV_2].\] 
\item If $m>n$, then $\d^\mu([x])$ has a summand of \[\scU_2^{n} \scV_2^{n+1}(\scU_1^{m-n-1} \cdot \ys)\in (\cA_{s+1}/\scV_1 \cA_{s})[\scU_2,\scV_2].\]
\end{enumerate}
\end{enumerate}
\end{lem}

\subsection{Applying the homological perturbation to $\cB_{s}(L)$}

The homological perturbation argument can also be applied to $\cB_{s}(L)$ and $\tilde{\cB}_{s}(L)$. However, it is important to note that these have a slightly different description, since $\scV_1\colon \cB_{s+1}(L)\to \cB_{s}(L)$ is an isomorphism. Similarly $\scU_1\colon \tilde{\cB}_{s+1}(L)\to \tilde{\cB}_{s}(L)$ is an isomorphism. Hence the small model for $\cB_{s}(L)$ has only one summand, and similarly for $\tilde{\cB}_{s}(L)$. We summarize this for $\cB_{s}(L)$:

\begin{lem}\label{lem:homperturbBs}
 Applying the homological perturbation lemma to $\cB_{s}(L)$, as we did to $\cA_{s}(L)$ in Lemma~\ref{lem:homological-perturbation-differential-A}, gives the following model $\cB^\mu_s(K)$ for $\cB_{s}(L)$ as a chain complex over $\bF[\scU_2,\scV_2]$:
 \begin{enumerate}
 \item As a group, $\cB_{s}^\mu(K)\iso (\cB_s/\scU_1 \cB_{s+1})(K)[\scU_2,\scV_2]$.
 \item The differential is as follows. Suppose $x=\scV_1^{i} \cdot \xs\in \cB_s(K)$. Suppose that $\scU_1^n \scV_1^m \cdot \ys$ is a summand of $\d x$ (the differential of $\cCFK(K)$). Then $\d^\mu([x])$ has a summand $\scU_2^n \scV_2^n\cdot (\scV_1^{m-n} \cdot \ys)$.  
 \end{enumerate}
\end{lem}

It is straightforward to verify $\cB_{s}(L)$ is homotopy equivalent over $\bF[\scU_2,\scV_2]$ to the complex $\bF[\scU_2,\scV_2]$, with vanishing differential.

\subsection{The small model of the dual knot mapping cone complex}
\label{sec:small-model-mapping-cone}

We now apply the small models $\cA_s^\mu(K)\simeq \cA_{s}(L)$ and $\cB_{s}^\mu(K)\simeq \cB_{s}(L)$ from the previous section to the mapping cone formula from Theorem~\ref{thm:expanded-model} to construct the $\bF[\scU, \scV]$-chain complex $\bX^\mu_n(K)$ underlying the $\iota_K$-complex $\bXI^\mu_n(K)$ described in Theorem~\ref{thm:main}. We will first derive formulas for the maps appearing in the homotopy equivalences.

As a first step, we wish to understand the homotopy equivalences between $\cA_{s}(L)$ and $\cA_s^{\mu}(K)$. As these complexes are themselves defined via a box tensor of $\cP\boxtimes \cK\boxtimes M_s$ and $\cP\boxtimes \cK\boxtimes Z_s$, we first describe the $A_\infty$-homotopy equivalences
\[
\Pi\colon {}_{\Lambda} M_s\to {}_{\Lambda} Z_s \quad \text{and} \quad I\colon {}_{\Lambda} Z_s\to {}_{\Lambda} M_s.
\]
The $A_\infty$-module maps $\Pi$ and $I$ are defined in Figure~\ref{fig:Z-structure-rels}. 
  One easily computes
\[
\Pi\circ I=\id.
\]  Similarly, 
\[
I\circ \Pi=\id+\d_{\Mor}(H),
\]
where $H\colon {}_{\Lambda} M_s\to {}_{\Lambda} M_s$ is the map shown in Figure~\ref{fig:Z-structure-rels}. Note also that above compositions are $A_\infty$ compositions.
 Here, $\d_{\Mor}(H)$ denotes the differential of $H$ as an $A_\infty$-morphism. Schematically, these maps are depicted in Figure~\ref{fig:Z-structure-rels}.

We write $\Pi_{\cA}$ and $I_{\cA}$ for
\[
\Pi_{\cA}=\id \boxtimes \id\boxtimes \Pi\quad \text{and} \quad I_{\cA}=\id\boxtimes \id\boxtimes I.
\]
Note that boxing morphisms is not in general strictly associative, so we must define
\[
\Pi_{\cA}=(\id\boxtimes \id)\boxtimes \Pi\quad \text{or} \quad \Pi_{\cA}=\id\boxtimes (\id\boxtimes \Pi).
\]
 In our present case, since $\cP$ is a dg-algebra (in fact, the differential vanishes), the distinction vanishes (see \cite{LOTBimodules}*{Remark~2.2.28}). In both cases, the box tensor product takes the form shown in Figure~\ref{fig:boxing-morphisms}.

Note that in fact, $\cP\boxtimes \cK\boxtimes M_s$ and $\cP\boxtimes \cK\boxtimes Z_s$ can be thought of as $A_\infty$-modules over $\cP$, and $\id\boxtimes \id\boxtimes \Pi$ and $\id \boxtimes \id\boxtimes I$ are morphisms of $A_\infty$-modules over $\cP$. However, the following elementary lemma implies that the higher terms of these $A_\infty$-morphisms vanish.

\begin{lem}
 The higher actions of $\cP$ on ${}_{\cP}\cP_{\cP}\boxtimes {}^{\cP}\cK^{\Lambda}\boxtimes {}_{\Lambda} M_s$ and ${}_{\cP}\cP_{\cP}\boxtimes {}^{\cP}\cK^{\Lambda}\boxtimes{}_{\Lambda} Z_s$ vanish. Furthermore, as $A_\infty$-module morphims over $\cP$, the higher terms of $\id\boxtimes \id\boxtimes \Pi$ and $\id\boxtimes \id \boxtimes I$ vanish as well (i.e. $(\id\boxtimes \id\boxtimes \Pi)_j=0$ unless $j=1$, and similarly for $\id\boxtimes \id\boxtimes I$).
\end{lem}
\begin{proof}
This follows from a general principle. Suppose $\scA$ is a dg-algebra and ${}^{\scA} N$ is a type-D module. Then ${}_{\scA}\scA_{\scA}\boxtimes {}^{\scA} N$ is a dg-module over $\scA$. In particular, the higher actions vanish. This follows immediately from the structure relations on the box tensor product, which are depicted schematically as follows
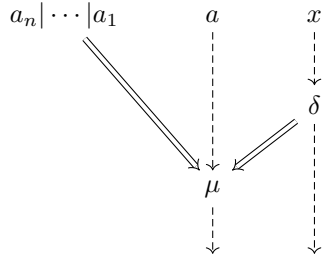
\[
\begin{tikzcd}a_n \otimes \cdots \otimes  a_1 \ar[ddr, Rightarrow] & a \ar[dd,dashed] & x \ar[d,dashed]\\
&& \delta\ar[dd,dashed] \ar[dl, Rightarrow]\\
& \mu \ar[d,dashed]& \\
&\,&\,  
\end{tikzcd}
\]
Here, $\delta$ is obtained by summing all ways of composing the map $\delta^1$ of $N$ a nonnegative number of times and concatenating the algebra elements (the composition of zero copies of $\delta^1$ is the identity map). The only case with non-vanishing contribution is when $n=1$ or $n=0$, since $\mu_i=0$ if $i>2$.

The case of morphisms is similar. If $\phi^1\colon {}^{\scA} N \to {}^{\scA} N'$ is a morphism of type-$D$ modules, then $\id\boxtimes \phi^1\colon \scA\boxtimes N\to \scA\boxtimes N'$ is obtained by a similar picture. Namely, we replace $\delta$ with the map $\phi$, obtained by summing all ways of composing $\delta^1$ of $N$ some nonnegative number of times, then $\phi^1$, then applying $\delta^1$ of $N'$ some nonnegative number of times, and inputting all of the algebra elements into $\mu_{j}$. The same argument applies to show that the only terms with non-trivial contribution have exactly one algebra input (which is from the module ${}_{\scA} \scA_{\scA}$). The claim follows.
\end{proof}

\begin{figure}[ht]
\[
\begin{tikzcd}[row sep=.2cm]
\cP \ar[dd,dashed] & \cK \ar[d,dashed] &M_s \ar[ddddddd,dashed]\\
&\delta^{1,1}\ar[dl]\ar[dd,dashed]\ar[ddddddr,bend left=10]&\\
m_2\ar[dd,dashed]& &\\
&\delta^{1,1}\ar[dd,dashed]\ar[dl]\ar[ddddr,bend left=10]&\\
m_2\ar[dd,dashed]& &\\
&\delta^{1,1}\ar[ddr,bend left=10]\ar[ddd,dashed]\ar[dl]&\\
m_2 \ar[dd,dashed]& &\\
& & \phi \ar[d,dashed]\\
\cP& \cK&Z_s
\end{tikzcd}
\]
\caption{A term contributing to the box tensor product $\id\boxtimes \id\boxtimes \phi$ (for either choice of parenthesization).}
\label{fig:boxing-morphisms}
\end{figure}

We will write $\Pi_{\cB}$ and $I_{\cB}$ for the analogous maps on the $\cB$-side. Define $H_{\cA}$ and  $H_{\cB}$ similarly.

For future reference, it will be helpful to have concrete formulas for the maps $\Pi_{\cA}$ and $I_{\cA}$.

\begin{lem}
\item
\begin{enumerate}
\item The map $I_{\cA}$ is given as follows:
\begin{enumerate}[label=($I_\cA$-\arabic*)]
\item Suppose that $x=[\scU^i_1 \xs]\in (\cA_{s+1}/\scV_1 \cA_s)(K)$. Then $I_{\cA}(x)$ is the sum of $\scU^i_1 \xs\in \cA_{s+1}(K)$ (viewed as being in the even summands of $\cA_s(L)$), together with the following terms which lie in the odd summands of $\cA_s(L)$. For each summand $\scU_1^n \scV_1^m \ys$ in $\d (\scU_1^i \xs)$ we have a term 
\[
\sum_{j=0}^{\min(n,m-1)} \scU_1^{n-j} \scV_1^{m-j-1}\ys \otimes \scU_2^j \scV_2^j\in \cA_s(K)
\]
as well as a term
\[
\sum_{j=0}^{\min(n-1,m-1)} \scU_1^{n-j-1} \scV_1^{m-j-1} \ys\otimes \scU_2^{j+1}\scV_2^{j}\in \cA_{s+1}(K).
\]
\item Suppose that $x=[\scV_1^i \xs]\in (\cA_{s}/ \scU_1 \cA_{s+1})(K)$. Then $I_{\cA}(x)$ is the sum of $\scV^i_1 \xs\in \cA_{s}(K)$ (viewed as being in the even summands of $\cA_s(L)$), together with the following terms which lie in the odd summands of $\cA_s(K)$. For each summand $\scU^n_1 \scV_1^m \ys$ in $\d(\scU_1^i \xs)$, we have a term
\[
\sum_{j=0}^{\min(n-1,m)} \scU_1^{n-j-1} \scV_1^{m-j} \ys \otimes \scU_2^j \scV_2^j\in \cA_{s+1}(K)
\]
as well as a term
\[
\sum_{j=0}^{\min(n-1,m-1)} \scU_1^{n-1-j} \scV_1^{m-1-j} \ys\otimes \scU_2^j \scV_2^{j+1}\in \cA_s(K).
\]
\end{enumerate}
\item The map $\Pi_{\cA}$ is given as follows:
\begin{enumerate}[label=($\Pi_\cA$-\arabic*)]
\item The map $\Pi_{\cA}$ vanishes on the odd summands of $\cA_s(L)$. 
\item If $x=\scU_1^n\scV_1^m\xs \in \cA_s(K)$ is in an even summand of $\cA_s(L)$, then
\[
\Pi_{\cA}(x)=
\begin{cases} 
\scV_1^{m-n} \xs\otimes \scU_2^n \scV_2^n \in (\cA_s/\scU_1 \cA_{s+1})(K)[\scU_2,\scV_2]  &\text{ if } n\le m\\
\scU_1^{n-m-1} \xs \otimes \scU_2^m \scV_2^{m+1}\in (\cA_{s+1}/\scV_1 \cA_s(K))[\scU_2,\scV_2] &\text{ if } n>m.
\end{cases}.
\] 
\item 
If $x=\scU_1^n\scV_1^m \xs\in \cA_{s+1}(K)$ is in an even summand of $\cA_s(L)$, then
\[
\Pi_{\cA}(x)=\begin{cases} \scU_1^{n-m} \xs\otimes \scU_2^{m} \scV_2^{m} \in (\cA_{s+1}/ \scV_1 \cA_{s})(K)[\scU_2,\scV_2] & \text{ if } n\ge m\\
\scV_1^{m-1-n} \xs\otimes \scU_2^{n+1} \scV_2^{n} \in (\cA_{s}/ \scU_1 \cA_{s+1})(K)[\scU_2,\scV_2] & \text{ if } n<m.
\end{cases}
\]
\end{enumerate}
\end{enumerate}
\end{lem}

We leave the computation in the previous lemma to the reader, as it follows immediately from the definitions of the maps involved.

We now define
\[
v^{\mu}:=\Pi_{\cB} \circ v_L\circ I_{\cA}\qquad \text{and} \qquad \tilde{v}^{\mu}:=\Pi_{\tilde{\cB}}\circ \tilde{v}_L\circ I_{\cA}.
\]

\begin{lem}\label{lem:v-maps}\,
\begin{enumerate}
\item The canonical inclusion maps $v_L\colon \cA_{s}(L)\to \cB_{s}(L)$ and $\tilde{v}_L\colon \cA_{s}(L)\to \tilde{\cB}_{s}(L)$ coincide with the maps $\id\boxtimes \id\boxtimes v_K$ and $\id\boxtimes \id\boxtimes \tilde{v}_K$, where $v_K$ and $\tilde{v}_K$ are the maps induced by the inclusions $\cA_s(K)\hookrightarrow\cB_{s}(K)$, and $\cA_s(K)\hookrightarrow \tilde{\cB}_s(K)$. (We henceforth write $v$ and $\tilde{v}$ for either definition).
\item The maps $v^\mu$ and $\tilde{v}^\mu$ are the identity on intersection points, with powers of $\scU$ and $\scV$ changed. In more detail:
\begin{enumerate}[label=($v^\mu$-\arabic*), ref=$v^\mu$-arabic*]
\item If $\scU_1^n\cdot  \ve{x}\in \cA_{s+1}(K)/\scV_1 \cA_s(K)$, then \[
v^{\mu}(\scU_1^n\cdot\ve{x})=\scU_2^{n+1} \scV_2^n\cdot(\scV_1^{-n-1}\cdot \ve{x})\in (\cB_s(K)/\scU_1 \cB_{s+1}(K))[\scU_2, \scV_2].
\] 
\item If $\scV_1^m\cdot\ve{x}\in \cA_{s}(K)/ \scU_1 \cA_{s+1}(K)$, then 
\[
v^{\mu}(\scV_1^m \cdot\ve{x})=\scV_1^m\cdot \ve{x}\in (\cB_s(K)/ \scU_1 \cB_{s+1}(K))[\scU_2, \scV_2].
\]
\end{enumerate}
Similarly:
\begin{enumerate}[label=($\tilde v^\mu$-\arabic*), ref=$\tilde v^\mu$-arabic*]
\item If $\scU_1^n\cdot  \ve{x}\in \cA_{s+1}(K)/\scV_1 \cA_s(K)$, then 
\[
\tilde{v}^{\mu}(\scU_1^n\cdot \ve{x})=\scU_1^n \cdot \ve{x}\in (\tilde\cB_{s+1}(K)/\scV_1 \tilde\cB_{s}(K))[\scU_2, \scV_2].
\]
\item If $\scV_1^m\cdot \ve{x}\in \cA_s(K)/\scU_1 \cA_{s+1}(K)$, then \[
\tilde{v}^{\mu}(\scV_1^m\cdot \ve{x})=\scU_2^m \scV_2^{m+1} \cdot (\scU_1^{-m-1}\cdot\ve{x})\in (\tilde{\cB}_{s+1}(K)/ \scV_1 \tilde{\cB}_s(K))[\scU_2, \scV_2].
\]
\end{enumerate}
\end{enumerate}
\end{lem}
\begin{proof}
(1) The maps $v_K$ and $\tilde{v}_K$ are maps of $A_\infty$-modules. However, they vanish when there is more than one input. Hence the box tensor product $\id\boxtimes \id \boxtimes v_K$ is just the ordinary tensor product $\id\otimes \id \otimes v_K$, which clearly coincides with $v_L$. The same remarks hold for $\tilde{v}_L$.

(2) The second statement is proven by explicitly computing the maps, as we do presently. We focus on $v^\mu$, since $\tilde{v}^\mu$ is handled similarly.  Recall that as $\bF[\scU_2,\scV_2]$-modules, there are isomorphisms
\[
\cA_s^{\mu}(K)\iso \bigg( (\cA_{s+1}/\scV_1 \cA_s)(K)\oplus (\cA_s/\scU_1 \cA_{s+1})(K)\bigg)[\scU_2,\scV_2]\] \[\cB_s^{\mu}(K)\iso (\cB_s/\scU_1 \cB_{s+1})(K)[\scU_2,\scV_2].
\]

We consider first $v^{\mu}$ evaluated on some $x=\scU_1^j\cdot \xs\in(\cA_{s+1}/\scV_1 \cA_s)(K)$. There is first the contribution of $I_{\cA}$, which is gotten by first applying $F\colon (\cA_{s+1}/\scV_1 \cA_s)\to \Cone(\scV_1\colon \cA_s\to \cA_{s+1})$ from Lemma~\ref{lem:snake}, and then summing contributions from the homological perturbation lemma, as in Figure~\ref{fig:Z-structure-rels}. The map $F$ has two summands. In the complex $\cA_{s}(L)$, let us call the codomain components of $M_s$ (viewed as a mapping cone) the \emph{even} components, and let us call the other two components the \emph{odd} components. The map $F$ has one component which maps into an even component (essentially via the identity map). There are no ways to apply $m_2(\theta,-)$ for $\theta\in \{\theta_{\scU},\theta_{\scV}\}$ non-trivially to this component. We then compose with $v_L$, and then apply $\Pi_{\cB}$. The map $\Pi_{\cB}$ may have contributions from the homological perturbation lemma of arbitrary length (i.e. many applications of $H$ and $m_2(\theta,-)$), but it is easy to see that the effect is only to transform the powers of $\scU_1$ and $\scV_1$ into powers of $\scU_2$ and $\scV_2$. There is one additional summand of $F$, which maps into an odd component of $\cA_{s}(L)$. This makes trivial contribution to $v^\mu$ since $v_L$ preserves evenness/oddness of the direct summands, and $\Pi_{\cB}$ vanishes on the odd summands.

The map $\tilde{v}^\mu$ is analyzed in much the same manner.
\end{proof}

\begin{rem} The maps $v^\mu$ and $\tilde{v}^\mu$ coincide with Hedden and Levine's canonical inclusion maps \cite{HeddenLevineSurgery}.
\end{rem}

We take the two hyperboxes
\[
\begin{tikzcd}
 \cA_{s}(L)\ar[r,"v"]& \cB_{s}(L)
 \end{tikzcd}
  \qquad 
 \begin{tikzcd}
 \cA_{s}(L)\ar[r,"\tilde{v}"]& \tilde{\cB}_{s}(L) \ar[r, "\frF_n"] &\cB_{s+n}(L) 
\end{tikzcd}
\]
apply the homological pertubation lemma of cubes to each one, and compress. We then take the direct product over all $s$ in the domain and codomain. This gives a model of the mapping cone
\begin{equation}\label{eq:small-cone}
\begin{tikzcd}[column sep=2cm]
\bA^\mu(K)\ar[r, "v^\mu+\frF_n^\mu \tilde{v}^\mu"] &\bB^\mu(K)
\end{tikzcd}
\end{equation}
as well as a homotopy  equivalence of 1-dimensional hypercubes (taking the form of a 2-dimensional  hypercube)
\[
\begin{tikzcd}[labels=description, row sep=2cm, column sep=2cm]
\bA(L)\ar[d, "\Pi_{\cA}"] \ar[dr] \ar[r,dashed, "v+\frF_n \tilde{v}"]& \bB(L) \ar[d, "\Pi_{\cB}"]\\
\bA^\mu(K)\ar[r,dashed, "v^\mu+\frF_n^\mu \tilde{v}^\mu"] &\bB^\mu(K).
\end{tikzcd}
\]

\noindent We define the complex $\bX^\mu_n(K)$ to be the mapping cone in Equation~\eqref{eq:small-cone}. 

\subsection{An algebraic model for knot-like complexes of $S^3$-space type}
\label{sec:algebraic-knot-like}

We now consider the constructions of Sections \ref{subsec:homperturbAs}-\ref{sec:small-model-mapping-cone} as applied to abstract knot-like complexes. We recall from Section~\ref{sec:iota-complexes} that a \emph{knot-like} complex $\scC$ is a free finitely-generated bigraded chain complex over $\bF[\scU, \scV]$ with the property that $H_*(\scC/(\scU-1))$ has a single $\scV$-nontorsion tower and $H_*(\scC/(\scV-1))$ has a single $\scU$-nontorsion tower. We will consider knot-like complexes of  $S^3$-type and $B_s \simeq \bF[U]$. Let us write $\bX^\mu_n(\scC)$ for the knot-like complex constructed from $\scC$ via the construction of the previous three sections, that is, the small model of the mapping cone with dual knot. We write $\bA^\mu(K)$ and $\bB^\mu(K)$ for the small models obtained from the homological perturbation lemma. Additionally, we introduce the notation of $\scC^H:= \scC\otimes_{\bF[\scU_1,\scV_1]} \cH^+$, where $\cH^+$ is the positive Hopf link complex.

We also pick a homotopy equivalence of $\bF[\scU_2,\scV_2]$-chain complexes
\[
\frF_n^\mu\colon \tilde{\bB}^\mu(K)\to \bB^\mu(K).
\]

It is helpful to also have the notation
\[
(\frF_n^\mu)^H:= I_{\cB} \circ \frF_n^\mu \circ \Pi_{\tilde{\cB}}.
\]
Note that since $\Pi_{\cB}\circ I_{\cB}=\id$, and similarly for the $\tilde{\cB}$ maps, we have
\[
\Pi_{\cB} \circ (\frF_n^\mu)^H\circ I_{\tilde{\cB}}=\frF_n^\mu.
\]

If $\scC_1$ and $\scC_2$ are two knot-like complexes of L-space-type with a single tower, and $F\colon \scC_1\to \scC_2$ is an $\bF[\scU,\scV]$-equivariant chain map, we define
\[
F_{\cA}^\mu:=\Pi_{\cA}\circ (F\otimes \id_H)\circ I_{\cA}.
\]
Similarly we define
\[
F_{\cB}^\mu:=\Pi_{\cB}\circ (F\otimes \id_H)\circ I_{\cB},
\]
We write $F_{\bA}^{\mu}$ for the direct product of the maps $F_{\cA}^\mu$ (over $s$), and similarly for the $\cB$ and $\tilde{\cB}$ versions.

Finally, we define a map $F_{\bX}^{\mu}$ via the following diagram:
\begin{equation}
F_{\bX}^\mu:=
\begin{tikzcd}[column sep=4cm, row sep=2cm, labels=description]
 \bA^\mu(\scC_1)\ar[d,"F_{\bA}^\mu"]\ar[dr, "H^{\can}_F+\frF_n^\mu\tilde{H}^{\can}_F+H^{\{*\}} \tilde{v}^\mu"] \ar[r, dashed, "v^\mu+\frF_n^\mu \tilde{v}^\mu"]& \bB^\mu(\scC_2) \ar[d, "F_{\bB}^\mu"]\\
\bA^\mu(\scC_2)\ar[r,dashed, "v^\mu+\frF_n^\mu\tilde{v}^\mu"]&\bB^\mu(\scC_2)
\end{tikzcd}
\label{eq:mu-ified-morphism-basic}
\end{equation}
The maps $H_F^{\can}$ and $\tilde{H}_F^{\can}$ are the canonical homotopies which satisfy the formulas 
\[
v^\mu F_{\cA}^\mu+F_{\cB}^{\mu} v^{\mu}=[\d, H_F^{\can}]\quad \text{and} \quad \tilde{v}^{\mu} F_{\cA}^\mu+F_{\tilde{\cB}}^{\mu} v^{\mu}=[\d, \tilde{H}_F^{\can}].
\]
They are given by the formulas
\begin{equation}
\begin{split}
H^{\can}_F&=\Pi_{\cB} \left(v H_{\cA} F+FH_{\cB} v\right) I_{\cA}\\
\tilde{H}^{\can}_F &=\Pi_{\tilde{\cB}}\left( \tilde{v} H_{\cA} F+F H_{\tilde{B}} \tilde{v}\right) I_{\cA}.
\end{split}
\label{eq:H-can-tilde-H-can-def}
\end{equation}
The map $H^{\{*\}}$ is any $\bF[\scU,\scV]$-equivariant  chain map which makes the diagram commute and has $(\gr_w,\gr_z)$-bigrading equal to $(1,1)$. We will shortly prove that such a map exists.

Here, the $\{*\}$ in the notation $H^{\{*\}}$ is chosen to indicate that there are a contractible set of choices in the construction of this homotopy.

The construction may also be performed when $F$ is skew-equivariant. Note that in this case, the homotopy $H^{\{*\}}$ is also chosen to be skew-equivariant.

\begin{rem} In the above, $v^\mu$ and $\tilde{v}^\mu$ are the maps
\[
v^\mu:= \Pi_{\cB} \circ v_L\circ I_{\cA}\quad \text{and} \quad \tilde{v}^\mu:=\Pi_{\tilde{\cB}}\circ \tilde{v}_L \circ I_{\cA}.
\]
Contrary to what the notation suggests, we do not need to assume that $\frF_n^\mu=\Pi_{\cB}\circ \frF_n\circ I_{\tilde \cB}$ for some map $\frF_n\colon \tilde{\bB}(L)\to \bB(L)$ (though such a $\frF_n^\mu$ is a valid choice). Instead $\frF^\mu_n$ may be  any $\bF[\scU,\scV]$-equivariant homotopy equivalence.
\end{rem} 

Here are the main properties of the map $F_{\bX}^{\mu}$:

\begin{prop} \label{prop:properties-F-X}
Suppose $\scC$ and $\scC'$ are two knot-like complexes of $L$-space type with a single tower and $F\colon \scC\to \scC'$ is an $\bF[\scU,\scV]$-equivariant chain map. Suppose further that $F$ is grading preserving.
\begin{enumerate}
\item\label{FX-1} There exists such an $H^{\{*\}}$, as claimed above. (This does not require the grading assumption).
\item\label{FX-2} Any two choices of $H^{\{*\}}$ give homotopic  $F_{\bX}^{\mu}$ maps.
\item\label{FX-3} The construction is functorial, i.e. $(G\circ F)_{\bX}^\mu\simeq G_{\bX}^\mu\circ F_{\bX}^\mu$.
\item \label{FX-4} If $F_1\simeq F_2$, then $F_{1,\bX}^\mu\simeq F_{2,\bX}^{\mu}$.
\item\label{FX-5}  $F$ is local if and only if $F^\mu_{\bX}$ is local. (This does not require the grading assumption).
\end{enumerate}
\end{prop}

Here is a preparatory lemma:

\begin{lem}\label{lem:F-X-prep-1} Suppose that $B$ is a chain complex over $\bF[\scU,\scV]$ which is homotopy equivalent to $\bF[\scU,\scV]$ over $\bF[\scU,\scV]$, and suppose that $F\colon B\to B$ is a map which is trivial on homology. Then $F\simeq 0$ via $\bF[\scU,\scV]$-equivariant chain homotopy. The same holds if instead $F$ is skew-equivariant.
\end{lem}
\begin{proof} Let
\[
\phi\colon B\to \bF[\scU,\scV],\quad \text{and} \quad \psi\colon \bF[\scU,\scV]\to B
\]
be homotopy equivalences. The map $\phi\circ F\circ \psi$ is 0, since it induces the 0 map on homology and $\bF[\scU,\scV]$ has vanishing differential. However
\[
F\simeq \psi \phi F \psi \phi=0,
\]
completing the proof.
\end{proof}

Here is another preparatory lemma:
\begin{lem}\label{lem:F-X-prep-2} Suppose that $\scC$ and $\scC'$ are knot-like complexes of L-space type with a single tower and $F\colon \scC\to \scC'$ is an $\bF[\scU,\scV]$-equivariant map which is grading preserving. Then, the induced maps $(F^\mu_{\cB})_*$ and $(F^\mu_{\tilde{\cB}})_*$ on $H_*(\cB^\mu_s)$ and $H_*(\tilde{\cB}^\mu_s)$ coincide under the canonical identification of both with $\bF[\scU,\scV]$. 
\end{lem}
\begin{proof} The claim is implied by the following subclaims: 
\begin{enumerate}
\item $(F^\mu_{\cB})_*$ is non-zero if and only if $F$ is a local map.
\item $(F^\mu_{\tilde{\cB}})_*$ is non-zero if and only if $F$ is a local map.
\end{enumerate}
We focus on the first claim, as the second claim is an analog. Since $F^\mu_{\cB}$ is defined by conjugating with the homotopy equivalences $\Pi_{\cB}$ and $I_{\cB}$, it suffices to prove the claim for $F\otimes \id_H$, on $\cB_{s}(\scC^H)$ (the big version). Localization is an exact functor, so $(\scU_2,\scV_2)^{-1}\cdot H_*(\cB_{s}(\scC^H))=H_*((\scU_2,\scV_2)^{-1}\cB_{s}(\scC^H))$. We may use the model of $\cB_{s}(\scC^H)$ in equation~\eqref{eq:large-surgery-B-s}, and simply invert $\scU_2$ and $\scV_2$. We also note that $\cB_s^\mu(\scC)$ and $\cB_{s+1}^\mu(\scC)$ are homotopy equivalent to $\bF[\scU_2,\scV_2]$ over the ring $\bF[\scU_2,\scV_2]$. The map $F$ is local and grading preserving if and only if, under such a homotopy equivalence, it is intertwined with $\id$. (In particular, a grading preserving map $F$ is not local if and only if it is intertwined with $0$). Hence, $F\otimes \id_H$ is intertwined with the identity map, if $F$ is local, and is intertwined with the 0 map, if $F$ is not local. The proof is complete.
\end{proof}

\begin{proof}[Proof of Proposition~\ref{prop:properties-F-X}]  The proofs of all claims can be obtained by constructing cubes on the expanded model, and then transporting them to the small model, as we now show.

 \eqref{FX-1}.  We build two hyperboxes. The first, for $v$, is shown below:
\[
\begin{tikzcd}[labels=description, column sep=2cm, row sep=2cm] \bA(\scC_1^H) \ar[d, "F\otimes \id"] \ar[r, "v"] &\bB(\scC_1^H) \ar[d, "F\otimes \id"]\\
\bA(\scC_2^H) \ar[r, "v"] &\bB(\scC_2^H).
\end{tikzcd}
\] 
No homotopy is necessary since $F$ is $\bF[\scU,\scV]$-equivariant. We now apply the homological perturbation lemma for cubes, Lemma~\ref{lem:homological-perturbation-cubes}, and we obtain the following diagram, which is homotopy equivalent as a hypercube of chain complexes
\begin{equation}
\begin{tikzcd}[labels=description, column sep=2cm, row sep=2cm] \bA^\mu(\scC_1) \ar[d, "F_{\bA}^\mu"] \ar[r, "v"] \ar[dr, dashed, "H^{\can}_F"] &\bB^\mu(\scC_1) \ar[d, "F_{\bB}^\mu"]\\
\bA^\mu(\scC_2) \ar[r, "v"] &\bB^\mu(\scC_2)
\end{tikzcd}
\label{eq:F_A_mu}
\end{equation}
where
\[
H^{\can}_F=\Pi_{\cB} v H_{\cA} F I_{\cA}+\Pi_{\cB} F H_{\cB} v  I_{\cA}.
\]

Next, to construct $\tilde{H}^{\can}_F$ and $H_F^{\{*\}}$, we build the following hyperbox
\begin{equation}
\begin{tikzcd}[labels=description, column sep=2cm, row sep=2cm]
\bA(\scC_1^H)
	\ar[d, "F\otimes \id"]
	\ar[r, "\tilde{v}"]
&
\tilde{\bB}(\scC_1^H)
	\ar[r, "(\frF_n^\mu)^H"]
	\ar[d,"F \otimes \id"]
	\ar[dr, "h_F^{\{*\}}", dashed] 
&
\bB(\scC_1^H)
	\ar[d, "F \otimes \id"]\\
\bA(\scC_2^H)
	\ar[r, "\tilde{v}"]
&
\tilde{\bB}(\scC_2^H)
	\ar[r, "(\frF_n^\mu)^H"]
&
\bB(\scC_2^H).
\end{tikzcd}
\label{eq:perturbed-hypercube}
\end{equation}
The left square commutes on the nose, since $F$ is $\bF[\scU,\scV]$-equivariant. The right square is constructed via an argument as in Lemma~\ref{lem:F-X-prep-1} and Lemma~\ref{lem:F-X-prep-2}. We now perform the homological perturbation lemma to both hypercubes, and obtain the following hyperbox.
\begin{equation}
\begin{tikzcd}[labels=description, column sep=2cm, row sep=2cm]
 \bA^\mu(\scC_1)
	\ar[dr, "\tilde{H}^{\can}_F",dashed]
	\ar[d, "F_{\bA}^\mu"]
	\ar[r, "\tilde{v}^\mu"] 
&
\tilde{\bB}^\mu(\scC_1) 
	\ar[r, "\frF_n^\mu"]
	\ar[d,"F_{\tilde{\bB}}^\mu"]
	\ar[dr, "H_F^{\{*\}}",dashed] 
&
\bB^\mu(\scC_1)
	\ar[d, "F_{\bB}^\mu"]
\\
\bA^\mu(\scC_2) 
	\ar[r, "\tilde{v}^\mu"]
&
\tilde{\bB}^\mu(\scC_2)
	\ar[r, "\frF_n^\mu"]
&
\bB^\mu(\scC_2).
\end{tikzcd}
\label{eq:perturbed-hypercube-tilde}
\end{equation}

We define $F_{\bX}^\mu$ to be the sum of $F_{\bA}^{\mu}$, $F^\mu_{\bB}$, the diagonal map in the compression of equation~\eqref{eq:perturbed-hypercube-tilde} and the diagonal map in equation~\eqref{eq:F_A_mu}.

\eqref{FX-2}. Any two choices of $H_F^{\{*\}}$ are homotopic because the difference of any two choices is a $(+1,+1)$-bigraded chain map between complexes which are both homotopy equivalent to $\bF[\scU,\scV]$ as graded complexes.

\eqref{FX-3}. The relation $(G\circ F)_{\bX}^\mu\simeq G_{\bX}^\mu\circ F_{\bX}^\mu$ is proven by further extensions of the above idea, as we now explain. Analogous to the previous situation, we build two 3-dimensional hypercubes. We form a new hypercube whose underlying complexes agree with these two hypercubes, and whose maps are suitable sums of the maps appearing therein.  The first is shown below:
\[
\begin{tikzcd}[
column sep={2.5cm,between origins},
row sep=.3cm,labels=description
]
\bA(\scC_1^H)
	\ar[dd, "F"]
	\ar[dr,  "\id"]
	\ar[rr, "v_L"]
&&[-1.5cm]
\bB(\scC_1^H)
	\ar[dd, "F"]
	\ar[dr,"\id"]
&
\\
&\bA(\scC_1^H)
	\ar[rr,crossing over, "v_L"]
&&
\bB(\scC_1^H)
	\ar[dd, "GF"]
\\[1.8cm]
\bA(\scC_2^H)
	\ar[rr, "v_L"]
	\ar[dr,"G"]
&&
\bB(\scC_2^H)
	\ar[dr, "G"]	
&
\\
&
\bA(\scC_3^H)
	\ar[rr, "v_L"]
	\ar[from =uu, crossing over, "GF"]
	&&
\bB(\scC_3^H)
\end{tikzcd}
\]
The second is gotten by stacking and compressing the cubes below
\[
\begin{tikzcd}[
column sep={2.5cm,between origins},
row sep=.3cm,labels=description
]
\bA(\scC_1^H)
	\ar[dd, "F"]
	\ar[dr,  "\id"]
	\ar[rr, "\tilde{v}_L"]
&&[-1.5cm]
\tilde{\bB}(\scC_1^H)
	\ar[dd, "F"]
	\ar[dr,"\id"]
&
\\
&\bA(\scC_1^H)
	\ar[rr,crossing over, "\tilde{v}_L"]
&&
\tilde{\bB}(\scC_1^H)
	\ar[dd, "GF"]
\\[1.8cm]
\bA(\scC_2^H)
	\ar[rr, "\tilde{v}_L"]
	\ar[dr,"G"]
&&
\tilde{\bB}(\scC_2^H)
	\ar[dr, "G"]	
&
\\
&
\bA(\scC_3^H)
	\ar[rr, "\tilde{v}_L"]
	\ar[from =uu, crossing over, "GF"]
	&&
\tilde{\bB}(\scC_3^H)
\end{tikzcd}
\begin{tikzcd}[
column sep={2.5cm,between origins},
row sep=.3cm,labels=description
]
\tilde{\bB}(\scC_1^H)
	\ar[dd, "F"]
	\ar[dr,  "\id"]
	\ar[rr, "(\frF_1^\mu)^H"]
	\ar[ddrr,dashed]
	\ar[dddrrr,dotted]
&&[-1.5cm]
\bB(\scC_1^H)
	\ar[dd, "F"]
	\ar[dr,"\id"]
	\ar[dddr,dashed]
&
\\
&\tilde{\bB}(\scC_1^H)
	\ar[rr,crossing over, "(\frF_1^\mu)^H"]
&&
\bB(\scC_1^H)
	\ar[dd, "GF"]
\\[1.8cm]
\tilde{\bB}(\scC_2^H)
	\ar[rr, "(\frF_2^\mu)^H"]
	\ar[dr,"G"]
	\ar[drrr,dashed]
&&
\bB(\scC_2^H)
	\ar[dr, "G"]	
&
\\
&
\tilde{\bB}(\scC_3^H)
	\ar[rr, "(\frF_3^\mu)^H"]
	\ar[from =uu, crossing over, "GF"]
	&&
\bB(\scC_3^H)
	\ar[from=uull,crossing over, dashed]
\end{tikzcd}
\]
The unlabeled maps above are built using our standard procedure, as in Lemma~\ref{lem:F-X-prep-1}.
After applying homological perturbation of hypercubes, compressing, and then gluing, these diagrams give exactly a filtered homotopy between $(G\circ F)_{\bX}^\mu$ and $G_{\bX}^\mu\circ F_{\bX}^\mu$. Note here the subscript $i$ on $(\frF_i^\mu)^H$ denotes the index of $\scC_i^H$, and not the surgery parameter.

Claim~\eqref{FX-4} is proven similarly. Namely, given a homotopy between $F_1$ and $F_2$, we construct 3-dimensional hypercubes, similar to the ones above. We apply homological perturbation of hypercubes, compress and then combine these into a single hypercube, similar to the above construction. We leave the details to the reader.

Claim~\eqref{FX-5} is straightforward, and follows from Lemma~\ref{lem:F-X-prep-2}.
\end{proof}

\section{The involution on the small model of the mapping cone} \label{sec:involution}

In this section, we compute the link involution on the dual knot surgery formula. We begin by computing it on the expanded model, and then we transfer it to the small model via homological perturbation theory.

\subsection{The involution on the expanded model}

We consider the link $L=K\# H$. We apply an essentially identical argument as in \cite{HHSZ} to the component $K$ of $L$ to obtain the following:

\begin{thm}\label{thm:involutive-mapping-cone-big} Suppose $K$ is a knot in $S^3$. Then there is a homotopy equivalence of (diagrams of) chain complexes
\[
\left(\begin{tikzcd}[row sep=1.2cm,labels=description] \ve{\cCFK}(S_n^3(K),\mu) \ar[d, "\iota_\mu"]
\\ 
 \ve{\cCFK}(S_n^3(K),\mu)
\end{tikzcd}
\right)\to
 \left(\begin{tikzcd}[labels=description, row sep=1.2cm, column sep=1.5cm]\bA(L) \ar[r, "v+\frF_n\tilde{v}"] \ar[dr,dashed, "H_n\tilde{v}"] \ar[d, "\iota_{L}"] &\bB(L) \ar[d,"\frF_n \iota_L"]\\
\bA(L) \ar[r, "v+\frF_n\tilde{v}"]& \bB(L)
\end{tikzcd}\right).
\]
More precisely, there is a $3$-d hypercube of chain complexes, giving a morphism from the left to the right, such that the top and bottom faces determine a homotopy equivalence of $\bF[\scU_2,\scV_2]$-chain complexes between $\ve{\cCFK}(S_n^3(K),\mu)$ and $\Cone( v+\frF_n \tilde{v}\colon \bA(L)\to \bB(L))$.  Here, $L=K\cup \mu$. Furthermore
\begin{enumerate}
\item $v$ and $\tilde{v}$ are the canonical inclusions.
\item $\frF_n\colon \tilde{\cB}_{s}\to \cB_{s+n}$ is a homotopy equivalence of $\bF[U,\scU_2,\scV_2]$-chain complexes.
\item The homotopy $H_n$ is an $\bF[\scU_2,\scV_2]$-equivariant map, and sends $\tilde{\cB}_{s}(L)$ to $\cB_{-s-1}(L)$.
\end{enumerate}
\end{thm}

\begin{rem} 
The above theorem is almost entirely obtained by repeating the proof of \cite{HHSZ} verbatim. There is one point which requires comment. In the construction of the hypercubes $\cK^{(1)}_{\cen, \fry_s}$ from \cite{HHSZ}*{Section~21}, we did simplify the argument by using the fact that $B_s\simeq \bF[U]$, so that any $+1$ graded map must be null-homotopic. In our present setting, we also need to use the fact that $\cB_{s}(L)$ is homotopy equivalent to the complex for an unknot in $S^3$ with an extra free-basepoint. In particular $\cB_{s}(L)\simeq \bF[\scU_2,\scV_2]$ over the ring $\bF[\scU_2,\scV_2]$. Hence, the same reasoning as in \cite{HHSZ} may be used by Lemma~\ref{lem:F-X-prep-1}.

The techniques of \cite{HHSZ} do not guarantee that $H_n$ can be taken to be $\bF[U,\scU_2,\scV_2]$-equivariant, but rather only $\bF[\scU_2,\scV_2]$-equivariant, cf. Lemma~\ref{lem:F-X-prep-1}. We expect it to be possible to choose $H_n$ to be $\bF[U,\scU_2,\scV_2]$-equivariant; however, we only need to control the homotopy type over $\bF[\scU_2,\scV_2]$, so the present techniques are sufficient to our purposes. 
\end{rem}

\subsection{The link involution $\iota_L$}

\label{sec:construct-iota_L}

We now compute the map labeled $\iota_L$ in Theorem~\ref{thm:involutive-mapping-cone-big}. We do this by computing the link involution for the Hopf link, and using the involutive connected sum formula from \cite{ZemConnectedSums}.

 We begin by computing the involution on the Hopf link complex. Recall that the link Floer complex of the positive Hopf link takes the following form:
\[
\begin{tikzcd}[labels=description,row sep=1cm, column sep=1cm] \ve{a} \ar[d, "\scV_1"]\ar[r, "\scU_2"]& \ve{b}\\
\ve{c}& \ve{d} \ar[u, "\scU_1"] \ar[l, "\scV_2"]
\end{tikzcd}
\]
\begin{lem} The link involution on the positive Hopf link complex satisfies
\[
\iota_H(\ve{a})=\ve{d},\quad \iota_H(\ve{b})=\ve{c},\quad \iota_{H}(\ve{c})=\ve{b}\quad \iota_H(\ve{d})=\ve{a},
\]
and is skew-equivariant with respect to $\bF[\scU_1,\scV_1,\scU_2,\scV_2]$; that is, it intertwines $\scU_1$ and $\scV_1$, and intertwines $\scU_2$ and $\scV_2$. 
\end{lem}
\begin{proof}
In general, if $L$ is a 2-component link, then $\iota_L^2\simeq (\id+\Phi_1\Psi_1)(\id+\Phi_2\Psi_2)$. This is because, just as in the case of knots \cite{HMInvolutive}, the map $\iota_L^2$ will be chain homotopic to the diffeomorphism map for performing a Dehn twist on each link component. Using \cite{ZemQuasi}, this diffeomorphism map may be computed to be $(\id+\Phi_1\Psi_1)(\id+\Phi_2\Psi_2)$ (see also \cite{SarkarMovingBasepoints}). In the case of the positive Hopf link, we have that $\Phi_i\Psi_i=0$, so $\iota_H^2\simeq \id$.  Additionally, $A(\iota_H(x))=-A(x)$, and $\iota_H$ is skew-equivariant. The only map which has these properties is the formula in the statement.
\end{proof}

Using the tensor product formula for the involution from \cite{ZemConnectedSums}*{Theorem~1.1}, we obtain
\[
\iota_L=(\id \otimes \id+\Phi_{K,1} \otimes \Psi_{H,1})(\iota_K \otimes  \iota_H)
\]
 Here, $\Phi_{K,1}$ is the $\Phi$ map for $\scU_1$ on $K$, while $\Psi_{H,1}$ denotes the $\Psi$ map for $\scV_1$ on the Hopf link complex.
 (Note that \cite{ZemConnectedSums}*{Theorem~1.1} is stated for connected sums of knots, though the same statement and proof hold for connected sums of links).
We obtain the model of 
\[
\iota_L\colon \cA_{s}(L)\to \cA_{-s-1}(L)
\]
 shown in Figure~\ref{fig:large-surgery-expanded}.
\begin{figure}[h]
\[
\begin{tikzcd}[labels=description,row sep=1cm, column sep=1cm] 
\cA_s(K) 
	\ar[d,dashed, "\scV_1"]
	\ar[r,dashed, "\scU_2"]
	\ar[dddr,"\iota_K", bend right=105, looseness=2.5]
& 
\cA_s(K) 	\ar[dddl,"\iota_K", bend left=105, looseness=2.5]\\
\cA_{s+1}(K)
	\ar[dr, "\iota_K"]
&
\cA_{s+1}(K)
	\ar[u, dashed,"\scU_1"]
	\ar[l, "\scV_2",dashed]
	\ar[dl,"\iota_K"]
	\ar[ddl, bend left=95, looseness=1.8, "\Phi_{K,1}\circ\iota_K", pos=.3]
\\
\cA_{-s-1}(K)
	\ar[d,dashed, "\scV_1"]
	\ar[r, dashed,"\scU_2"]
	&
\cA_{-s-1}(K)\\
\cA_{-s}(K)&
\cA_{-s}(K)
	\ar[u, dashed,"\scU_1"]
	\ar[l,dashed, "\scV_2"]
\end{tikzcd}
\]
\caption{The involution $\iota_L\colon \cA_{s}(L)\to \cA_{-s-1}(L)$. Dashed arrows are internal differentials. Also, we write $\cA_s(K)$ for $\cA_s(K)[\scU_2,\scV_2]$.}
\label{fig:large-surgery-expanded}
\end{figure}

\subsection{Transferring the involution to the small model}
\label{sec:Transferring-to-small}

We now describe a preliminary version of the small model of the involutive dual knot surgery formula. In this section, the maps are described in terms of the auxiliary inclusion, projection and homotopy maps described in Section~\ref{sec:small-model-mapping-cone}.  In the subsequent Section~\ref{sec:formulas}, we will give explicit formulas for all of the maps which appear on the level of generators.

\begin{thm}\label{thm:surgery-hypercube}
Suppose that $K$ is a knot in $S^3$. Then there is a homotopy equivalence of 2-dimensional hypercubes
\begin{equation}\label{eq:iotacomplex}
\left(\begin{tikzcd}[row sep=1.5cm, labels=description] \ve{\cCFK}(S_n^3(K),\mu) \ar[d, "\iota_\mu"]
\\ 
 \ve{\cCFK}(S_n^3(K),\mu)
\end{tikzcd}
\right)\to
 \left(
 \begin{tikzcd}[labels=description, row sep=1.5cm, column sep=5cm]
\bA^{\mu}(K)
	\ar[r, "v^\mu+\frF_n^\mu\tilde{v}^\mu"]
	\ar[dr,dashed, "\frF_n^\mu \tilde{H}_{\Omega} \iota_K^\mu+H_{\Omega}\iota_K^\mu+H^{\{*\}}\tilde{v}^\mu"]
	\ar[d, "\iota_{K}^{\mu}+\Omega^\mu \iota_K^\mu"] 
&\bB^{\mu}(K)
	\ar[d,"\frF_n^{\mu} (\iota_K^{\mu}+\Omega^\mu\iota_K^\mu)"]
\\
\bA^{\mu}(K)
	\ar[r, "v^{\mu}+\frF_n^{\mu}\tilde{v}^{\mu}"]
&
\bB^{\mu}(K)
\end{tikzcd}\right).
\end{equation}
Here, we view the left hypercube as being a 2-dimensional hypercube with vanishing groups at $\veps=(1,0),(1,1)$.
All arrows in the equivalence are either $\bF[\scU_2,\scV_2]$-equivariant or skew-equivariant. Furthermore, we have the following:
\begin{enumerate} 
\item $\iota_K^\mu$ is the $\mu$-ification of $\iota_K\otimes \iota_H$ (i.e. $\iota_K^\mu=\Pi_{\cA}\circ (\iota_K\otimes \iota_H)\circ  I_{\cA}$ and similarly on $\cB^\mu_s$).
\item $\Omega^\mu$ is the $\mu$-ification of $\Phi_{K,1}\otimes \Psi_{H,1}$.
\item $\tilde{H}_{\Omega}$ is the canonical null-homotopy of $\Omega^\mu \tilde{v}^\mu+\tilde{v}^\mu \Omega^\mu$.
\item $H_{\Omega}$ is the canonical null-homotopy of $\Omega^\mu v^\mu+v^\mu \Omega^\mu$.
\end{enumerate}
\end{thm}

\noindent The chain complex $\bX^\mu_n(K)$ together with the involution shown on the right-hand side of (\ref{eq:iotacomplex}) is the $\iota_K$-complex $\bXI^{\mu}_n(K)$.

\begin{rem} 
\begin{enumerate}
\item 
The maps $H_\Omega$ and $\tilde{H}_{\Omega}$ are \emph{canonical} in the sense that they are uniquely determined by a choice of basis for the chain complex $\cCFK(S^3,K)$ via nested applications of the homological perturbation lemma (one for $A_\infty$-modules and one for hypercubes). Indeed, we have (reasonably) explicit formulas:
\[
\begin{split}
H_\Omega&=\Pi_{\tilde{\cB}} (\Phi_{K,1} \otimes  \Psi_{H,1}) H_{\cB} v  I_{\cA}+\Pi_{\tilde{\cB}} \tilde{v} H_{\cA} (\Phi_{K,1} \otimes  \Psi_{H,1})  I_{\cA}\\
\tilde{H}_\Omega&=\Pi_{\cB} (\Phi_{K,1} \otimes \Psi_{H,1}) H_{\tilde{\cB}} \tilde{v}  I_{\cA}+\Pi_{\cB} v H_{\cA} (\Phi_{K,1} \otimes \Psi_{H,1})  I_{\cA}.
\end{split}
\]
We will shortly give much more explicit formulas for $\tilde{H}_{\Omega}$ and $H_{\Omega}$.
\item $\frF_n^\mu \tilde{H}_\Omega\iota_K^\mu$ sends $\cA_s^\mu(K)$ to $\cB_{-s}^\mu(K)$.
\item $H_\Omega \iota_K^\mu$ sends $\cA_s^\mu(K)$ to $\cB_{-s-1}^\mu(K)$. 
\item $\iota_K^\mu$ sends $\cA_s^\mu(K)$ to $\cA_{-s-1}^\mu(K)$, and $\cB_s^{\mu}(K)$ to $\tilde{\cB}_{-s-1}^\mu(K)$. The map $\frF_n^\mu$ sends $\tilde{\cB}_{s}^\mu(K)$ to $\cB_{s+n}^\mu(K)$.
\item We will shortly see that that $\tilde{v}^\mu \iota_K^\mu+\iota_K^\mu v^\mu$ and $v^\mu \iota_K^\mu+\iota_K^\mu \tilde{v}^\mu$ both vanish.
\item The map $H^{\{*\}}$ is a null-homotopy of 
\[
\frF_n^\mu \iota_K^\mu \frF_n^\mu+\iota_K^\mu+\frF_n^\mu \Omega^\mu \iota_K^\mu \frF_n^\mu+\Omega^\mu \iota_K^\mu,
\]
and sends $\tilde{\cB}_s^\mu$ to $\cB^\mu_{-s-1}$.
\end{enumerate}
\end{rem}

\begin{proof}[Proof of Theorem~\ref{thm:surgery-hypercube}] The proof follows from the homological perturbation lemma for hypercubes (Lemma~\ref{lem:homological-perturbation-cubes}), after a few preliminary steps. As a first step, we observe that any two hypercubes constructed as in the statement of Theorem~\ref{thm:involutive-mapping-cone-big} (i.e. for different choices of $\frF_n$ and $H_n$) are homotopy equivalent. This follows from the same techniques as in Section~\ref{sec:algebraic-knot-like}. See \cite{HHSZ}*{Section~3.5} for a very similar construction.

As a consequence, we may view the expanded model of the involutive surgery hypercube as being obtained by combining four hypercubes of the following form together:
\begin{equation}
\begin{array}{cc}
\begin{tikzcd}[labels=description, row sep=1.5cm, column sep=2.3cm]
\bA(L)
	\ar[r, "v"]
	\ar[d, "\iota_K \otimes  \iota_H"]
&
\bB(L)
	\ar[d, "\frF_n(\iota_K \otimes  \iota_H)"]
\\
\bA(L)
	\ar[r, "\frF_n \tilde{v}"]
&
\bB(L)
\end{tikzcd}
&
\begin{tikzcd}[labels=description, row sep=1.5cm, column sep=2.3cm]
\bA(L)
	\ar[r, "\frF_n\tilde{v}"]
	\ar[d, "\iota_K \otimes \iota_H"]
	\ar[dr,dashed]
&
\bB(L)
	\ar[d, "\frF_n (\iota_K \otimes  \iota_H)"]
\\
\bA(L)
	\ar[r, "v"]
&
\bB(L)
\end{tikzcd}
\\
\begin{tikzcd}[labels=description, row sep=1.5cm, column sep=2.3cm]
\bA(L)
	\ar[r, "v"]
	\ar[d, "(\Phi_{K,1} \otimes  \Psi_{H,1})(\iota_K \otimes  \iota_H)"]
&
\bB(L)
	\ar[d, "\frF_n(\Phi_{K,1} \otimes \Psi_{H,1})(\iota_K \otimes \iota_H)"]
\\
\bA(L)
	\ar[r, "\frF_n \tilde{v}"]
&
\bB(L)
\end{tikzcd}
&
\begin{tikzcd}[labels=description, row sep=1.5cm, column sep=2.3cm]
\bA(L)
	\ar[r, "\frF_n \tilde{v}"]
	\ar[d, "(\Phi_{K,1} \otimes \Psi_{H,1})(\iota_K \otimes \iota_H)"]
	\ar[dr,dashed]
&
\bB(L)
	\ar[d, "\frF_n(\Phi_{K,1} \otimes \Psi_{H,1})(\iota_K \otimes \iota_H)"]
\\
\bA(L)
	\ar[r, "v"]
&
\bB(L)
\end{tikzcd}
\end{array}
\label{eq:unglue-hypercubes-1}
\end{equation}

The hypercubes in~\eqref{eq:unglue-hypercubes-1} may be further decompressed, up to homotopy, as in Figure~\ref{fig:build-hyperboxes-quickly-2}. Therein, the map $\frG\colon \bB(L)\to \tilde \bB(L)$ denotes any $\bF[U,\scU_2,\scV_2]$-equivariant homotopy inverse to $\frF_n$.

 We then apply the homological perturbation lemma of hypercubes to each hypercube in Figure~\ref{fig:build-hyperboxes-quickly-2}. We compress the resulting hyperboxes (note that compressing hyperboxes in general requires a choice of an ordering of the axis directions; we compress vertically first and then horizontally). Then we form a new hypercube by combining the maps from these hypercubes, as follows. The horizontal maps are $v+\frF_n \tilde{v}$. The vertical map from $\bA(L)$ to $\bA(L)$ is $(1+\Phi_{K,1} \otimes \Psi_{H,1})(\iota_K \otimes \iota_H)$. The vertical map from $\bB(L)$ to $\bB(L)$ is $\frF_n (1+\Phi_{K,1} \otimes \Psi_{H,1})(\iota_K \otimes \iota_H)$. The diagonal map is the sum of the diagonal maps from equation~\eqref{eq:unglue-hypercubes-1}.

This gives us a model of the surgery hypercube which is homotopy equivalent to the expanded model, and coincides with the model in the statement, except that there are some additional summands appearing in the diagonal map. To show that the terms in the statement are the only terms which make non-vanishing contribution, we make the following observations:
\begin{enumerate}
\item\label{eq:trivial-1} If we apply the homological perturbation lemma of hypercubes to any sub-hypercube in Figure~\ref{fig:build-hyperboxes-quickly-2}  which has the properties that there are length 1 maps which are the identity, and also there is no length 2 map, then the resulting hypercube also has no length 2 map.
\item\label{eq:trivial-2} If we apply the homological perturbation lemma of hypercubes to either of the following two hypercubes, we obtain a hypercube with no length 2 map:
\[
\begin{tikzcd}[labels=description, row sep=1.5cm, column sep=1.5cm]
\bA(L)
	\ar[r, "v"]
	\ar[d, "\iota_K \otimes  \iota_H"]
&
\bB(L)
	\ar[d, "\iota_K \otimes  \iota_H"]
\\
\bA(L)
	\ar[r, "\tilde{v}"]
&
\tilde{\bB}(L)
\end{tikzcd}\qquad \text{or} \qquad \begin{tikzcd}[labels=description, row sep=1.5cm, column sep=1.5cm]
\bA(L)
	\ar[r, "\tilde{v}"]
	\ar[d, "\iota_K \otimes  \iota_H"]
&
\tilde{\bB}(L)
	\ar[d, "\iota_K \otimes  \iota_H"]
\\
\bA(L)
	\ar[r, "v"]
&
\bB(L)
\end{tikzcd}
\]
\end{enumerate}
Claim~\eqref{eq:trivial-1} above follows from the fact that $H_{\cA}  I_{\cA}=0$ and $\Pi_{\cA}H_{\cA}=0$, as well as the analogous claims for the $\cB$ and $\tilde{\cB}$ versions of the maps. For example, the homotopy produced from the homological perturbation lemma for the cube with $\tilde{v}$ on opposite faces, and $\id$ on opposite faces reads
\[
\Pi_{\tilde{\cB}} H_{\tilde{\cB}} \tilde{v}  I_{\cA}+\Pi_{\tilde{\cB}} \tilde{v} H_{\cA}  I_{\cA},
\]
which is zero for the above reasons. Additionally, there is also a cube which has two $\id$ terms, and two $\frF_n$ terms. The version of this sub-cube obtained by applying the homological perturbation lemma has zero length 2 maps because the two homotopy terms are both zero, by the same reasoning.

Next, we consider claim~\eqref{eq:trivial-2}. For this claim, we observe that  $\Pi_{\cB} v H_{\cA}$ is zero. Also $\Pi_{\tilde{\cB}} \tilde{v} H_{\cA}$ is zero. Slightly more subtlety, both maps $\Pi_{\cB}(\iota_K \otimes \iota_H) H_{\tilde{\cB}}$ and $\Pi_{\tilde{\cB}}(\iota_K \otimes \iota_H) H_{\cB}$ are zero. This is seen as follows. We recall the module $M_s$ from Equation~\eqref{eq:module-M-s}. We refer to the two summands of $\cA_s(L)$ corresponding to the codomain of $M_s$ (viewed as a direct sum of two mapping cones) as the \emph{even} summands. We refer to the summands corresponding to the domain subspace of $M_s$ as the \emph{odd} summands. The maps $H_{\tilde{\cB}}$ and $H_{\cB}$ both have image in the odd summands. The map $\iota_K \otimes \iota_H$ preserves the even and odd summands. Finally $\Pi_{\cB}$ and $\Pi_{\tilde{\cB}}$ vanish on the odd summands. Hence, $\Pi_{\cB}(\iota_K \otimes \iota_H) H_{\tilde{\cB}}$ and $\Pi_{\tilde{\cB}}(\iota_K \otimes \iota_H) H_{\cB}$ are both zero, completing the proof.
\end{proof}

\begin{figure}[ht!]
\[
\begin{tikzcd}[labels=description, column sep=1.6cm, row sep=1.6cm]
 \bA(L)
	\ar[d, "\iota_K \otimes  \iota_H"]
	\ar[r, "v"] 
&
\bB(L)
	\ar[r, "\id"]
	\ar[d,"\iota_K \otimes  \iota_H"]
&
\bB(L)
	\ar[d, "\iota_K \otimes  \iota_H"]
\\
\bA(L)
	\ar[r, "\tilde{v}"]
	\ar[d,"\id"]
&
\tilde{\bB}(L)
	\ar[r, "\id"]
	\ar[d, "\id"]
&
\tilde{\bB}(L)
	\ar[d,"\frF_n"]
\\
\bA(L)
	\ar[r, "\tilde{v}"]
&
\tilde{\bB}(L)
	\ar[r, "\frF_n"]
&
\bB(L)
\end{tikzcd}
\qquad 
\begin{tikzcd}[labels=description, column sep=1.6cm, row sep=1.6cm]
 \bA(L)
	\ar[d, "\iota_K \otimes  \iota_H"]
	\ar[r, "\tilde{v}"] 
&
\tilde{\bB}(L)
	\ar[r, "\frF_n"]
	\ar[d,"\iota_K \otimes \iota_H"]
	\ar[dr ,dashed] 
&
\bB(L)
	\ar[d, "\iota_K \otimes  \iota_H"]
\\
\bA(L)
	\ar[r, "v"]
	\ar[d,"\id"]
&
\bB(L)
	\ar[r, "\frG"]
	\ar[d, "\id"]
	\ar[dr,dashed]
&
\tilde{\bB}(L)
	\ar[d,"\frF_n"]
\\
\bA(L)
	\ar[r, "v"]
&
\bB(L)
	\ar[r, "\id"]
&
\bB(L)
\end{tikzcd}
\]
\[
\begin{tikzcd}[labels=description, column sep=1.6cm, row sep=1.6cm]
 \bA(L)
	\ar[d, "\iota_K \otimes  \iota_H"]
	\ar[r, "v"] 
&
\bB(L)
	\ar[r, "\id"]
	\ar[d,"\iota_K \otimes  \iota_H"]
&
\bB(L)
	\ar[d, "\iota_K \otimes  \iota_H"]
\\
\bA(L)
	\ar[r, "\tilde{v}"]
	\ar[d,"\Phi_{K,1} \otimes  \Psi_{H,1}"]
&
\tilde{\bB}(L)
	\ar[r, "\id"]
	\ar[d, "\Phi_{K,1} \otimes  \Psi_{H,1}"]
&
\tilde{\bB}(L)
	\ar[d,"\Phi_{K,1} \otimes  \Psi_{H,1}"]
\\
\bA(L)
	\ar[r, "\tilde{v}"]
	\ar[d,"\id"]
&
\tilde{\bB}(L)
	\ar[r, "\id"]
	\ar[d, "\id"]
&
\tilde{\bB}(L)
	\ar[d,"\frF_n"]
\\
\bA(L)
	\ar[r, "\tilde{v}"]
&
\tilde{\bB}(L)
	\ar[r, "\frF_n"]
&
\bB(L)
\end{tikzcd}
\qquad 
\begin{tikzcd}[labels=description, column sep=1.6cm, row sep=1.6cm]
 \bA(L)
	\ar[d, "\iota_K \otimes \iota_H"]
	\ar[r, "\tilde{v}"] 
&
\tilde{\bB}(L)
	\ar[r, "\frF_n"]
	\ar[d,"\iota_K \otimes  \iota_H"]
	\ar[dr ,dashed] 
&
\bB(L)
	\ar[d, "\iota_K \otimes  \iota_H"]
\\
\bA(L)
	\ar[r, "v"]
	\ar[d,"\Phi_{K,1} \otimes \Psi_{H,1}"]
&
\bB(L)
	\ar[r, "\frG"]
	\ar[d, "\Phi_{K,1} \otimes \Psi_{H,1}"]
	\ar[dr, dashed]
&
\tilde{\bB}(L)
	\ar[d,"\Phi_{K,1} \otimes \Psi_{H,1}"]
\\
\bA(L)
	\ar[r, "v"]
	\ar[d,"\id"]
&
\bB(L)
	\ar[r, "\frG"]
	\ar[d, "\id"]
	\ar[dr,dashed]
&
\tilde{\bB}(L)
	\ar[d,"\frF_n"]
\\
\bA(L)
	\ar[r, "v"]
&
\bB(L)
	\ar[r, "\id"]
&
\bB(L)
\end{tikzcd}
\]
\caption{Hyperboxes we use to build the expanded model of the involutive surgery hypercube. Applying homological perturbation of hypercubes to each subcube and then compressing gives the small model of the surgery hypercube.}
\label{fig:build-hyperboxes-quickly-2}
\end{figure}

\subsection{Well-definedness of the involutive dual knot mapping cone formula}

We now prove that Theorem~\ref{thm:surgery-hypercube} provides a uniquely-specified model of the dual knot complex when $K\subset S^3$. We observe firstly that the constructions from Sections~\ref{sec:construct-iota_L} and ~\ref{sec:Transferring-to-small} can be applied to abstract $\iota_K$ complexes of L-space type with a single tower. In particular, we can pick an algebraic flip map $\frF_n$ and a homotopy $H^{\{*\}}$ and transfer these to the small model $\bX^\mu_n(\scC)$ algebraically to obtain an algebraic involutive dual knot formula $\bXI_n^\mu(\scC)$, analogous to Equation~\eqref{eq:iotacomplex}. We now show that this construction is independent of the choices of $\frF_n$ and $H^{\{*\}}$:

\begin{prop}\label{prop:well-defined}
 Suppose that $\scC_1$ and $\scC_2$ are two $\iota_K$-complexes of $L$-space type with a single tower, and let $\XI_n^{\mu}(\scC_1)$ and $\XI_n^{\mu}(\scC_2)$ be two knot-like complexes with endomorphisms $\iota_{\bX,i}$ constructed by picking flip maps $\frF_{n,i}^\mu$ and homotopies $H^{\{*\}}_{n,i}$ for $\scC_1$ and $\scC_2$. Suppose that $F\colon \scC_1\to \scC_2$ is an $\iota_K$-local map. Then there is an induced map 
 \[
 \XI_n^{\mu}(F)\colon \bXI_n^{\mu}(\scC_1)\to \bXI_n^{\mu}(\scC_2),
 \]
 which commutes with $\iota_{\bX,i}$ up to $\bF[\scU,\scV]$ skew-equivariant chain homotopy. Furthermore,
 \begin{enumerate}
 \item The construction of $\XI_n^\mu(F)$ gives a map which is well-defined up to $\iota_K$-chain homotopy.
 \item The construction is functorial in the sense that $\XI_n^{\mu}(G\circ F)\simeq \XI_n^{\mu}(G)\circ \XI_n^\mu(F)$.
 \item $\XI_n^{\mu}(F)$ is a homotopy equivalence if $F$ is.
 \item $\XI_n^{\mu}(F)$ is local if and only if $F$ is.
 \end{enumerate}
\end{prop}

\begin{proof} The proof is similar to the proofs of Proposition~\ref{prop:properties-F-X} and Theorem~\ref{thm:surgery-hypercube}.

As a first step, we claim that we may indeed always construct the model in Theorem~\ref{thm:surgery-hypercube} for any $\iota_K$-complex of L-space type with a single tower. To see this, we observe that if $(\scC,\iota_K)$ is an $\iota_K$-complex of L-space type with a single tower, then the maps for $\iota_K^\mu$, $\Omega^\mu$, $\tilde{H}_{\Omega}$ and $H_{\Omega}$ are all given by concrete formulas. Furthermore, the maps $I_{\cA}$, $H_{\cB}$, $\Pi_{\cB}$ (and so forth) used in the homological perturbation argument are also given by concrete formulas. By virtue of being of L-space type with a single tower, there always exists a choice of flip map $\frF_n^\mu$. It remains to verify that there exists a map $H^{\{*\}}\colon \tilde \cB_s^\mu\to \cB^\mu_{-s-1}$ so that the diagram is a hypercube of chain complexes. For this claim, it suffices to show (as claimed in Theorem~\ref{thm:surgery-hypercube}) that 
\begin{equation}
\begin{split}
&\frF_n^\mu(\id+\Omega^\mu) \frF_n^\mu(v^\mu+\frF^\mu_n\tilde v^\mu)+(v^\mu+\frF_n^\mu \tilde v^{\mu})(\id+\Omega^\mu)\iota_K^\mu\\
+&\left(\frF_n^\mu \iota_K^\mu \frF_n^\mu+\iota_K^\mu+\frF_n^\mu \Omega^\mu \iota_K^\mu \frF_n^\mu+\Omega^\mu \iota_K^\mu\right)\tilde v^\mu\\
=&[\d^\mu, \frF^\mu_n \tilde H_\Omega\iota_K^\mu+H_\Omega \iota_K^\mu].
\end{split}
\label{eq:algebraic-model-involution-is-valid}
\end{equation}
This is a straightforward computation from the proof of Theorem~\ref{thm:surgery-hypercube}, which we leave to the reader.

If $(F,h)\colon (\scC_1,\iota_1)\to (\scC_2,\iota_2)$ is a morphism of $\iota_K$-complexes of $L$-space type with a single tower, then we may consider the $\bF[\scU,\scV]$-complex $\scC':=\Cone(F\colon \scC_1\to \scC_2)$ equipped with the endomorphism $J'=\iota_1+\iota_2+h$. 

We may formally define a flip-map $\frF'$ on $\scC'$ to be $\frF_1$ on $\scC_1$, $\frF_2$ on $\scC_2$ and also a choice of homotopy between $F\circ \frF_1$ and $\frF_2\circ F$. Such a homotopy exists because $\scC_1$ and $\scC_2$ are of L-space type with a single tower. 

The same construction as for ordinary knot complexes gives maps $\Omega^\mu$, $H_\Omega$ and $\tilde H_{\Omega}$ on $\scC'$ which we can use to partially construct the surgery hypercube. This does not give a fully valid hypercube. Instead, we know that the hypercube relations are satisfied on faces of the surgery hypercube which do not increment the surgery direction of the cube (i.e. the one with $v^\mu$ and $\frF^\mu \tilde v^\mu$). More generally, for subcubes which do increment the surgery direction, the same algebraic argument which gives Equation~\eqref{eq:algebraic-model-involution-is-valid} shows that the sum of the terms in the hypercube relation factors through $\tilde v^\mu$ and $v^\mu$. Hence the standard filling procedure can be inductively used to construct additional chains of the hypercube (added in the surgery direction) which factor through $v^\mu$ and $\tilde v^\mu$ so that the hypercube relations are satisfied.

An easy extension of this argument shows that the map $\XI_n^\mu(F)$ is well-defined up to $\iota_K$-chain homotopy. The remaining claims are also straightforward to verify.
\end{proof}

Applying the above result to the case when $F=\id$, we obtain the following.

\begin{cor} Any two choices of auxiliary data $\frF_n^\mu$ and $H_n^{\{*\}}$ give homotopy equivalent $\iota_K$-complexes $\XI_n^\mu(\scC_1)$ and $\XI_n^{\mu}(\scC_2)$.
\end{cor}

\begin{rem} We expect that one can also show via purely algebraic methods that each $(\bXI_n^{\mu}(\scC_i),\iota_{\bX,i})$ is an $\iota_K$-complex by adapting the techniques of \cite{HHSZ}*{Section~3.6}, however we will not pursue this line of reasoning.
\end{rem}

We may now conclude the proof of Theorem~\ref{thm:main}.

\begin{proof}[Proof of Theorem~\ref{thm:main}] It follows from Section~\ref{sec:small-model-mapping-cone} that the chain complex $\bX_n^\mu(K)$ is chain homotopy equivalent to $\cCFK(S^3_n(K))$ and from Proposition~\ref{prop:properties-F-X} that this chain complex does not depend up to chain homotopy equivalence on the choices of flip maps and homotopies involved in its definition. Then by Theorem~\ref{thm:surgery-hypercube}, there is an $\iota_K$-homotopy equivalence between the $\iota_K$-complex $\XI_n^{\mu}(K)$ and $(\cCFK(S^3_n(K)),\iota_{\mu})$, and it follows from Proposition \ref{prop:well-defined} that the construction of $\XI_n^\mu(K)$ is independent of the choices of flip map and homotopies made in its construction up to $\iota_K$-homotopy equivalence. \end{proof}
\section{Computing the small model}
\label{sec:formulas}

In this section, we give explicit formulas for the maps appearing in Theorem~\ref{thm:surgery-hypercube}. We also describe several important properties, such as the Maslov and Alexander gradings on the small model.

\subsection{The map $\iota_K^\mu$}

We now give an explicit description of the map $\iota_K^\mu$.

\begin{lem}\label{lem:iotaKmu}On $\cA_s^\mu(K)$, $\cB_s^\mu(K)$ and $\tilde{\cB}_s^\mu(K)$, the map $\iota_K^\mu$ is equal to the original map $\iota_K$, but with powers of $\scU$ and $\scV$ changed. In more detail:
\begin{enumerate}
\item If $x=\scU_1^i \cdot \xs\in \cA_{s+1}(K)/\scV_1 \cA_s(K)$ and $\scU_1^m \scV_1^n\cdot \ys \in \cA_{-s-1}(K)$ is a summand of  $\iota_K(x)$, then 
\begin{enumerate}
\item If $m\le n$, then $\scU_2^m\scV_2^m \cdot (\scV_1^{n-m} \cdot \ys)\in (\cA_{-s-1}(K)/ \scU_1\cA_{-s}(K))[\scU_2, \scV_2]$ is a summand of $\iota_K^\mu(x)$.
\item If $m>n$, then $\scU_2^n \scV_2^{n+1} (\scU_1^{m-n-1} \cdot \ys)\in (\cA_{-s}(K)/ \scV_1 \cA_{-s-1}(K))[\scU_2, \scV_2]$ will be a summand of $\iota_K^\mu(x)$.
\end{enumerate}
\item If $x=\scV_1^i \cdot \xs\in \cA_s(K)/\scU_1\cA_{s+1}(K)$ and $\scU_1^m\scV_1^n \cdot \ys\in \cA_{-s}(K)$ is a summand of $\iota_K(x)$, \begin{enumerate}
\item If $m<n$, then $\iota_K^\mu(x)$ has a summand of $\scU_2^{m+1} \scV_2^{m} (\scV_1^{n-m-1}\cdot \ys)\in (\cA_{-s-1}(K)/ \scU_1 \cA_{-s}(K))[\scU_2, \scV_2]$.
\item If $m\ge n$, then $\iota_K^\mu(x)$ has a summand of $\scU_2^n \scV_2^n (\scU_1^{m-n} \cdot \ys)\in (\cA_{-s}(K)/\scV_1 \cA_{-s-1}(K))[\scU_2, \scV_2]$.
\end{enumerate}
\item If $x=\scV_1^i \cdot \xs\in \cB_s(K)/\scU_1 \cB_{s+1}(K)$ and $\scU_1^m \scV_1^n \cdot \ys$ is a summand of $\iota_K(x)$, then $\scU_2^m\scV_2^m\cdot (\scV_1^{n-m}\cdot \ys)$ is a summand of $\iota_K^\mu(x)$.

\item If $x=\scU_1^i \cdot \xs \in \widetilde{\cB}_{s+1}(K)/\scV_1\widetilde{\cB}_s$ and $\scU_1^m \scV_1^n \cdot \ys$ is a summand of $\iota_K(x)$, then $\scU_2^n\scV_2^n \cdot (\scU_1^{m-n} \cdot \ys)$ is a summand of $\iota_K^{\mu}(x)$.

\end{enumerate}
\end{lem}

\begin{proof}
We prove the first formula and leave the rest to the reader, since the computations are basically the same. Recall
\[
\iota_K^\mu=\Pi_{\cA}\circ (\iota_K\otimes \iota_H)\circ  I_{\cA}.
\]
Recall that we refer to mapping cone domain summands of the module $M_s$ from Equation~\eqref{eq:module-M-s} as \emph{odd} summands, and we refer to codomain summands of $M_s$ as \emph{even} summands. The map has an even term and an odd term.  The even term maps into $\cA_{s+1}(K)$ by $I_{\cA}$ via the identity map. The odd term makes no total contribution, since $(\iota_K\otimes \iota_H)$ preserves the even/odd decomposition, and $\Pi_{\cA}$ vanishes on odd terms. The even term of $\Pi_{\cA}$ involves no higher maps from homological perturbation, since the exterior algebra action of $\Lambda$ vanishes on the even summands. We apply $\iota_K\otimes \iota_H$ to obtain a $\scU_1^m\scV_1^n\cdot \ys$ in the even copy of $\cA_{-s-1}(K)$. We now apply $\Pi_{\cA}$, which in this case only amounts to transforming powers of $\scU_1$ and $\scV_1$ into powers of $\scU_2$ and $\scV_2$, following the recipe given by the homological perturbation lemma. The result follows.
\end{proof}

\subsection{The map $\Omega^\mu$}
\label{sec:algebra-Omega}

We now consider the map $\Omega^\mu$, which we recall is defined on $\bA$ as
\[
\Pi_{\cA}\circ (\Phi_{K,1}\otimes \Psi_{H,1})\circ I_{\cA}.
\]
By inspection, $\Psi_{H,1}$ is the $\bF[\scU_1,\scV_1,\scU_2,\scV_2]$-equivariant chain map on the Hopf link complex which sends $\ve{a}$ to $\ve{c}$.

It is helpful to manipulate $\Omega^\mu$ slightly. We claim
\begin{equation}
\begin{split} &\Pi_{\cA}\circ (\Phi_{K,1}\otimes \Psi_{H,1})\circ I_{\cA}\\
=&\Pi_{\cA}\circ (\Phi_{K,1}\otimes \id) \circ (\id\otimes \Psi_{H,1})\circ  I_{\cA}\\
=&\Pi_{\cA}\circ (\Phi_{K,1}\otimes \id)\circ  I_{\cA} \circ \Pi_{\cA}\circ  (\id\otimes \Psi_{H,1})\circ  I_{\cA}.
\end{split}
\label{eq:secretly-split}
\end{equation}
Here, we are viewing $(\id \otimes \Psi_{H,1})$ as the ``identity'' map from $\cA_s(K)\otimes \ve{a}\subset \cA_{s}(L)$ to $\cA_s(K)\otimes \ve{c}\subset\cA_{s-1}(L)$. Here, $\Phi_{K,1}\otimes \id$ is the map from $\cA_{s-1}(L)$ to $\cA_{s}(L)$ which is $\Phi_{K,1}$ on each summand.

The first equality in~\eqref{eq:secretly-split} obvious. The second equality is less obvious. There is a canonical homotopy between the two lines, which is given by
\[
\Pi_{\cA}\circ (\Phi_{K,1}\otimes \id)\circ H_{\cA} \circ (\id\otimes \Psi_{H,1})\circ  I_{\cA}.
\]
However this homotopy vanishes because $\Pi_{\cA}\circ (\Phi_{K,1}\otimes \id) \circ H_{\cA}=0$, since $H_{\cA}$ has image only in odd summands, $\Phi_{K,1}\otimes \id$ preserves even/odd summands, and $\Pi_{\cA}$ vanishes on odd summands.  The same holds on the $\cB$ versions.
In particular, we obtain 
\begin{equation}\label{eqn:Omegamu}
\Omega^\mu=\Phi_K^\mu\circ \Psi_{H,1}^\mu,
\end{equation}
where 
\[
\Phi_K^\mu=\Pi_{\cA}\circ (\Phi_{K,1}\otimes \id) \circ  I_{\cA}\quad \text{and} \quad \Psi_{H,1}^\mu=\Pi_{\cA}\circ (\id \otimes \Psi_{H,1})\circ  I_{\cA}.
\]

In the following lemma, we compute the map
\[
\Phi_K^{\mu}\colon \cA_{s-1}^\mu(K)\to \cA_{s}^{\mu}(K),
\]
as well as its analogs on $\cB_{s-1}^{\mu}$ and $\tilde{\cB}_{s-1}^{\mu}$.

\begin{lem}\label{lem:PhiKmu}  On $\cA_{s-1}^\mu(K)$, $\cB_{s-1}^\mu(K)$ and $\tilde{\cB}_{s-1}^\mu(K)$, the map $\Phi_K^\mu$ is gotten by applying the $\Phi_{K,1}$-map of $\cCFK(K)$, and then changing powers of the $\scU$ and $\scV$ variables.  In more detail:
\begin{enumerate}
\item Suppose $x=\scU_1^i \cdot \xs\in \cA_{s}(K)/\scV_1 \cA_{s-1}(K)$ and suppose that $\scU_1^m \scV_1^n \cdot \ys$ is a summand of $\d(x)$, where $\d$ is the differential of $\cCFK(K)$. 
\begin{enumerate} 
\item If $n\le m-1$, then $(m-i) \scU_2^n \scV_2^n\cdot (\scU_1^{m-n-1} \cdot \ys)\in \cA_{s+1}(K)/ \scV_1 \cA_s(K)$ is a summand of $\Phi_K^\mu(x)$.
\item If $n> m-1$, then $(m-i) \scU_2^{m-1} \scV_2^{m-1} (\scV_1^{n-m+1} \cdot \ys)\in \cA_{s}(K)/ \scU_1 \cA_{s+1}(K)$ is a summand of $\Phi_K^\mu(x)$.
\end{enumerate}
\item Suppose $x=\scV_1^i \cdot \xs\in \cA_{s-1}(K)/\scU_1 \cA_{s}(K)$ and suppose that $\scU_1^m \scV_1^n \cdot \ys$ is a summand of $\d(x)$.
\begin{enumerate}
\item If $m-1\le n$, then $m \scU_2^{m-1} \scV_2^{m-1} (\scV_1^{n-m+1} \cdot \ys)\in \cA_{s}(K)/ \scU_1\cA_{s+1}(K)$ is a summand of $\Phi_{K}^\mu(x)$. 
\item If $m-1>n$, then $m \scU_2^{m-1} \scV_2^{m}(\scU_1^{m-n-1} \cdot \ys)\in \cA_{s+1}(K)/\scV_1 \cA_s(K)$ is a summand of $\Phi_{K}^\mu(x)$.  
\end{enumerate}
\item Suppose $x\in \scV_1^i \cdot \xs\in \cB_{s}(K)/ \scU_1 \cB_{s+1}(K)$ and $\scU_1^m \scV_1^n\cdot \ys$ is a summand of $\d (x)$. Then $m \scU_2^{m-1} \scV_2^{m-1}\cdot (\scV_1^{n-m+1}\cdot \ys)$ is a summand of $\Phi_{K}^\mu(x)$.

\end{enumerate}
\end{lem}
\begin{proof}
By definition, $\Phi_{K}^\mu$ coincides with the map $\Pi_{\cA}\circ (\Phi_{K,1}\otimes \id)\circ  I_{\cA}$. The map $ I_{\cA}$ has terms in the even and odd summands. The terms in the even summands are essentially just the inclusion maps (involving no trees or $m_2$ maps), since the action of the exterior algebra $\Lambda$ on the even summands vanishes. These contribute exactly the stated quantities (i.e. we then apply $\Phi_{K,1}\otimes \id$, and then apply $\Pi_{\cA}$). The components of $ I_{\cA}$ in the odd summands makes no contribution, because $\Phi_{K,1}\otimes \id$ preserves the $M$-summands, and $\Pi_{\cA}$ vanishes on the odd summands.
\end{proof}

\begin{rem}
 Note that $\Phi_{K}^\mu$ commutes with $v^\mu$ and $\tilde{v}^\mu$. This may either be seen directly, or via the fact that the canonical null-homotopy of $[\Phi_{K}^\mu,v^{\mu}]$ given by the homological perturbation lemma vanishes.
\end{rem}

\begin{lem}\label{lem:Psi-mu-map} The maps \[
\Psi^\mu_{H,1}\colon \cA_{s}^\mu(K)\to \cA_{s-1}^\mu(K), \quad \Psi^\mu_{H,1}\colon \cB_s^{\mu}(K)\to \cB_{s-1}^\mu(K) \quad \text{and}\quad \Psi^\mu_{H,1}\colon \tilde{\cB}_{s}^\mu(K)\to \tilde{\cB}^\mu_{s-1}(K)
\]
 have the following form: 
\begin{enumerate}
\item Suppose $x=\scU_1^i \cdot \xs\in \cA_{s+1}(K)/\scV_1 \cA_{s}(K)$ and suppose that $\scU_1^m \scV_1^n \cdot \ys$ is a summand of $\d(x)$, where $\d$ is the differential of $\cCFK(K)$. 
\begin{enumerate} 
\item If $n\le m+1$, then there is a summand of \[n \scU_2^{n-1} \scV_2^{n-1}(\scU_1^{m-n+1} \cdot \ys)\in \cA_s(K)/ \scV_1 \cA_{s-1}(K)\] in $\Psi_{H,1}^{\mu}(x)$.
\item If $n> m+1$, then there is a summand of \[(m+1) \scU_2^{m+1} \scV_2^{m}(\scV_1^{n-m-2} \cdot \ys)\in \cA_{s-1}(K)/\scU_1 \cA_s(K)\] in $\Psi_{H,1}^{\mu}(x)$.
\end{enumerate}
\item Suppose $x=\scV_1^i \cdot \xs\in \cA_{s}(K)/\scU_1 \cA_{s+1}(K)$ and suppose that $\scU_1^m \scV_1^n \cdot \ys$ is a summand of $\d(x)$.
\begin{enumerate}
\item If $n\ge m+1$, then there is a summand of $m \scU_2^{m} \scV_2^{m} (\scV_1^{n-m-1} \cdot \ys)\in \cA_{s-1}(K)/\scU_1 \cA_s(K)$ in $\Psi_{H,1}^{\mu}(x)$.
\item If $n<m+1$, then there is a summand of $n \scU_2^{n-1} \scV_2^{n} (\scU_1^{m-n} \cdot \ys)\in \cA_s(K)/\scV_1\cA_{s-1}(K)$ in $\Psi^\mu_{H,1}(x)$.
\end{enumerate}
\item Suppose $x\in \scV_1^i \cdot \xs\in \cB_{s}(K)/ \scU_1 \cB_{s+1}(K)$ and $\scU_1^m \scV_1^n\cdot \ys$ is a summand of $\d (x)$.  Then there is a summand of $m\scV_1^{n-m-1}\cdot \ys$ in $\Psi_{H,1}^{\mu}(x)$.
\item Suppose $x=\scU_1^i \cdot\xs\in \tilde{\cB}_{s+1}(K)/\scV_1 \tilde{\cB}_s(K)$ and $\scU_1^m \scV_1^n \cdot \ys$ is a summand of $\d(x)$. Then $\Psi^\mu_{H,1}(x)$ has a coefficient of $n\scU_2^{n-1} \scV_2^{n-1}(\scU_1^{m-n+1}\cdot \xs)$.
\end{enumerate}
\end{lem}
\begin{proof} The map $\id\otimes \Psi_{H,1}$ has support on an odd summand, and maps into an even summand. Therefore the terms contributed by $ I_{\cA}$ in the even summands make no contribution. It remains to consider the terms contributed by $ I_{\cA}$ in the odd summands. These involve applying a differential to $x$, and then lowering the $\scV_1$ or $\scU_1$ power by 1, and then applying the homotopy and $m_2$ maps repeatedly. We then apply $\id\otimes \Psi_{H,1}$. This simply moves backwards along the $\scV_1$ arrow. We then apply $\Pi_{\cA}$, which has the effect of doing further loops until we cannot go any further. The only ambiguity in this process is at which point we travel backwards along the $\scV_1$ arrow (i.e. apply $\id \otimes\Psi_{H,1}$). It is straightforward to see that the stated formulas hold, via a case-by-case analysis. 
\end{proof}

\begin{rem} We see in Lemma~\ref{lem:Psi-mu-map} that $\Psi^\mu_{H,1}$ does not generally commute with $v^\mu$ and $\tilde{v}^\mu$. Nonetheless, they do commute up to chain homotopy, since the expanded versions $v_L$ and $\tilde{v}_L$ commute with $\id\otimes \Psi_{H,1}$ on the nose, and hence the homological perturbation lemma gives a canonical null-homotopy of both $[\Psi_{H,1}^\mu, v^\mu]$ and $[\Psi_{H,1}^\mu,\tilde v^\mu]$. 
\end{rem}

\subsection{The homotopies $H_{\Omega}$ and $\tilde{H}_{\Omega}$}

We now consider the homotopies $H_{\Omega}$ and $\tilde{H}_{\Omega}$ appearing in Theorem~\ref{thm:surgery-hypercube}.

\begin{lem}\label{lem:HOmega} The maps $H_{\Omega}$ and $\tilde{H}_{\Omega}$ satisfy the following:
\[
\begin{split}
H_{\Omega}&=\Phi_{K}^\mu \circ H_0, \quad \text{and} \\
\tilde{H}_{\Omega}&=\Phi_{K}^\mu\circ \tilde{H}_0,
\end{split}
\]
where:
\begin{enumerate}
\item $H_0\colon \cA_s^\mu(K)\to \cB_{s-1}^\mu(K)$ takes the following form:
\begin{enumerate}
\item If $x=\scU_1^i \cdot \xs\in \cA_{s+1}/\scV_1 \cA_s$, then 
\[
H_0(x)=(i+1)\scU_2^{i+1} \scV_2^{i} (\scV_1^{-i-2}\cdot \xs).
\]
\item If $\scV_1^i\cdot \xs\in \cA_s/\scU_1\cA_{s+1}$, then $H_0(x)=0$.
\end{enumerate}
\item $\tilde{H}_0\colon \cA_s^\mu(K)\to \tilde{\cB}_{s-1}^\mu(K)$ takes the following form:
\begin{enumerate}
\item If $x=\scU_1^i \cdot \xs\in \cA_{s+1}/\scV_1 \cA_s$, then $\tilde{H}_0(x)=0$.
\item If $x=\scV_1^i\cdot \xs\in \cA_{s}/ \scU_1 \cA_{s+1}$, then 
\[
\tilde{H}_0(x)=i \scU_2^i \scV_2^{i+1} (\scU_1^{-i-1} \cdot \xs).
\]
\end{enumerate}
\end{enumerate}
\end{lem}
\begin{proof}
We recall that $H_{\Omega}$ and $\tilde{H}_{\Omega}$ are given by the formulas
\[
\begin{split}
H_\Omega&=\Pi_{\tilde{\cB}} (\Phi_{K,1} \otimes  \Psi_{H,1}) H_{\cB} v  I_{\cA}+\Pi_{\tilde{\cB}} \tilde{v} H_{\cA} (\Phi_{K,1} \otimes  \Psi_{H,1})  I_{\cA}\\
\tilde{H}_\Omega&=\Pi_{\cB} (\Phi_{K,1} \otimes \Psi_{H,1}) H_{\tilde{\cB}} \tilde{v}  I_{\cA}+\Pi_{\cB} v H_{\cA} (\Phi_{K,1} \otimes \Psi_{H,1})  I_{\cA}.
\end{split}
\]

As a first step, we note
\[
H_{\Omega}=\Pi_{\tilde{\cB}} (\Phi_{K,1} \otimes \Psi_{H,1}) H_{\cB} v  I_{\cA}\qquad \text{and} \qquad \tilde{H}_{\Omega}= \Pi_{\cB} (\Phi_{K,1} \otimes \Psi_{H,1}) H_{\tilde{\cB}} \tilde{v}  I_{\cA},
\]
since $\Pi_{\tilde{\cB}} \tilde{v} H_{\cA}=0$ and $\Pi_{\cB} v H_{\cA}=0$. These latter two equalities follow from the fact that $H_{\cA}$ and $H_{\tilde{\cB}}$ have image in the odd summands, while $v$ and $\tilde{v}$ preserve the even and odd summands, and $\Pi_{\tilde{\cB}}$ and $\Pi_{\cB}$ vanish on the odd summands. 

Next, the same argument giving equation~\eqref{eq:secretly-split} shows that
\[
H_{\Omega}=\Phi_{K}^\mu\circ \Pi_{\cB}\circ  (\id \otimes \Psi_{H,1}) \circ H_{\cB} \circ v\circ  I_{\cA}\qquad \text{and} \qquad \tilde{H}_{\Omega}=\Phi_{K,1}^\mu \circ \Pi_{\cB} \circ (\id \otimes \Psi_{H,1})\circ H_{\tilde{\cB}}\circ \tilde{v}\circ   I_{\cA}.
\]

We define
\[
H_0=\Pi_{\cB}\circ (\id \otimes \Psi_{H,1})\circ H_{\cB}\circ v\circ  I_{\cA}\qquad \text{and} \qquad \tilde{H}_0=\Pi_{\cB}\circ (\id \otimes \Psi_{H,1})\circ H_{\tilde{\cB}} \circ \tilde{v} \circ  I_{\cA},
\]
which we now show have the form stated in the lemma. We first consider $H_0$. Suppose $x=\scU_1^i\cdot \xs \cA_{s+1}/\scV_1 \cA_s$. Since $H_{\cB}$ vanishes on odd summands of $\cA_{s}(L)$, the only term of $ I_{\cA}$ which contributes is the canonical inclusion, with no applications of $m_2$ or $H$. We then apply $v$, $H_{\cB}$, $\id \otimes \Psi_{H,1}$, and then $\Pi_{\cB}$. This amounts to including into $\cB_{s}(L)$, and then traveling clockwise around $\cB_{s}(L)$, changing powers of variables, until we reach an element of the form $\scU_2^n \scV^m_2\cdot (\scU_1^{-j} \cdot \xs)\in (\cB_{s}/ \scU_1 \cB_{s+1})(K)$, for some $n,m,j\ge 0$. The only ambiguity is at what point of the cycle we apply the map $\id \otimes \Psi_{H,1}$. We can enumerate these explicitly as follows. Each summand of $H_{\cB}$ must have an initial $H$ which sends $\scU_1^i \cdot \xs$ to $\scU_1^i \scV_1^{-1} \cdot \xs$, in the top left corner of $\cA_{s}(L)$. Then, $H_{\cB}$ does $k$ clockwise loops around $\cB_{s}(L)$, outputting the elements $\scU_2^k\scV_2^k(\scU_1^{i-k}\scV_1^{-1-k})$. Then we apply $\id\otimes \Psi_{H,1}$, which simply sends this to the inclusion of itself, in the bottom left corner of $\cB_{s}(L)$. Then we finish with $\Pi_{\cB}$, which outputs $\scU_2^{i+1} \scV_2^i( \scV^{-i-2}\cdot \xs)$. Here, $k$ can be any integer in $\{0,\dots, i\}$, so we have an overall factor of $i$.

We consider now the map $H_0$ applied to an element $x=\scV_1^i\cdot \xs\in \cA_s/\scU_1\cA_{s+1}$. We first apply $ I_{\cA}$. As before, the only component with contributing output is the component in the top left of $\cA_{s}(L)$, which involves no higher towers from the homological perturbation lemma. We include into $\cB_{s}(L)$ with $v$. However, the map $H_{\cB}$ vanishes on this component, since $x$ already has 0 $\scU_1$-power.

The argument for $\tilde{H}_0$ is only a simple modification.
\end{proof}

\subsection{Gradings on the small model}

We can now transfer gradings from the expanded model to the small model. Concretely, this amounts to combining the gradings from~\eqref{eq:gradings} with the gradings of the Hopf link. We set
\[
\cA_s^\mu(K)=\bigg(\cA_{s+1}(K)/\scV_1 \cA_s(K)[-2,0]\oplus \cA_{s}(K)/\scU_1\cA_{s+1}(K)[0,-2]\bigg)[\scU_2,\scV_2],
\]
and
\[
\cB_s^\mu(K)=\bigg(\cB_s(K)/\scU_1 \cB_{s+1}(K)[0,-2]\bigg)[\scU_2,\scV_2].
\]
With these definitions, we set
\[
\begin{split}
\bA^\mu(K)&=\prod_{s\in \Z} \cA_s^\mu(K)\left[\frac{(n-2s)^2/n+\epsilon(n)}{4}, \frac{(n+2(s+1))^2/n+\epsilon(n)}{4}\right]\\
\bB^\mu(K)&=\prod_{s\in \Z} \cB_s^\mu(K)\left[\frac{(n-2s)^2/n+\epsilon(n)}{4}-1, \frac{(n+2(s+1))^2/n+\epsilon(n)}{4}-1\right].
\end{split}
\]
In both equations, $[i,j]$ denotes a shift in the $(\gr_w,\gr_z)$-bigrading, and $\epsilon(n)=-1$ if $n>0$ and $\epsilon(n)=5$ if $n<0$.

\section{An example}
\label{sec:example}
In this section, we use the small model to compute the $\iota_K$-complex of $(Y, \mu)$, where $Y=S^3_{1}(4_1)$ and $\mu$ denotes the image of the meridian of $4_1$ after surgery.

Recall from the proof of Lemma \ref{lem:v-maps} that 
\[
\cA_s^{\mu}(K)\iso \bigg( (\cA_{s+1}/\scV_1 \cA_s)(K)\oplus \cA_s/\scU_1 \cA_{s+1}(K)\bigg)[\scU_2,\scV_2] \quad \text{and} \quad \cB_s^{\mu}(K)\iso (\cB_s/\scU_1 \cB_{s+1})(K)[\scU_2,\scV_2].
\]
As with the Ozsv\'ath-Szab\'o knot surgery mapping cone \cite{OSIntegerSurgeries}, we may truncate horizontally, once $v^\mu$, respectively $\frF_1^\mu \tilde{v}^\mu$, has become an isomorphism; in the present case, the reader may verify that after truncation, we are left with 
\[ \bigoplus_{s=-g}^{g-1} \cA_s^\mu(K) \oplus \bigoplus_{s=-g+1}^{g-1} \cB_s^\mu(K), \]
where $g$ denotes the genus of $K$.

Let $K$ be the figure-eight knot $4_1$, and let $\cCFK(S^3, K)$ be generated by $a, b, c, d,$ and $e$ with
\begin{align*}
	\d a &= \scV_1 d & \iota_K(a) &= e \\
	\d b &= \scU_1 a + \scV_1 e & \iota_K(b) &= b+c \\
	\d c &= 0 & \iota_K(c) &= c+d \\
	\d d &= 0 & \iota_K(d) &= d \\
	\d e &= \scU_1 d & \iota_K(e) &= a.
\end{align*}
See Figure \ref{fig:figure-eight}. We will take $\frF^\mu_1$ to be given by
\[ \frF^\mu_1(a) = e \qquad \frF^\mu_1(b) = b \qquad \frF^\mu_1(c) = c \qquad \frF^\mu_1 (d) = d \qquad \frF^\mu_1(e) = a. \]

\begin{figure}[ht]
\begin{tikzpicture}[scale=1]
	\filldraw (3.8, 3.8) circle (2pt) node[label=above:{\lab{c}}] (a) {};
	\filldraw (3.5, 3.5) circle (2pt) node[label=above:{\lab{b}}] (b) {};
	\filldraw (2.5, 3.5) circle (2pt) node[label=above :{\lab{\scU_1 a}}] (c) {};
	\filldraw (3.5, 2.5) circle (2pt) node[label=right:{\lab{\scV_1 e}}] (d) {};
	\filldraw (2.5, 2.5) circle (2pt) node[label=left:{\lab{\scU_1 \scV_1 d}}] (e) {};

	\draw[->] (b) to (c);
	\draw[->] (b) to (d);
	\draw[->] (c) to (e);
	\draw[->] (d) to (e);
	
\end{tikzpicture}
\caption{A graphical depiction of $\cCFK(S^3, 4_1)$.}
\label{fig:figure-eight}
\end{figure}

Using the aforementioned descriptions of $\cA_s^{\mu}(K)$ and $\cB_s^{\mu}(K)$, we have that
\begin{align*}
	\cA_{-1}^\mu(K) &= \langle \scU_1 a_{-1}, b_{-1}, c_{-1}, d_{-1}; e_{-1} \rangle_{\bF[\scU_2, \scV_2]} \\
	\tilde{\cB}^\mu_{-1}(K) &= \langle \scU_1 \tilde{a}'_{-1};  \tilde{b}'_{-1}, \tilde{c}'_{-1}, \tilde{d}'_{-1},  \scU_1^{-1} \tilde{e}'_{-1} \rangle_{\bF[\scU_2, \scV_2]}\\
	\cA_{0}^\mu(K) &= \langle a_{0};  b_{0}, c_{0}, d_{0},  \scV_1 e_{0} \rangle_{\bF[\scU_2, \scV_2]}\\
	\cB_{0}^\mu(K) &= \langle \scV_1^{-1} a'_{0}, b'_{0}, c'_{0},  d'_{0}, \scV_1 e'_{0} \rangle_{\bF[\scU_2, \scV_2]},
\end{align*}
where the semi-colon in $\cA_s^\mu$ separates the generators in $(\cA_{s+1}/\scV_1 \cA_s)(K)$ and $(\cA_s/\scU_1 \cA_{s+1})(K)$, the subscripts on the generators of $\cA_s^\mu$ denote $s$, and we use the marker $'$ to distinguish generators of $\cB_s^\mu$ from those of $\cA_s^\mu$. (Similar notation was utilized in \cite[Section 6]{HLL}.)

\begin{figure}[htb!]
\subfigure[]{
\begin{tikzpicture}[scale=1]

\filldraw[black!10!white] (0, -4) rectangle (1, -1);
\filldraw[black!25!white] (-3, -1) rectangle (1, 0);

	\begin{scope}[thin, black!20!white]
		\draw [<->] (-3, 0.5) -- (2, 0.5);
		\draw [<->] (0.5, -4) -- (0.5, 2);
	\end{scope}
	\draw[step=1, black!50!white, very thin] (-2.9, -3.9) grid (1.9, 1.9);

	\begin{scope}[black!20!white]
	\foreach \x in {-3,...,0}
	{	
		\filldraw (\x+0.65, \x+0.35) circle (2pt) node[] (a){};
		\filldraw (\x+1.35, \x+0.35) circle (2pt) node[] (b){};
		\filldraw (\x+1.5, \x+0.5) circle (2pt) node[] (c){};
		\filldraw (\x+0.65, \x-.35) circle (2pt) node[] (d){};
		\filldraw (\x+1.35, \x-.35) circle (2pt) node[] (e){};
		\draw [thick, ->] (a) -- (d);
		\draw [thick, ->] (e) -- (d);
		\draw [thick, ->] (b) -- (a);
		\draw [thick, ->] (b) -- (e);	
	}
	\end{scope}

	\filldraw (-0.35, -0.65) circle (2pt) node[left] (a){\small$\scU_1 a_{-1}$};
	\filldraw (0.35, -0.65) circle (2pt) node[right] (b){\scriptsize$b_{-1}$};
	\filldraw (0.5, -0.5) circle (2pt) node[right] (c){\scriptsize$c_{-1}$};
	\filldraw (0.65, -0.35) circle (2pt) node[right] (d){\scriptsize$d_{-1}$};
	\filldraw (0.35, -1.35) circle (2pt) node[right] (e){\small$e_{-1}$};

	\node[] at (-1.5, -0.2) {\small$(\cA_0/\scV_1 \cA_{-1})(K)$};
	\node[] at (0.7, -2.7) {\small$(\cA_{-1}/\scU_1 \cA_{0})(K)$};

\end{tikzpicture}
\label{subfig:A1}
}
\hspace{10pt}
\subfigure[]{
\begin{tikzpicture}[scale=1]

\filldraw[black!10!white] (0, -4) rectangle (1, 0);
\filldraw[black!25!white] (-3, 0) rectangle (1, 1);

	\begin{scope}[thin, black!20!white]
		\draw [<->] (-3, 0.5) -- (2, 0.5);
		\draw [<->] (0.5, -4) -- (0.5, 2);
	\end{scope}
	\draw[step=1, black!50!white, very thin] (-2.9, -3.9) grid (1.9, 1.9);

	\begin{scope}[black!20!white]
	\foreach \x in {-3,...,0}
	{	
		\filldraw (\x+0.65, \x+0.35) circle (2pt) node[] (a){};
		\filldraw (\x+1.35, \x+0.35) circle (2pt) node[] (b){};
		\filldraw (\x+1.5, \x+0.5) circle (2pt) node[] (c){};
		\filldraw (\x+0.65, \x-.35) circle (2pt) node[] (d){};
		\filldraw (\x+1.35, \x-.35) circle (2pt) node[] (e){};
		\draw [thick, ->] (a) -- (d);
		\draw [thick, ->] (e) -- (d);
		\draw [thick, ->] (b) -- (a);
		\draw [thick, ->] (b) -- (e);	
	}
	\end{scope}

	\filldraw (0.65, 0.35) circle (2pt) node[above] (a){\small$a_0$};
	\filldraw (0.35, -0.65) circle (2pt) node[right] (b){\scriptsize$b_0$};
	\filldraw (0.5, -0.5) circle (2pt) node[right] (c){\scriptsize$c_0$};
	\filldraw (0.65, -0.35) circle (2pt) node[right] (d){\scriptsize$d_0$};
	\filldraw (0.35, -1.35) circle (2pt) node[right] (e){\small$\scV_1 e_0$};

	\node[] at (-1.5, 0.5) {\small$(\cA_1/\scV_1 \cA_{0})(K)$};
	\node[] at (0.7, -2.2) {\small$(\cA_{0}/\scU_1 \cA_{1})(K)$};

\end{tikzpicture}\label{subfig:A2}
}
\subfigure[]{
\begin{tikzpicture}[scale=1]

\filldraw[black!10!white] (0, -4) rectangle (1, 2);

	\begin{scope}[thin, black!20!white]
		\draw [<->] (-3, 0.5) -- (2, 0.5);
		\draw [<->] (0.5, -4) -- (0.5, 2);
	\end{scope}
	\draw[step=1, black!50!white, very thin] (-2.9, -3.9) grid (1.9, 1.9);

	\begin{scope}[black!20!white]
	\foreach \x in {-3,...,0}
	{	
		\filldraw (\x+0.65, \x+0.35) circle (2pt) node[] (a){};
		\filldraw (\x+1.35, \x+0.35) circle (2pt) node[] (b){};
		\filldraw (\x+1.5, \x+0.5) circle (2pt) node[] (c){};
		\filldraw (\x+0.65, \x-.35) circle (2pt) node[] (d){};
		\filldraw (\x+1.35, \x-.35) circle (2pt) node[] (e){};
		\draw [thick, ->] (a) -- (d);
		\draw [thick, ->] (e) -- (d);
		\draw [thick, ->] (b) -- (a);
		\draw [thick, ->] (b) -- (e);	
	}
	\end{scope}

	\filldraw (0.65, 0.35) circle (2pt) node[above] (a){\small$\scV_1^{-1}a'_0$};
	\filldraw (0.35, -0.65) circle (2pt) node[below] (b){\scriptsize$b'_0$};
	\filldraw (0.5, -0.5) circle (2pt) node[right] (c){\scriptsize$c'_0$};
	\filldraw (0.65, -0.35) circle (2pt) node[above] (d){\scriptsize$d'_0$};
	\filldraw (0.35, -1.35) circle (2pt) node[right] (e){\small$\scV_1 e'_0$};

	\node[] at (0.7, -2.2) {\small$(\cB_{0}/\scU_1 \cB_{1})(K)$};
\end{tikzpicture}
\label{subfig:B2}
}
\caption{Clockwise from top left, the shaded regions depict the generators (over $\mathbb{F}[\scU_2, \scV_2]$) of $\cA^\mu_{-1}$, $\cA^\mu_0$, and $\cB^\mu_0$. In $\cA^\mu_{-1}$ and $\cA^\mu_0$, the darker region depicts $(\cA_{s+1}/\scV_1 \cA_s)(K)$ and the lighter region $(\cA_s/\scU_1 \cA_{s+1})(K)$. We have drawn things in a way that is meant to be evocative of \cite{HeddenLevineSurgery}.}
\label{fig:CFK_1core}
\end{figure}

The internal differential on $\cA_{s}^\mu(K)$ is described in \ref{small-del-1} and  \ref{small-del-2} of Lemma~\ref{lem:homological-perturbation-differential-A} in Section \ref{subsec:homperturbAs}. For our example, this yields
\begin{align*}
	\d^\mu (\scU_1 a_{-1}) &=\scU_2 \scV_2 d_{-1}   &\d^\mu a_0 &= \scU_2 d_0\ \\
	\d^\mu b_{-1}&=(\scU_1 a_{-1})+\scU_2 e_{-1}  &\d^\mu b_0 &= \scV_2 a_0 + (\scV_1 e_0)\\
	\d^\mu c_{-1} &= 0  &\d^\mu c_0 &= 0 \\
	\d^\mu d_{-1}&= 0  &\d^\mu d_0 &= 0 \\
	\d^\mu e_{-1}&=\scV_2 d_{-1} &\d^\mu (\scV_1 e_0) &= \scU_2 \scV_2 d_0.
\end{align*}

Similarly, the internal differential on $\cB_s^\mu(K)$ is described in Lemma \ref{lem:homperturbBs}, which for our example yields
\begin{align*}
	\d^\mu (\scV_1^{-1} a'_0) &= d'_0 \\
	\d^\mu b'_0 &= \scU_2 \scV_2 (\scV_1^{-1} a'_0) + (\scV_1 e'_0) \\
	\d^\mu c'_0 &= 0 \\
	\d^\mu d'_0 &= 0 \\
	\d^\mu (\scV_1 e'_0) &= \scU_2 \scV_2 d'_0.
\end{align*}

The map $v^\mu \co \cA_0^\mu(K) \to \cB_0^\mu(K)$, described in Lemma \ref{lem:v-maps}, is given by
\begin{align*}
	v^\mu(a_0) &= \scU_2 (\scV^{-1}_1 a'_0)\\
	v^\mu(b_0) &= b'_0 \\
	v^\mu(c_0) &= c'_0 \\
	v^\mu(d_0) &= d'_0 \\
	v^\mu(\scV_1 e_0) &= (\scV_1 e'_0) \\
\end{align*}
and the maps $\tilde{v}^\mu \co \cA_{-1}^\mu(K) \to \tilde\cB_{-1}^\mu(K)$ and $\frF_1 \circ \tilde{v}^\mu \co \cA_{-1}^\mu(K) \to \cB_0^\mu(K)$  are given by
\begin{align*}
	\tilde{v}^\mu(\scU_1 a_{-1}) &= (\scU_1 \tilde{a}'_{-1}) & \frF_1\circ \tilde{v}^\mu(\scU_1 a_{-1}) &= (\scV_1 e'_0) \\
	\tilde{v}^\mu(b_{-1}) &= \tilde{b}'_{-1}  & \frF_1\circ \tilde{v}^\mu( b_{-1}) &= b'_0 \\
	\tilde{v}^\mu(c_{-1}) &= \tilde{c}'_{-1} & \frF_1\circ \tilde{v}^\mu( c_{-1}) &= c'_0  \\
	\tilde{v}^\mu(d_{-1}) &= \tilde{d}'_{-1} & \frF_1\circ \tilde{v}^\mu( d_{-1}) &= d'_0  \\
	\tilde{v}^\mu(e_{-1}) &= \scV_2(\scU_1^{-1}\tilde{e}'_{-1}) & \frF_1\circ \tilde{v}^\mu(\scU_1 e_{-1}) &= \scU_2(\scV_1^{-1} a'_0).
\end{align*}

We now use Lemma \ref{lem:iotaKmu} to describe $\iota_K^\mu \co \cA^\mu_s \to \cA^\mu_{-s-1}$ and $\iota_K^\mu \co \cB^\mu_0 \to \tilde{\cB}^\mu_{-1}$: 
\begin{align*}
	\iota_K^\mu (\scU_1 a_{-1} ) &= \scV_1 e_0 & \iota_K^\mu (a_0) &= e_{-1} &\iota_K^\mu(\scV_1^{-1} a'_0) &= \scU_2^{-1} \scV_2^{-1} (\scV_1 \tilde{e}'_{-1}) \\
	\iota_K^\mu (b_{-1} ) &= b_0+c_0 & \iota_K^\mu (b_0) &= b_{-1}+c_{-1} &\iota_K^\mu(b'_0) &= \tilde{b}'_{-1}+\tilde{c}'_{-1} \\
	\iota_K^\mu (c_{-1}) &= c_0+d_0 & \iota_K^\mu (c_0) &= c_{-1}+d_{-1} &\iota_K^\mu(c'_0) &= \tilde{c}'_{-1}+ \tilde{d}'_{-1} \\
	\iota_K^\mu (d_{-1} ) &= d_0 & \iota_K^\mu (d_0) &= d_{-1} &\iota_K^\mu(d'_0) &= \tilde{d}'_{-1} \\
	\iota_K^\mu (e_{-1}) &=  a_0 & \iota_K^\mu (\scV_1 e_0 ) &=  (\scU_1 a_{-1}) &\iota_K^\mu(\scV_1 e'_0) &= \scU_2 \scV_2 (\scV_1^{-1} \tilde{a}'_{-1}).
\end{align*}

Next, we use Lemma \ref{lem:PhiKmu} to describe $\Phi_{K}^\mu \co \cA^\mu_{-1} \to \cA^\mu_0$:
\begin{align*}
	\Phi_{K}^\mu (\scU_1 a_{-1} ) &=  0 \\
	\Phi_{K}^\mu (b_{-1} ) &=  a_0  \\
	\Phi_{K}^\mu (c_{-1}) &=  0 \\
	\Phi_{K}^\mu (d_{-1} ) &= 0    \\
	\Phi_{K}^\mu (e_{-1}) &=  d_0   .
\end{align*}
Note that because of our truncation, we do not need to consider $\Phi_{K}^\mu \co \cA^\mu_0 \to \cA^\mu_1$ nor $\Phi_{K}^\mu \co \cB^\mu_0 \to \cB^\mu_1$. 

We now compute $\Psi^\mu_{H,1}$ using Lemma \ref{lem:Psi-mu-map}. Because of the truncation, we only need to consider $\Psi^\mu_{H,1} \co \cA^\mu_0 \to \cA^\mu_{-1}$:
\begin{align*}
	\Psi^\mu_{H,1} (a_0 ) &=  d_{-1} \\
	\Psi^\mu_{H,1} (b_0 ) &=  0  \\
	\Psi^\mu_{H,1} (c_0) &= 0 \\
	\Psi^\mu_{H,1} (d_0 ) &= 0    \\
	\Psi^\mu_{H,1} (\scV_1e_0) &= \scV_2 d_{-1}    .
\end{align*}

Let us look at the various maps in Theorem \ref{thm:surgery-hypercube}. Recall from that \eqref{eqn:Omegamu} that $\Omega^\mu=\Phi_{K}^\mu\circ \Psi^\mu_{H,1}$ and from Lemma \ref{lem:HOmega} that 
\begin{enumerate}
	\item $H_\Omega = \Phi^\mu_{K} \circ H_0 $, where $H_0 \co \cA^\mu_s(K) \to \cB^\mu_{s-1}(K)$ and $\Phi^\mu_{K} \co \cB^\mu_{s-1}(K) \to  \cB^\mu_{s}(K)$, 
	\item  $\tilde{H}_\Omega = \Phi^\mu_{K} \circ \tilde{H}_0$, where $\tilde{H}_0 \co  \cA_s^\mu(K)\to \tilde{\cB}_{s-1}^\mu(K)$ and $\Phi^\mu_{K} \co \cB^\mu_{s-1}(K) \to  \cB^\mu_{s}(K)$.
\end{enumerate}
Since for truncation reasons there is only one $\cB^\mu_s$, namely $\cB^\mu_0$, in our example, it follows that both $H_\Omega$ and $\tilde{H}_\Omega$ are zero.

Lastly, we consider $H^{\{*\}}$, which is a null-homotopy of
\[ \frF_1^\mu \iota_K^\mu \frF_1^\mu+\iota_K^\mu+\frF_1^\mu \Omega^\mu \iota_K^\mu \frF_1^\mu+\Omega^\mu \iota_K^\mu. \]
Again, for truncation reasons, the last two terms above are zero. Furthermore, it is straightforward to verify that $\frF_1^\mu \iota_K^\mu \frF_1^\mu = \iota_K^\mu$, and hence we can take $H^{\{*\}}$ to be the zero map. This concludes the computation of the maps in Theorem \ref{thm:surgery-hypercube}. 

We now use these calculations to determine the $\iota_K$-complex of the dual knot. Consider the subcomplex $\cC$ generated by
\[ a_0, \quad  d_0+\scV^{-1}_1 a'_0, \quad c_0+c_{-1}, \quad d_{-1}+\scV^{-1}_1a'_0, \quad e_{-1}. \] 
We leave it as an exercise for the reader to verify that if we quotient by this subcomplex, the result is acyclic.  The differential $\d$ on $\cC$ is 
\begin{align*} 
	\d a_0 &= \scU_2(d_0+\scV^{-1}_1 a'_0) \\
	\d (d_0+\scV^{-1}_1 a'_0) &= 0 \\ 
	\d (c_0+c_{-1}) &= 0 \\
	\d (d_{-1}+\scV^{-1}_1) &= 0 \\
	\d e_{-1} &= \scV_2(d_{-1}+\scV^{-1}_1 a'_0),
\end{align*}
and the knot involution $\iota_K$ on $\cC$ is
\begin{align*} 
	\iota_K (a_0) &= e_{-1} \\
	\iota_K (d_0+\scV^{-1}_1 a'_0) &= d_{-1} + \scV_1^{-1} a'_0 \\ 
	\iota_K (c_0+c_{-1}) &= c_{-1} + d_{-1} + c_0 + d_0 \\
	\iota_K (d_{-1}+\scV^{-1}_1 a'_0) &= d_0 + \scV_1^{-1} a'_0  \\
	\iota_K( e_{-1}) &= a_0.
\end{align*}
Lastly, we observe that $\cC$ is $\iota_K$-locally equivalent to the trivial $\iota_K$-complex $(\langle x \rangle, \id)$, with local equivalences 
\[ f \co \cC \to (x, \id) \qquad \textup{ and } \qquad g \co (x, \id) \to \cC \]
given by 
\[ f(c_0+c_{-1}) = x \qquad f(a_0) = f(d_0+\scV^{-1}_1 a'_0) = f(d_{-1}+\scV^{-1}_1)= f( e_{-1})=0 \]
and
\[ g(x) = c_0 + c_{-1} + d_0+\scV_1^{-1}a'_0. \]

\section{Local models}
\label{sec:local-models}

In this section, we compute the local model of the dual knot surgeries in the case that $n=\pm 1$. These are the only cases that $\mu$ is null-homologous in $S^3_n(K)$, so that $(\cCFK(S^3_n(K), \mu),\iota_K)$ is an $\iota_K$-complex in the sense of \cite{HMInvolutive} (cf. \cite{ZemConnectedSums}). 

 We prove the following more specific version of Theorem~\ref{thm:local-classes-intro} from the introduction:

\begin{thm}\label{thm:local-class} Suppose that $K$ is a knot in $S^3$.  The $\iota_K$-local equivalence class of $\bXI^\mu_1(K)$ is equal to the truncation:
\[
\begin{tikzcd}[labels=description, column sep=0cm] A_{-1}^{\mu}(K) \ar[dr, "\frF_1^\mu \tilde v^\mu"]&& A_0^\mu(K) \ar[dl, "v^\mu"]\\
& B_0^{\mu}(K)
\end{tikzcd}
\]
with the restriction of the involution from $\bXI^\mu_1(K)$. Similarly, the $\iota_K$-local equivalence class of $\XI_{-1}^\mu(K)$ is equal to the truncation
\[
\begin{tikzcd}[labels=description,column sep=0cm] &A_{-1}^{\mu}(K) \ar[dr, "v^\mu"]\ar[dl, "\frF_{-1}^\mu\tilde v^\mu"]&& A_0^\mu(K) \ar[dl, "\frF_{-1}^\mu\tilde v^\mu"] \ar[dr,"v^\mu"]\\
B_{-2}^\mu(K)&& B_{-1}^{\mu}(K)&&B_{0}^\mu(K)
\end{tikzcd}.
\]
\end{thm}

We will focus mostly on the case that $n=1$, since the case of $n=-1$ is similar, and since one may obtain a model with $3\rank_{\bF[\scU,\scV]} \cCFK(K)$ generators by dualizing the model for $(S_{+1}^3(m(K)),\mu)$. 

\subsection{Truncations}

\label{sec:truncations}

We now recall notation for horizontal truncations of the dual knot formula. We focus on the case that the framing is either $+1$ or $-1$.

 When $n=1$ and $b>0$, we write $\bX_{1}^\mu(K)\langle b\rangle$ for the quotient complex of $\bX_1^{\mu}(K)$ generated by $A_s^\mu(K)$ for $-b\le s<b$ and $B_s^\mu(K)$ for $-b<s<b$.  Similarly, when $n=-1$ and $b>0$, we write $\bX_{-1}^\mu(K)$ for subcomplex of $\bX_{-1}^\mu(K)$ generated by $A_s^\mu(K)$ for $-b\le s<b$ and $B_s^{\mu}(K)$ for $-b-1\le s<b$. Note that the map $\iota_\mu$ described in Theorem~\ref{thm:surgery-hypercube} descends to both of these complexes.
 
 We define truncations of the larger complexes $\bX_{\pm 1}^H(K)$ by replacing $\mu$ with $H$ throughout, and using the same range of indices.

\subsection{Preliminary constructions}

In this section, we construct two families of maps:
\[
\scF_s\colon \cA_{s}(L)\to \cA_{s+1}(L) \quad \text{for} \quad  s\ge 0
\]
and
\[
\scG_s\colon \cA_{s}(L)\to \cA_{s-1}(L)\quad \text{for} \quad s<0.
\]
These maps will be used to construct the local equivalence appearing in Theorem~\ref{thm:local-class}.

Let $\cH^+$ denote the link Floer complex of the positive Hopf link:
\[
\cH^+=\begin{tikzcd}[labels=description,row sep=1cm, column sep=1cm] \ve{a} \ar[d, "\scV_1"]\ar[r, "\scU_2"]& \ve{b}\\
\ve{c}& \ve{d} \ar[u, "\scU_1"] \ar[l, "\scV_2"]
\end{tikzcd}
\]

\noindent Define the map
\[
F\colon \cH^+\to \cH^+,
\]
via the formula
\[
\begin{tikzcd}[row sep=0cm]
\ve{a}\ar[r,mapsto]&\scV_1\ve{d}+\scV_2\ve{a} \\
\ve{b}\ar[r,mapsto]& \scV_2\ve{b}\\
\ve{c}\ar[r,mapsto]& \scV_1\ve{b}\\
\ve{d}\ar[r,mapsto]& 0,
\end{tikzcd}
\]
extended $\bF[\scU_1,\scV_1]$-skew equivariantly and $\bF[\scU_2,\scV_2]$-equivariantly. Since $F$ is $\bF[\scU_1,\scV_1]$-skew equivariant, it induces a map from $\scU_1^{-1} \cH^+$ to $\cH^+$.

Since $F$ is $\bF[\scU_1,\scV_1]$-skew equivariant, we may define the tensor product $\iota_K\otimes F$. This tensor product $\iota_K\otimes F$ interacts in a simple manner with gradings on the mapping cone complex. We summarize the important results in Lemmas~\ref{lem:Alexander-grading-scF} and~\ref{lem:Maslov-gradings-scF}.

\begin{lem}\label{lem:Alexander-grading-scF}
 Suppose $K\subset S^3$ and let $L=K\# H$. The map $\iota_K\otimes F$ sends $\cA_{s}(L)$ to $\cA_{-s}(L)$. Furthermore, as a map from $\cA_{s}(L)$ to $\cA_{-s}(L)$, the map $\iota_K\otimes F$ has homogeneous $(\gr_w,\gr_z)$-grading $(-2s,2s-2)$ with respect to the Maslov grading of $\cCFK(K)\otimes \cH^+$ (that is, before considering the grading shift in the mapping cone formula).
\end{lem}
\begin{proof} Consider the first claim, and suppose $x=\xs_s\otimes \ys$ where $\xs_s\in \cA_s(K)$, that is, $\xs_s\in \cCFK(K)$ and $A(\xs_s)=s$, and $\ys\in \{\ve{a},\ve{b},\ve{c},\ve{d}\}$. The map $\iota_K\otimes F$ is indicated below:
\[
\begin{tikzcd}[row sep=0cm]
\cA_{s}(L)\ni \xs_s\otimes \ve{a} \ar[r,mapsto]&\scV_1\iota_K(\xs_s)\otimes \ve{d}+ \iota_K(\xs_s)\otimes \scV_2\ve{a}\in \cA_{-s}(L) \\
 \cA_{s}(L)\ni \xs_s\otimes \ve{b}\ar[r,mapsto]& \iota_K(\xs_s)\otimes \scV_2 \ve{b}\in \cA_{-s}(L)\\
\cA_{s-1}(L)\ni \xs_s\otimes \ve{c}\ar[r,mapsto]& \scV_1\iota_K(\xs_s)\otimes \ve{b} \in \cA_{-s+1}(L)\\
 \cA_{s-1}(L)\ni \xs_s\otimes \ve{d}\ar[r,mapsto]& 0.
\end{tikzcd}
\]
By inspection, $\cA_{s}(L)$ is sent to $\cA_{-s}(L)$ in all cases.

We now consider the claim about the $(\gr_w,\gr_z)$-grading changes. We verify the claim for $\ve{x}_s\otimes \ve{b}\in \cA_{s}(L)$. Suppose that $\gr(\ve{x}_s)=(i,j)$ where $2s=i-j$. Then 
\[
\gr(\xs_s\otimes \ve{b})=(i+\tfrac{1}{2},j-\tfrac{3}{2})\quad \text{and} \quad \gr(\iota_K(\xs_s)\otimes \scV_2 \ve{b})=(j+\tfrac{1}{2}, i-\tfrac{3}{2}-2)=(i-2s+\tfrac{1}{2},j+2s-\tfrac{7}{2}). 
\]
This verifies the claim for $(\iota_K\otimes F)(\xs_s\otimes \ve{b})$. The case when $\ve{b}$ is replaced by $\ve{a}$, $\ve{c}$ or $\ve{d}$ is easily verified to give the same result.
\end{proof}

 We now define the map
\[
\scF_s\colon \cA_{s}(L)\to \cA_{s+1}(L)
\]
 via the formula
\[
\scF_s:=\scV_1^{2s+1} \cdot (\iota_K\otimes F),
\]
if $s\ge 0$.

\begin{lem}\label{lem:Maslov-gradings-scF} Suppose $s\ge 0$. The map $\scF_s$ is $(\gr_{w},\gr_{z})$-grading preserving with respect to the gradings on the mapping cone complex.
\end{lem}
\begin{proof} It follows from Lemma~\ref{lem:Alexander-grading-scF} that $\scF_s$ shifts the $(\gr_w,\gr_z)$ bigrading from $\cCFK(K)\otimes \cH^+$ by $(-2s,-2s-4)$. On the other hand, the grading shifts in the mapping cone formula are computed in equation~\eqref{eq:gradings}. The difference in the shifts from $\cA_{s+1}(L)$ and $\cA_{s}(L)$ is equal to
\[
\left(\frac{(1-2(s+1))^2-(1-2s)^2}{4}, \frac{(1+2(s+2))^2-(1+2(s+1))^2}{4}\right)=(2s,2s+4).
\]
Hence, the total shift of $\scV^{2s+1}\cdot \iota_K\otimes F$ on $\cA_{s}(L)$ is
\[
(0,-4s-2)+(-2s,2s-2)+(2s,2s+4)=(0,0),
\]
completing the proof.
\end{proof}

Symmetrically, there is a $\bF[\scU_1,\scV_1]$-skew equivariant, $\bF[\scU_2,\scV_2]$-equivariant map
\[
G\colon \cH^+\to \cH^+,
\]
given by the formula
\[
\begin{tikzcd}[row sep=0cm]
\ve{a}\ar[r,mapsto]&0 \\
\ve{b}\ar[r,mapsto]& \scU_1\ve{c}\\
\ve{c}\ar[r,mapsto]& \scU_2\ve{c}\\
\ve{d}\ar[r,mapsto]&\scU_1\ve{a}+\scU_2\ve{d}.
\end{tikzcd}
\]

\noindent The map $G$ sends $\cA_{s}(L)$ to $\cA_{-s-2}(L)$. If $s<0$, we may define a map $\scG_s\colon \cA_{s}(L)\to \cA_{s-1}(L)$
via the formula
\[
\scG_s=\scU_1^{-2s-1}\cdot (\iota_K\otimes G).
\]

\begin{lem} The map $\scG_s$ sends $\cA_{s}(L)$ to $\cA_{s-1}(L)$ and preserves the $\gr=(\gr_w,\gr_z)$-bigrading on the mapping cone.
\end{lem}
\begin{proof} This is a computation, similar to the argument for $\scF_s$. Consider, for example, $\scG_s(\xs_s\otimes \ve{b})=\scU_1^{-2s}\iota_K(\xs_s)\otimes \ve{c}$. Recall $\ve{x}_s\otimes \ve{b}\in \cA_{s}(L)$ and  $\scG_s(\xs_s\otimes \ve{b})\in \cA_{s-1}(L)$. If $\gr(\xs_s)=(i,j)$, where $i-j=2s$, we compute that the $\gr_w$ grading of $\xs_s\otimes \ve{b}$ to be
\[
i+\frac{1}{2}+\frac{(1-2s)^2-1}{4}
\]
On the other hand, the cone-$\gr_w$-grading of $\scG_s(\xs_s\otimes \ve{b})=\scU_1^{-2s}\iota_K(\xs_s)\otimes \ve{c}$ is given by
\[
i-2s+4s-\frac{3}{2}+\frac{(1-2(s-1))^2-1}{4}.
\]
These expressions are equal. The $\gr_z$-grading changes are similar. 
\end{proof}

\subsection{Proof of Theorem~\ref{thm:local-class}}

In this section, we describe our proof of Theorem~\ref{thm:local-class}. We focus first on the case of $+1$-surgery. Our proof is to define local maps between $\bXI_{1}^\mu(K)$ and $\bXI_1^\mu(K)\langle 1 \rangle$, where the latter object is the truncation defined in Section~\ref{sec:truncations} (and appearing in the statement of Theorem~\ref{thm:local-class}).

As a first step, we will focus on the expanded model of the mapping cone formula as in Equation~\eqref{eq:iotacomplex}, which we denote by $\bXI_1^H(K)$.  One direction is easy: we may define a local map from $\bXI_1^H(K)$ to $\bXI_1^H(K)\langle 1 \rangle$ via the canonical projection map. We now construct a local map in the opposite direction. We focus on defining a map from $\bXI_1^H(K)\langle 1\rangle$ to $\bXI_1^H(K)\langle 2\rangle$. The argument extends easily to construct a local map from $\bXI_1^H(K)\langle 1\rangle$ to $\bXI^H(K)\langle g\rangle$, which is homotopy equivalent to the infinitely generated complex.

We define
\[
\widehat \scG_{-1}=(\id \otimes \id+\Phi_{K,1} \otimes \Psi_{H,1})\circ \scG_{-1}\quad \text{and} \quad \widehat{\scF}_0=(\id|\id+\Phi_{K,1}|\Psi_{H,1})\circ \scF_0.
\]
Compare \cite{ZemConnectedSums}*{Lemma~2.14}.

We construct a map via a diagram, as follows:
\[
\begin{tikzcd}[labels=description, column sep={1.6cm,between origins}] &&
\cA_{-1}(L)
	\ar[dr, "\frF_1 \tilde v_L",dashed]
	\ar[dd,"\id"]
	\ar[ddll, bend right=7, "\widehat\scG_{-1}"]
	\ar[dddl,"\a v_L"]
&& 
\cA_{0}(L)
	\ar[dl, "v_L",dashed]
	\ar[dd,"\id"]
	\ar[ddrr, "\widehat\scF_0", bend left=7]
	\ar[dddr, "\b \tilde v_L"]
&&
\,
\\
&&& 
B_{0}(L)
	\ar[dd,"\id"]
\\[1cm]
\cA_{-2}(L)
	\ar[dr, "\frF_1 \tilde v_L",dashed] 
&&
\cA_{-1}(L)
	\ar[dl, "v_L",dashed]
	\ar[dr, "\frF_1 \tilde v_L",dashed]
&& 
\cA_{0}(L)
	\ar[dr, "\frF_1 \tilde v_L",dashed]
	\ar[dl, "v_L",dashed]
&&
\cA_{1}(L)
	\ar[dl, "v_L",dashed]
\\
&
\cB_{-1}(L)&& 
\cB_{0}(L)&&
\cB_{1}(L)
\end{tikzcd}
\]
The maps 
\[
\a\colon \cB_{-1}(L)\to \cB_{-1}(L)\quad \text{and} \quad \b\colon \cB_{0}(L)\to \cB_{1}(L)
\]
 will have $\gr=(\gr_w,\gr_z)$-grading of $(1,1)$, and will be $\bF[\scU_2,\scV_2]$-equivariant. The existence of $\a$ is as follows. Since $\widehat\scG_{-1}$ is $\bF[\scU_1,\scV_1]$-skew equivariant, we know that
\[
\frF_1 \tilde v_L\widehat\scG_{-1}+v_L=(\frF_1\widehat\scG_{-1} +\id)v_L.
\]
The map $\frF_1\widehat\scG_{-1}+\id$ is a $(0,0)$-graded map from $\cB_{-1}(L)$ to $\cB_{-1}(L)$, and is trivial on homology since $\id$, $\widehat\scG_{-1}$ and $\frF_1$ are grading preserving local maps. Hence, by Lemma~\ref{lem:F-X-prep-1}, we choose $\a$ to be a nullhomotopy of $\frF_1\widehat\scG_{-1} + \id$. The map $\b$ is constructed similarly.

This gives us a $\cCFK$-local map 
\[
\Phi\colon \bX_1^H(K)\langle 1\rangle\to \bX_1^H(K)\langle 2\rangle.
\]
 We claim that this map is in fact $\iota_K$-local:

\begin{lem}
 The map $\Phi$ commutes with the knot involutions $\iota_{\bX}^H$ induced by $\bX^H_1(K)$, up to $\bF[\scU_2,\scV_2]$-skew equivariant chain homotopy. 
\end{lem}
\begin{proof}
 We indicate $[\Phi,\iota_{\bX}^H]$ below:
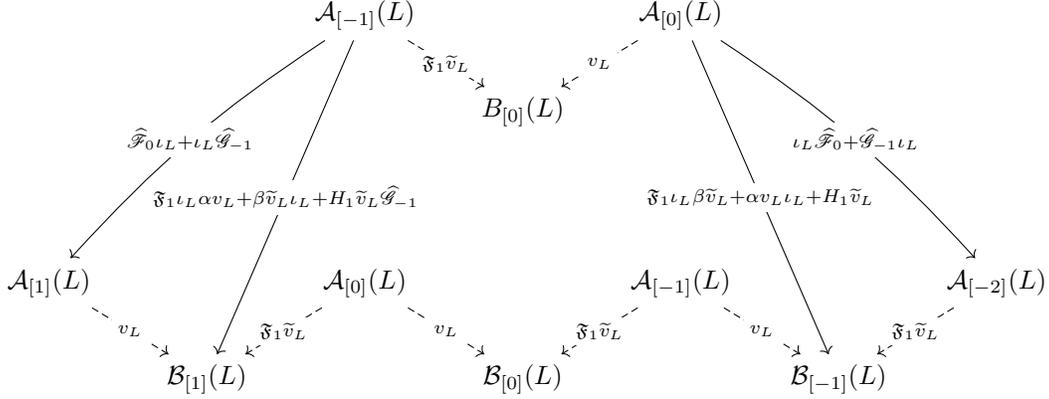
\[
\begin{tikzcd}[labels=description, column sep={2.1cm,between origins}] &&
\cA_{-1}(L)
	\ar[dr, "\frF_1 \tilde v_L",dashed]
	\ar[ddll, bend right=7, "\widehat\scF_0\iota_L+\iota_L\widehat\scG_{-1}"]
	\ar[dddl,"\frF_1\iota_L \a v_L+\b \tilde v_L \iota_L+H_1\tilde v_L \widehat\scG_{-1}"]
&& 
\cA_{0}(L)
	\ar[dl, "v_L",dashed]
	\ar[ddrr, "\iota_L\widehat\scF_0+\widehat\scG_{-1}\iota_L", bend left=7]
	\ar[dddr, "\frF_1\iota_L\b \tilde v_L+\a v_L \iota_L+H_1\tilde v_L"]
&&
\,
\\
&&& 
B_{0}(L)
\\[1cm]
\cA_{1}(L)
	\ar[dr, " v_L",dashed] 
&&
\cA_{0}(L)
	\ar[dl, "\frF_1 \tilde v_L",dashed]
	\ar[dr, "v_L",dashed]
&& 
\cA_{-1}(L)
	\ar[dr, "v_L",dashed]
	\ar[dl, "\frF_1 \tilde v_L",dashed]
&&
\cA_{-2}(L)
	\ar[dl, "\frF_1 \tilde v_L",dashed]
\\
&
\cB_{1}(L)&& 
\cB_{0}(L)&&
\cB_{-1}(L)
\end{tikzcd}
\]
To construct a null-homotopy of the above map, it is sufficient to construct skew-equivariant null-homotopies of $\widehat\scF_0 \iota_L+\iota_L \widehat\scG_{-1}$ and $\iota_L \widehat\scF_0+\widehat\scG_{-1} \iota_L$, since our technique of factoring through $v_L$ and $\tilde{v}_L$ may be used to construct the remainder of the null-homotopy.

Note first that, $\scF_0=\scV_1(\iota_K|F)$ and $\scG_{-1}=\scU_1(\iota_K|G)$, where to shorten notation in the computations that follow we let $|$ replace $\otimes$. Also, $\iota_L=(\id|\id+\Phi_{K,1}|\Psi_{H,1})$. We compute directly that 
\[
\iota_H G=F\iota_H.
\]
Hence, we simply compute
\[
\begin{split} \widehat\scF_0\iota_L
=&(\id|\id+\Phi_{K,1}|\Psi_{H,1})\scV_1(\iota_K|F) (\id|\id+\Phi_{K,1}|\Psi_{H,1})(\iota_K|\iota_H)\\
\eqsim&(\id|\id+\Phi_{K,1}|\Psi_{H,1})(\iota_K|\iota_H)(\id|\id+\Phi_{K,1}|\Psi_{H,1})\scU_1(\iota_K|G)\\
=&\iota_L \widehat\scG_{-1}.
\end{split}
\]
Similarly
\[
\begin{split} \iota_L \widehat\scF_0
=&(\id|\id+\Phi_{K,1}|\Psi_{H,1})(\iota_K|\iota_H)(\id|\id+\Phi_{K,1}|\Psi_{H,1})\scV_1(\iota_K|F) \\
\eqsim &(\id|\id+\Phi_{K,1}|\Psi_{H,1})\scU_1(\iota_K|G)(\id|\id+\Phi_{K,1}|\Psi_{H,1})(\iota_K|\iota_H)\\
=&\widehat \scG_{-1} \iota_L
\end{split}
\]
The proof is complete.
\end{proof}

As mentioned above, the proof when $n=-1$ is similar to the case that $n=1$. Furthermore, since a model with only $3\cdot \rank_{\bF[\scU,\scV]} \cCFK(K)$ generators may be obtained by dualizing the model for $(S_{+1}^3(m(K)),\mu)$ and then reversing string orientation, we only sketch the details.

 To describe a local equivalence as in the statement, write $\XI^\mu_{-1}(K)\langle b\rangle $ for the truncation which contains $\cA_{s}^\mu(K)$ for $s\in \{-1-b,\dots, b\}$ and $\cB_s^\mu(K)$ for $s\in \{-2-b,\dots, b\}$. One first observes that there is a natural inclusion map from the truncation $\XI^\mu_{-1}(K)\langle b\rangle$ to $\XI^\mu_{-1}(K)\langle b+1\rangle$, which is a clearly a local map. It suffices to construct a local map from $\XI^\mu_{-1}(K)\langle b\rangle$ to $\XI^\mu_{-1}(K)\langle b-1\rangle$ when $b\ge 1$. One may construct such a map by modifying the diagram in \cite{HHSZ}*{Figure~3.9}. We leave the details to the interested reader.

Theorem~\ref{thm:local-classes-intro} of the introduction now follows directly from Theorem~\ref{thm:local-class}.

Finally, we prove Corollary \ref{cor:phiij}.
\begin{proof}[Proof of Corollary \ref{cor:phiij}]
We see from the non-involutive version of Theorem \ref{thm:local-class} that the generators of the local equivalence class of $\cCFK(S^3_n(K), \mu)$ are in Alexander grading $-1, 0,$ and $1$. It then follows from the definition of standard complexes \cite[Section 5]{DHSThomcobknot} that the standard complex of $\cCFK(S^3_n(K), \mu)$ is of the form $C(b_1, \dots , b_n)$ with 
$|b_{2k-1}| = U^{i_k}_B W^{j_k}_{B_0}$ where $|i_k -j_k| \leq 2$. It follows from the definition of the $\varphi_{i,j}$ \cite[Definition 8.1]{DHSThomcobknot} that $\varphi_{i,j}(S^3_n(K), \mu) = 0$ if $|i-j| > 2$, as desired.
\end{proof}

\bibliographystyle{custom}
\def\MR#1{}
\bibliography{biblio}

\end{document}